\pdfoutput=1
\documentclass[pdflatex, twoside,12pt]{article}
\usepackage{amsmath}
\usepackage{amsmath}
\usepackage{amssymb}
\usepackage{graphicx}
\usepackage{amsthm}
\usepackage{pifont}
\usepackage{eucal}
\usepackage{mathrsfs}
\usepackage{color}
\usepackage{enumitem}
\usepackage{hyperref}
\usepackage{siunitx}
\usepackage{cancel}
\usepackage[normalem]{ulem}
\usepackage[all]{xy}
\usepackage{tikz-cd}
\usepackage{multirow}%
\usepackage{amsfonts}%
\usepackage[title]{appendix}%
\usepackage{xcolor}%
\usepackage{textcomp}%
\usepackage{manyfoot}%
\usepackage{booktabs}%
\usepackage{algorithm}%
\usepackage{algorithmicx}%
\usepackage{algpseudocode}%
\usepackage{listings}%
\usepackage[left=1in,top=1in,right=1in,bottom=1in]{geometry}
\usepackage{mathalfa}
\usepackage[cbgreek]{textgreek}
\DeclareFontFamily{OT1}{pzc}{}
\DeclareFontShape{OT1}{pzc}{m}{it}{<-> s * [1.20] pzcmi7t}{}
\DeclareMathAlphabet{\mathpzc}{OT1}{pzc}{m}{it}

\newtheorem{theorem}{Theorem}[section]
\newtheorem*{theorem*}{Theorem}
\newtheorem{lemma}[theorem]{Lemma}
\newtheorem{proposition}[theorem]{Proposition}
\newtheorem{corollary}[theorem]{Corollary}
\newtheorem{definition}[theorem]{Definition}

\newcommand{\brk}[1]{\left(#1\right)}          % \brk{.}     => (.)
\newcommand{\Brk}[1]{\left[#1\right]}          % \Brk{.}     => [.]
\newcommand{\BRK}[1]{\left\{#1\right\}}        % \BRK{.}     => {.}
\newcommand{\mymat}[1]{\begin{pmatrix} #1 \end{pmatrix}}
\newcommand{\Cases}[1]{\begin{cases} #1 \end{cases}}
\newcommand{\bra}{\langle}
\newcommand{\ket}{\rangle}

\newcommand{\secref}[1]{Section~\ref{#1}}

\newcommand{\thmref}[1]{Theorem~\ref{#1}}
\newcommand{\defref}[1]{Definition~\ref{#1}}

\newcommand{\propref}[1]{Proposition~\ref{#1}}
\newcommand{\lemref}[1]{Lemma~\ref{#1}}
\newcommand{\corrref}[1]{Corollary~\ref{#1}}

\newcommand{\beq}{\begin{equation}}
\newcommand{\eeq}{\end{equation}}

\newcommand{\frakA}{\mathfrak{A}}

\newcommand{\frakG}{\mathfrak{G}}
\newcommand{\frakS}{\mathfrak{S}}
\newcommand{\frakT}{\mathfrak{T}}
\newcommand{\frakX}{\mathfrak{X}}
\newcommand{\frake}{\mathfrak{e}}
\newcommand{\frakn}{\mathfrak{n}}
\newcommand{\frakt}{\mathfrak{t}}

\newcommand{\bbC}{{\mathbb C}}
\newcommand{\bbE}{{\mathbb E}}
\newcommand{\bbF}{{\mathbb F}}
\newcommand{\bbG}{{\mathbb G}}
\newcommand{\bbJ}{{\mathbb J}}
\newcommand{\bbM}{{\mathbb M}}
\newcommand{\bbN}{{\mathbb N}}
\newcommand{\bbP}{{\mathbb P}}
\newcommand{\bbR}{{\mathbb R}}
\newcommand{\bbS}{{\mathbb S}}
\newcommand{\bbU}{{\mathbb U}}
\newcommand{\bbZ}{{\mathbb Z}}

\newcommand{\scrB}{\mathscr{B}}
\newcommand{\scrC}{\mathscr{C}}

\newcommand{\scrH}{\mathscr{H}}
\newcommand{\scrN}{\mathscr{N}}
\newcommand{\scrR}{\mathscr{R}}
\newcommand{\scrS}{\mathscr{S}}

\newcommand{\calA}{{\mathcal{A}}}
\newcommand{\calD}{{\mathcal{D}}}
\newcommand{\calE}{{\mathcal{E}}}
\newcommand{\calG}{{\mathcal{G}}}
\newcommand{\calL}{{\mathcal{L}}}
\newcommand{\calP}{{\mathcal{P}}}
\newcommand{\calQ}{{\mathcal{Q}}}
\newcommand{\calR}{{\mathcal{R}}}
\newcommand{\calS}{{\mathcal{S}}}

\newcommand{\boldd}{\boldsymbol{d}}

\providecommand{\R}{\bbR}
\newcommand{\VF}{\frakX}

\newcommand{\bP}{\calP}
\newcommand{\bS}{\calS}
\newcommand{\D}{\calD}
\newcommand{\euc}{\frake}
\newcommand{\module}{\scrH}
\newcommand{\bQ}{\calQ}
\newcommand{\bA}{\calA}

\newcommand{\trace}{\operatorname{tr}}
\newcommand{\image}{\operatorname{Im}}

\newcommand{\Rm}{\operatorname{Rm}}

\newcommand{\ord}{\operatorname{ord}}
\newcommand{\supp}{\operatorname{supp}}
\newcommand{\End}{{\operatorname{End}}}
\newcommand{\id}{{\operatorname{Id}}}

\newcommand{\textand}{\quad\text{ and }\quad}
\newcommand{\Textand}{\qquad\text{ and }\qquad}

\newcommand{\overbar}[1]{\mkern 1.5mu\overline{\mkern-1.5mu#1\mkern-1.5mu}\mkern 1.5mu}
\newcommand{\tM}{\tilde{M}}
\newcommand{\dM}{{\partial M}}
\newcommand{\Nzero}{{\bbN_0}}

\newcommand{\limn}{\lim_{n\to\infty}}

\newcommand{\tE}{\tilde{\E}}
\newcommand{\tF}{\tilde{\bbF}}
\newcommand{\tg}{\tilde{g}}

\newcommand{\E}{\bbE}

\newcommand{\W}{\Omega}

\newcommand{\g}{g}

\newcommand{\PtD}{\bbP^{\frakt}}
\newcommand{\PnD}{\bbP^{\frakn}}
\newcommand{\PttD}{\bbP^{\frakt\frakt}}
\newcommand{\PtnD}{\bbP^{\frakt\frakn}}
\newcommand{\PntD}{\bbP^{\frakn\frakt}}
\newcommand{\PnnD}{\bbP^{\frakn\frakn}}
\newcommand{\nabg}{\nabla}
\newcommand{\nabU}{\nabla}

\newcommand{\OP}{\mathrm{OP}}
\newcommand{\OPAm}{\OP(\frakA^m)}
\newcommand{\OPAi}{\OP(\frakA^i)}
\newcommand{\OPA}{\OP(\frakA)}
\newcommand{\OPT}{\OP(\frakT)}
\newcommand{\OPTmr}{\OP(\frakT^{m,r})}

\newcommand{\OPSm}{\OP(\frakS^{m,r})}
\newcommand{\OPS}{\OP(\frakS)}
\newcommand{\OPSi}{\OP(\frakS^{-\infty,r})}
\newcommand{\OPSii}{\OP(\frakS^{-\infty})}

\newcommand{\Bmodule}[2]{\scrB^{#1}_{#2}(M,g)}
\newcommand{\PseudoComplex}{(A_\bullet)}
\newcommand{\Complex}{(\bA_\bullet)}
\newcommand{\dU}{d^{\nabU}}

\newcommand{\bHg}{\mathbf{H}}

\newcommand{\dg}{d^{\nabg}}

\newcommand{\deltag}{\delta^{\nabg}}

\newcommand{\dgV}{d^{\nabg}_V}

\newcommand{\deltagV}{\delta^{\nabg}_V}

\newcommand{\Volume}{\text{Vol}}

\newcommand{\pzcd}{\boldd}
\newcommand{\pzcdel}{\text{\textit{\textbf{\textdelta}}}}

\newcommand{\starG}{\star_g}
\newcommand{\starGV}{\starG^V}
\newcommand{\G}{\frakG}
\newcommand{\GV}{\frakG_V}
\newcommand{\dr}{\partial_r}
\newcommand{\idr}{i_{\partial_r}}
\newcommand{\PttG}{\bbP^{\frakt\frakt}_\calG}
\newcommand{\PtnG}{\bbP^{\frakt\frakn}_\calG}
\newcommand{\PntG}{\bbP^{\frakn\frakt}_\calG}
\newcommand{\PnnG}{\bbP^{\frakn\frakn}_\calG}

\newcommand{\LkmAll}{\Lambda_M}
\newcommand{\GkmAll}{\calG_M}
\newcommand{\WkmAll}{\W_M}
\newcommand{\CkmAll}{\scrC_M}

\newcommand{\plLkmAll}{\Lambda_{\dM}}

\newcommand{\plWkmAll}{\W_{\dM}}
\newcommand{\plCkmAll}{\scrC_{\dM}}

\newcommand{\Lkm}[2]{\LkmAll^{#1,#2}}
\newcommand{\Gkm}[2]{\GkmAll^{#1,#2}}
\newcommand{\Wkm}[2]{\WkmAll^{#1,#2}}
\newcommand{\Ckm}[2]{\CkmAll^{#1,#2}}
\newcommand{\ixi}{i_{\xi^{\sharp}}}
\newcommand{\ixiV}{i^V_{\xi^{\sharp}}}

\newcommand{\plLkm}[2]{\plLkmAll^{#1,#2}}

\newcommand{\plWkm}[2]{\plWkmAll^{#1,#2}}
\newcommand{\plCkm}[2]{\plCkmAll^{#1,#2}}

\newcommand{\dBianchi}{d^\calG}
\newcommand{\delBianchi}{\delta^\calG}
\newcommand{\dBianchiV}{d^\calG_V}
\newcommand{\delBianchiV}{\delta^\calG_V}

\newcommand{\DBianchi}{\pzcd^\calG}
\newcommand{\DelBianchi}{\pzcdel^\calG}

\newcommand{\calPG}{\calP_\calG}

\numberwithin{equation}{section}

\begin{document}

\title{
Elliptic Pre-Complexes, Hodge-like Decompositions and Overdetermined Boundary-Value Problems   
}

\author{
Raz Kupferman and
Roee Leder  \footnote{
raz@math.huji.ac.il,
roee.leder@mail.huji.ac.il
}
\\
\\
Institute of Mathematics \\
The Hebrew University \\
Jerusalem 9190401 Israel
}
\maketitle

\begin{abstract}
We solve a problem posed by Calabi more than 60 years ago, known as the \emph{Saint-Venant compatibility problem}: Given a compact Riemannian manifold, generally with boundary, find a compatibility operator for Lie derivatives of the metric tensor. This problem is related to other compatibility problems in mathematical physics, and to their inherent gauge freedom. 
To this end, we develop a framework generalizing the theory of elliptic complexes for sequences of linear differential operators $\PseudoComplex$  between sections of vector bundles. 
We call such a sequence an \emph{elliptic pre-complex} if the operators satisfy overdetermined ellipticity conditions, and the order of $A_{k+1}A_k$ does not exceed the order of $A_k$. 
We show that every elliptic pre-complex $\PseudoComplex$ can be ``corrected" into a complex $\Complex$ of pseudodifferential operators, where $\bA_k - A_k$ is a zero-order correction within this class. The induced complex $\Complex$ yields Hodge-like decompositions, which in turn lead to explicit integrability conditions for overdetermined boundary-value problems, with uniqueness and gauge freedom clauses. 
We apply the theory on elliptic pre-complexes of exterior covariant derivatives of vector-valued forms and double forms satisfying generalized algebraic Bianchi identities, thus resolving a set of compatibility and gauge problems, among which one is the Saint-Venant problem. 
\end{abstract}

{
\footnotesize
\tableofcontents
}
\pagebreak
%%%%%%%%%%%%%%%%%%%%%%%%%%%%%%%%%%%%%%%%%%%%%
\section{Introduction, main results and applications}
\label{sec:intro}

\subsection{Statement of the problem}
%%%%%%%%%%%%%%%%%%%
This paper was originally motivated by a study initiated by Calabi more than sixty years ago \cite{Cal61}. At the center of his study is the following problem: 

%%%%%%%%%%%%%%%%%%%%%%%%
\begin{quote}
Let $(M,g)$ be a compact Riemannian manifold (generally with boundary). 
What are necessary and sufficient conditions for a symmetric $(2,0)$-tensor on $M$ to be a Lie derivative of the metric? 
\end{quote}
%%%%%%%%%%%%%%%%%%

Restating the question in local coordinates, given a tensor field $(\sigma_{ij})_{i,j=1}^d$ satisfying $\sigma_{ij} = \sigma_{ji}$, what are necessary and sufficient conditions for the existence of a vector field $(Y^i)_{i=1}^d$, such that
\beq
\sigma_{ij} = (\calL_Yg)_{ij} = \nabg_i\omega_j + \nabg_j\omega_i ,
\label{eq:Saint_Venant_Problem} 
\eeq
where $\omega_i = g_{ij}Y^j$ and $\nabg$ is the covariant derivative?

The operator $Y\mapsto \calL_{Y}g$ from the space of vector fields over $M$ to the space of symmetric tensor fields is also commonly known as the \emph{Killing operator}, or sometimes the \emph{deformation operator}. Thus, in other words, the aim of the problem is to characterize the range of this operator. 

Calabi provided an answer for $(M,g)$ closed, simply-connected, and having constant sectional curvature. In particular, he commented at the end of his paper:

%%%%%%%%%%%%%%%%%%%%%
\begin{quotation}
``In a subsequent article... (the) theorem will be supplemented by an analogue of Hodge's theorem... satisfying globally certain elliptic systems of equations."
\end{quotation}
%%%%%%%%%%%%%%

As far as we know, such an article has never been published. 

\subsection{\bfseries History of the Saint-Venant compatibility problem}
The Euclidean version of the above problem arises in the theory of linear elasticity under the name ``Saint-Venant compatibility": Let $\Omega$ be a Euclidean domain modeling an elastic bulk of material under strain; the stress field $\sigma$, which represents internal forces, is a symmetric $(2,0)$-tensor field satisfying the so-called \emph{Saint-Venant compatibility condition} \cite{Gur72}, 
\beq
\nabla\times\nabla\times\sigma=0,
\label{eq:curlcurlIntro}
\eeq
where $\nabla\times\nabla\times$ is the curl-curl operator, mapping symmetric $(2,0)$-tensors into $(4,0)$-tensors satisfying the Bianchi symmetries of algebraic curvatures. In Euclidean coordinates,
\[
(\nabla\times\nabla\times\sigma)_{ijkl}=\partial_{ik}\sigma_{jl}-\partial_{jk}\sigma_{il}-\partial_{il}\sigma_{jk}+\partial_{jl}\sigma_{ik}. 
\]
It is a classical result that when $\W$ is a simply-connected domain, $\sigma=\calL_{Y}\frake$ for some $Y\in\frakX(\Omega)$ if and only if $\nabla\times\nabla\times\sigma=0$, where $\frakX(\Omega)$ denotes the space of vector fields on $\W$, and $\euc$ is the Euclidean metric. Thus, for $\W$ Euclidean and simply-connected, the Saint-Venant compatibility condition \eqref{eq:curlcurlIntro} is a necessary and sufficient condition for a symmetric $(2,0)$-tensor to be the symmetrized gradient of a vector field.

Over the years, several authors generalized the Saint-Venant compatibility condition to the Riemannian setting, by relaxing the assumptions on either the topology, the geometry or the regularity of the fields in question. Specifically:
\begin{enumerate}[itemsep=0pt,label=(\alph*)]
\item Calabi found a compatibility condition for $(M,g)$ closed,  simply-connected, and having constant sectional curvature \cite[Prop.~3]{Cal61}. 

\item Gasqui and Goldschmidt  improved Calabi's result by extending it to closed, simply-connected symmetric spaces \cite{GG88A,GG88B}.

\item Several authors \cite{Kha19,Pom22,CELM21,CELM23} have recently extended this condition to certain classes of locally-symmetric spaces without boundary, and to certain Lorentzian manifolds arising in general relativity (Minkowski, Schwarzschild and Kerr).

\item In a series of works, Ciarlet, Geymonat and co-workers addressed the Saint-Venant problem for 
three-dimensional Euclidean domains having Lipschitz boundary and $\sigma$ having $L^2$-regularity,  obtaining weak versions of condition \eqref{eq:curlcurlIntro} (\cite{CCGK07,GK09} and references therein).

\item Yavari and Angoshtari showed how similar results can be obtained in a locally-flat setting by using the Hodge decomposition for scalar differential forms \cite{Yav13,YA16}.

\item The authors of the present paper obtained compatibility conditions for compact manifolds with boundary having constant sectional curvature and arbitrary topology, and connected this analysis to 
an elliptic theory and other problems in elasticity, such as the representation of stresses by stress potentials \cite{KL21b,KL21a}.
\end{enumerate}

As is apparent from this brief survey, no one has managed to relieve the assumption that the underlying Riemannian manifold has, at the very least, a parallel curvature tensor (the Lorentzian cases involve symmetry assumptions as well). All aforementioned work recognizes that at the heart of the Saint-Venant problem is a search for a second-order linear differential operator acting on symmetric tensor fields, which annihilates the image of $U\mapsto\calL_Ug$. In other words, all attempts to resolve the Saint-Venant problem funneled towards a search for a complex, 
\beq
\begin{tikzcd}
\frakX(M) \arrow[rr, "U\mapsto\calL_Ug"] &  & \Ckm11  \arrow[rr, "\text{?}"] &  & \Ckm22,
\end{tikzcd}
\label{eq:saint_venant_complex_mission}
\eeq
where $\Ckm11$ denotes the space of symmetric $(2,0)$-tensor fields and $\Ckm22$ denotes the space of $(4,0)$-tensor fields satisfying algebraic Bianchi symmetries (these notations will be clarified below).
The difficulty in constructing such a complex in general Riemannian geometries is essentially what has been limiting the progress on the Saint-Venant problem. Such a complex, in cases where it has been constructed, has become known in the literature as the \emph{Calabi complex} \cite{GG88A,Eas00}.

From the point of view of the theory of partial differential equations, the Saint-Venant problem \eqref{eq:Saint_Venant_Problem} is an overdetermined system of partial differential equations (\cite{Gol67,Spe69,BE69}; see also the encyclopedic entry \cite{DS96} and references therein). 
In such systems, a diagram of the form \eqref{eq:saint_venant_complex_mission} is typical: Given an overdetermined differential operator $A_0:\Gamma(\bbE_0)\to \Gamma(\bbE_1)$ between sections of vector bundles, one seeks for a differential operator $A_1:\Gamma(\bbE_1)\to \Gamma(\bbE_2)$ such that  $A_1\sigma=0$ if and only if $\sigma\in\operatorname{Range}A_0$.  The operator $A_1$ is called a \emph{compatibility operator} for $A_0$. The complex 
\[
\begin{tikzcd}
%0\arrow[rr, "0"]  & &  \Gamma(\bbE_0) \arrow[rr, "A_0"] &  & \Gamma(\bbE_1)  \arrow[rr, "A_1"] &  & \Gamma(\bbE_2) \arrow[rr, "A_2"]  &  & \dots
0\arrow[r, "0"]  &  \Gamma(\bbE_0) \arrow[r, "A_0"] &  \Gamma(\bbE_1)  \arrow[r, "A_1"] & \Gamma(\bbE_2) \arrow[r, "A_2"]  & \cdots
\end{tikzcd}
\]
is then called a \emph{compatibility complex} for $A_0$. 

There are known local methods for verifying the existence of a compatibility complex for a given operator $A_0$, applicable only under quite restrictive conditions on its coefficients \cite{Gol67,Spe69,DS96}. 
For example, a differential operator with constant coefficients always admits a compatibility complex; in the context of the Saint-Venant problem, the Killing operator has constant coefficients in Euclidean space. 
In fact, such a procedure was applied in \cite{GG88A,GG88B} to obtain a compatibility condition in symmetric spaces. The Killing operator, however, does not satisfy the required conditions in a general Riemannian manifold, and in particular, in one having a boundary and a non-trivial topology, an area of study having scarce literature. 

Thus, the resolution of the Saint-Venant problem is expected to impact the more general area of overdetermined systems (and their dual counterpart, underdetermined systems), which is a subject gaining a renewed interest not just from an abstract perspective but also motivated by applications in elasticity and general relativity \cite{Kha19,KL21b,Pom22,Van23,Hin23,CELM23}. Specifically, it is expected to be applicable in gauge fixing in various potential theories.

As a classical example, consider the theory of electromagnetism, where the electromagnetic field $F\in \Omega^2(M)$ satisfies the system,
\[
\delta F = J
\textand
dF = 0.
\]
The equation $dF=0$ turns into a compatibility condition for the existence of a potential $A\in \Omega^1(M)$, which satisfies the system,
\[
\delta A = 0 
\textand
dA = F.
\]
The choice of a gauge $\delta A=0$ results in an elliptic system $\Delta A=J$, which can then be solved by standard methods. 
The exploitation of similar gauge freedoms has long been sought in elasticity and in general relativity, and has been recognized to be related to the Saint-Venant problem. 

\subsection{\bfseries A non-local, zero order correction} As Calabi first found out, for $(M,g)$ having constant sectional curvature $\kappa\in\R$, the following linear $\Ckm11\to\Ckm22$ operator,
\beq
\label{eq:CalabisIntro}
\begin{split}
\sigma\mapsto\brk{H\sigma}_{ijkl}- \kappa\brk{g_{ik}\sigma_{jl} - g_{jk}\sigma_{il}  - g_{il}\sigma_{jk}  + g_{jl}\sigma_{ik}}, 
\end{split}
\eeq
annihilates the range of the Killing operator, where
\beq
\begin{split}
\label{eq:HgIntro}
(H\sigma)_{ijkl} &=
\tfrac12(\nabg_{ik}\sigma_{jl} - \nabg_{jk}\sigma_{il}  - \nabg_{il}\sigma_{jk}  + \nabg_{jl}\sigma_{ik}) \\
&+
\tfrac12(\nabg_{jl}\sigma_{ik} - \nabg_{il}\sigma_{jk}  - \nabg_{jk}\sigma_{il}  + \nabg_{ik}\sigma_{jl}).
\end{split}
\eeq
The second-order differential operator $H:\Ckm{1}{1}\rightarrow\Ckm{2}{2}$ is the covariant
generalization of the curl-curl operator, with the Euclidean derivatives replaced by covariant derivatives and imposing symmetry preservation. 

Steps toward extending Calabi's construction to a general Riemannian setting were made in \cite{KL21b}, observing that \eqref{eq:CalabisIntro} arises from the variation formula of the $(2,2)$ curvature operator $g\mapsto\calR_g$ about a metric having constant sectional curvature \cite[pp.~559--560]{Tay11b}. A candidate for generalizing the Calabi operator is obtained from the variation formula about a general metric, resulting in
\[
\sigma\mapsto (H+D)\sigma,
\]
where $D:\Ckm11\to\Ckm22$ is the tensorial map
\[
D\sigma=\frac12\brk{\trace_g(\Rm_g\wedge\sigma)-\trace_g\Rm_g\wedge\sigma-\Rm_g\wedge\trace_g\sigma},
\]
and $\Rm_g$ is the $(4,0)$ Riemann curvature tensor. However, for every $U\in \frakX(M)$, 
\beq
(H+D)\calL_Ug=2\,\calL_U\Rm_g +2\,D\calL_Ug.
\label{eq:order_reducing_intro}
\eeq
Thus, unless $\Rm_g$ satisfies restrictive symmetries, this operator does not annihilate the range of the Killing operator.

Yet, in certain instances, the operator $H$ may be ``corrected" by lower-order terms, to annihilate the range of the Killing operator (or subspaces of that range). Moreover,  
while in general $H\calL_Ug\ne 0$, the right-hand side of \eqref{eq:order_reducing_intro} differentiates the vector field $U$ only once, even though the left-hand side differentiates it (formally) three times. 
Thus, even though the sequence
\[
\begin{tikzcd}
\frakX(M) \arrow[rr, "U\mapsto\calL_Ug"] &  & \Ckm11  \arrow[rr, "H"] &  & \Ckm22
\end{tikzcd}
\]
does not form an exact sequence, it satisfies a weaker property: the order of $H$ composed with the Killing operator does not exceed the order of the latter.
This phenomenon  occurs in several other situations in geometry. For example, let $\bbU\to M$ be a Riemannian vector bundle  endowed with a connection $\nabU$; the sequence of exterior covariant derivatives acting on $\bbU$-valued forms \cite[pp.~362--363]{Pet16},
\[
\begin{tikzcd}
\cdots \arrow[r]   & \Omega^k(M;\bbU) \arrow[r, "\dU"]  & \Omega^{k+1}(M;\bbU) \arrow[r, "\dU"] & \Omega^{k+2}(M;\bbU) \arrow[r]  & \cdots
\end{tikzcd}
\]
satisfies $\dU\dU=R^{\nabU}\wedge$, where $R^{\nabU}\in \Omega^2(M;\End(\bbU))$ is the curvature endomorphism of the connection $\nabU$. Hence the order of $\dU\dU$ does not exceed the order of $\dU$. Note that unless $R^{\nabU}=0$, $\dU$ does not fall into the theory of compatibility complexes discussed above.

Motivated by the study of compatibility conditions for overdetermined systems, yet retaining $H$ as an initial guess for a compatibility operator, we combine the above insights into the following question: is there a lower-order correction operator $G:\Ckm{1}{1}\to \Ckm{2}{2}$, such that
\[
(H + G)\calL_Ug = 0
\qquad
\text{for every $U\in\VF(M)$}.
\]
A comparison with \eqref{eq:order_reducing_intro} implies that such a $G$ must satisfy the operator-valued equation
\beq
G\calL_Ug = - 2\calL_U\Rm_g - D\calL_Ug
\qquad
\text{for every $U\in\VF(M)$}.
\label{eq:OpEqsIntro}
\eeq
Since the right-hand side differentiates $U$ once, the correction operator $G$ is expected to be of order zero. The only differential operators of order zero are tensorial; there are however no tensorial operations satisfying the above identity, unless $\Rm_g$ satisfies restrictive symmetries. Thus, a solution $G$ to \eqref{eq:OpEqsIntro}, if it exists, must be non-tensorial, hence non-local. 

\subsection{\bfseries Elliptic pre-complexes} 

We recall the definition of an elliptic complex on a compact manifold with boundary (our exposition is based on \cite[Ch.~12.A]{Tay11b}). Consider the following diagram:
\[
\begin{xy}
(-20,0)*+{0}="Em1";
(-20,-20)*+{0}="Gm1";
(0,0)*+{\Gamma(\E_0)}="E0";
(30,0)*+{\Gamma(\E_1)}="E1";
(60,0)*+{\Gamma(\E_2)}="E2";
(90,0)*+{\Gamma(\E_3)}="E3";
(100,0)*+{\cdots}="E4";
(0,-20)*+{\Gamma(\bbG_0)}="G0";
(30,-20)*+{\Gamma(\bbG_1)}="G1";
(60,-20)*+{\Gamma(\bbG_2)}="G2";
(90,-20)*+{\Gamma(\bbG_3)}="G3";
(100,-20)*+{\cdots}="G4";
{\ar@{->}@/^{1pc}/^{A_0}"E0";"E1"};
{\ar@{->}@/^{1pc}/^{A_0^*}"E1";"E0"};
{\ar@{->}@/^{1pc}/^{A_1}"E1";"E2"};
{\ar@{->}@/^{1pc}/^{A_1^*}"E2";"E1"};
{\ar@{->}@/^{1pc}/^{A_2}"E2";"E3"};
{\ar@{->}@/^{1pc}/^{A_2^*}"E3";"E2"};
{\ar@{->}@/^{1pc}/^0"Em1";"E0"};
{\ar@{->}@/^{1pc}/^0"E0";"Em1"};
{\ar@{->}@/_{0pc}/^{B_0}"E0";"G0"};
{\ar@{->}@/_{0pc}/^{B_1}"E1";"G1"};
{\ar@{->}@/_{0pc}/^{B_2}"E2";"G2"};
{\ar@{->}@/^{1pc}/^{B_0^*}"E1";"G0"};
{\ar@{->}@/^{1pc}/^{B_1^*}"E2";"G1"};
{\ar@{->}@/^{1pc}/^0"E0";"Gm1"};
{\ar@{->}@/^{1pc}/^{B_2^*}"E3";"G2"};
{\ar@{->}@/_{0pc}/^{B_3}"E3";"G3"};
\end{xy}
\]
where $\E_k\to M$ and $\bbG_k\to\dM$ are sequences of vector bundles, $\Gamma$ stands for the global sections functor \cite[Ch.~10]{Lee12},  $\PseudoComplex = (A_k)_{k\in\Nzero}$ is a sequence of first-order differential operators, $A_k:\Gamma(\E_k)\to \Gamma(\E_{k+1})$, and  $B_k:\Gamma(\E_k)\to \Gamma(\bbG_k)$ and $B_k^*:\Gamma(\E_{k+1})\to \Gamma(\bbG_k)$ are differential boundary operators, such that the following analog of Green's formula hold,
\[
\bra A_k\psi,\eta\ket = \bra\psi,A_k^*\eta\ket  + \bra B_k\psi, B_k^*\eta\ket 
\qquad
\text{for every $\psi\in \Gamma(\E_k)$ and $\eta\in \Gamma(\E_{k+1})$},
\]
where $\bra\cdot,\cdot\ket$ denotes the $L^2$-pairing of both interior and boundary sections.
We call the sequence $\PseudoComplex$ an \emph{elliptic complex} if:

\begin{enumerate}[itemsep=0pt,label=(\alph*)]
\item 
The system $(D_k^*D_k,T_k)$ is elliptic, where 
%$D_k:\Gamma(\E_k)\to\Gamma(\E_{k-1}\oplus\E_{k+1})$ be given by
\[
D_k = A_{k-1}^* \oplus A_k 
\Textand
T_k= B_{k-1}^* \oplus B_k^*A_k.
\]  
%$T_k:\Gamma(\E_k)\to\Gamma(\bbG_{k-1}\oplus\bbG_k)$ is given by
%$T_k= B_{k-1}^* \oplus B_k^*A_k$. 
Alternatively, the system $(D_k^*D_k,T_k^*)$ is elliptic, where $T_k^* = B_{k-1}A_k^*\oplus B_k$.
\item $A_k A_{k-1}=0$.
\end{enumerate}

In the literature, the common definition for an elliptic complex uses the notion of \emph{exactness} of the corresponding sequence of \emph{symbols} of $(A_{\bullet})$. However, the ellipticity condition above together with $A_{k}A_{k-1}=0$ contain this exactness as a byproduct.   

The prime example of an elliptic complex is the \emph{de Rham complex},
\[
\begin{xy}
(-20,0)*+{\cdots}="Em1";
(-20,-20)*+{\cdots}="Gm1";
(0,0)*+{\W^{k-1}(M)}="E0";
(30,0)*+{\W^k(M)}="E1";
(60,0)*+{\W^{k+1}(M)}="E2";
(80,0)*+{\cdots}="E3";
(0,-20)*+{\W^{k-1}(\dM)}="G0";
(30,-20)*+{\W^k(\dM)}="G1";
(60,-20)*+{\W^{k+1}(\dM)}="G2";
{\ar@{->}@/^{1pc}/^d"E0";"E1"};
{\ar@{->}@/^{1pc}/^{\delta}"E1";"E0"};
{\ar@{->}@/^{1pc}/^d"E1";"E2"};
{\ar@{->}@/^{1pc}/^{\delta}"E2";"E1"};
{\ar@{->}@/^{1pc}/^{}"E2";"E3"};
{\ar@{->}@/^{1pc}/^{}"E3";"E2"};
{\ar@{->}@/^{1pc}/^{}"Em1";"E0"};
{\ar@{->}@/^{1pc}/^{}"E0";"Em1"};
{\ar@{->}@/_{0pc}/^{\PtD}"E0";"G0"};
{\ar@{->}@/_{0pc}/^{\PtD}"E1";"G1"};
{\ar@{->}@/_{0pc}/^{\PtD}"E2";"G2"};
{\ar@{->}@/^{1pc}/^{\PnD}"E1";"G0"};
{\ar@{->}@/^{1pc}/^{\PnD}"E2";"G1"};
{\ar@{->}@/^{1pc}/^{}"E0";"Gm1"};
{\ar@{->}@/^{1pc}/^{}"E3";"G2"};
\end{xy}
\]
where  $D_k^*D_k=d\delta+\delta d$ is the Hodge Laplacian, and the boundary conditions $T_k$ and $T_k^*$ are Neumann and Dirichlet boundary operators, respectively \cite{Tay11b}. The central result concerning this sequence is the Hodge decomposition theorem (of either Dirichlet or Neumann type), which among other things identifies the cohomology groups of the complexes $\PseudoComplex=(d)$ and $(A_{\bullet}^*)=(\delta)$. These classical results can be generalized to every elliptic complex \cite[p.~463]{Tay11b}. We note that a sequence of operators $(A_{\bullet})$ satisfying (b) cannot always be supplemented with boundary operators that turn it into an elliptic complex \cite{Tay11b,SS19}. In our study, however, such boundary operators arise naturally.

Motivated by the observation that composition with $H$ does not annihilate the Killing operator, but yet, does not raise the order, we generalize the notion of an elliptic complex, and define a new notion of \emph{elliptic pre-complex}, by loosening the following assumptions in the above diagram:
\begin{enumerate}[itemsep=0pt,label=(\alph*)]
\item The operators $A_k$ may be of arbitrary and varying orders $m_k$
(this has already been considered in the literature of elliptic complexes, in the context of Douglis-Nirenberg systems \cite{RS82,SS19}).
\item The ellipticity of the boundary-value problem $(D_k^*D_k,T_k)$ is replaced by a notion of \emph{overdetermined ellipticity} of the boundary-value problem $(A_k\oplus A_{k-1}^*,B_{k-1}^*)$, along with the surjectivity of the boundary operators $B_k$, $B_k^*$. Overdetermined ellipticity amounts to its symbol satisfying an injectivity condition, instead of a bijectivity condition in elliptic systems. 
\item The condition $A_{k+1}A_k=0$ is replaced by the weaker condition that the order of $A_{k+1}A_k$ does not exceed the order of $A_k$.
\end{enumerate} 
In particular, every elliptic complex of first-order operators is an elliptic pre-complex.

\subsection{\bfseries Main results}
Our main result is that every elliptic pre-complex $\PseudoComplex$ can be corrected into a complex, which we denote by $\Complex$, such that $G_k=\bA_k-A_k$ is a linear operator of \emph{order zero} (a notion which is elaborated below).

The procedure turning an elliptic pre-complex into an elliptic complex involves some substantial delicacies. First, as discussed above, the correction $G_k=\bA_k - A_k$ is generally not tensorial. It is a linear operator belonging to a class of so-called \emph{Green operators} \cite{RS82,Gru96}, arising as solution operators for \emph{pseudodifferential boundary-value problems}. Green operator are nonlocal, which on a manifold with boundary implies that they are characterized both by an order  (loosely speaking, accounting for the maximal number of derivatives) and a class (loosely speaking, accounting for the number of normal derivatives at the boundary). Green operators form their own algebra \cite{Bou71}, with notions of adjoints and inverses, which are complicated by the need to track both orders and classes.  

The fact that the correction $G_k$ is of order zero is an indispensable element of the theory. First and foremost, it guarantees the existence of an adjoint, $\bA_k^*$, which is also a Green operator differing from $A_k^*$ by an operator of order zero. The sequence of operators $\Complex$ inherit the integration by parts of the original sequence,
\[
\bra \bA_k\psi,\eta\ket = \bra\psi,\bA_k^*\eta\ket + \bra B_k\psi, B_k^*\eta\ket,
\]
with unaltered boundary operators. 

Our main theorem is:

%%%%%%%%%%%%%%%% 
\begin{theorem*}[Induced elliptic complex]
Every elliptic pre-complex $\PseudoComplex$ induces a complex of Green operators $\Complex$, uniquely characterized by the following properties: 
\begin{enumerate}[itemsep=0pt,label=(\alph*)]
\item $\bA_{k+1}\bA_k=0$.
\item
$\bA_{k+1} = A_{k+1}$ on $\scrN(\bA_k^*,B_k^*)  = \ker(\bA_k^*\oplus B_k^*)$.
\end{enumerate}
\end{theorem*}
%%%%%%%%%%%%%%%%  

In particular, since $\bA_{-1}^*=0$ and $B_{-1}^*=0$, it holds trivially that $\bA_0=A_0$, i.e., the first operator in the sequence in unaltered. Moreover, since every elliptic complex is an elliptic pre-complex, by the uniqueness clause every elliptic complex is its own induced complex.

The fact that $\Complex$ differs from $\PseudoComplex$ by terms of order zero also implies that
$(\bA_{k-1}^*\oplus\bA_k,B_{k-1}^*)$ is an overdetermined elliptic system within the calculus of Green operators. As a result:

%%%%%%%%%%%%%%%%%%%%%%%%%%%%
\begin{theorem*}[Hodge-like decomposition]
For every $k\in\Nzero$, there exists an $L^2$-orthogonal, topologically-direct decomposition of Fréchet spaces,
\beq
\Gamma(\E_k) = \scrR(\bA_{k-1})\oplus\scrR(\bA_k^*;B_k^*)\oplus \module^k\Complex,
\label{eq:Hodge_like_intro}
\eeq
where $\scrR(\bA_{k-1})$ is the range of the map $\bA_{k-1}$, $\scrR(\bA_k^*;B_k^*)$ is the range of $\bA_k^*|_{\ker B_k^*}$, and $\module^k\Complex = \ker(\bA_{k-1}^*\oplus\bA_k\oplus B_{k-1}^*)$ is finite-dimensional.  
\end{theorem*}
%%%%%%%%%%%%%%%%%%%%%%%%%%%%%

These decompositions extend to $W^{s,p}$-Sobolev regularity for every $s\in\Nzero$ and $1<p<\infty$.  

The Hodge-like decompositions imply that the finite-dimensional modules $\module^k\Complex$ are the cohomology modules of the complex $\Complex$, analogous to the harmonic modules in classical Hodge theory \cite{Sch95b}. 
Denoting $\scrN(\bA_k)=\ker(\bA_k)$, we obtain the compound decompositions:
\[
\Gamma(\E_k) = 
\lefteqn{\overbrace{\phantom{\scrR(\bA_{k-1})\oplus \module^k\Complex}}^{\scrN(\bA_k)}} \scrR(\bA_{k-1}) \oplus \underbrace{\module^k\Complex\oplus \scrR(\bA_k^*;B_k^*)}_{\scrN(\bA_{k-1}^*,B_{k-1}^*)}.
\]

%%%%%%%%%
\begin{theorem*}[Cohomology groups]
The finite-dimensional modules $\module^k\Complex$ are both the cohomology groups of the complex $\Complex$, 
and the cohomology groups of the complex $(\bA^{*}_\bullet|_{\ker B_\bullet^*})$, i.e.,
\[
\begin{aligned}
&\psi\in\scrR(\bA_{k-1})  
&\text{if and only if} 
&\qquad&
&\bA_k\psi=0 \textand \psi\,\bot\,\module^k\Complex,
\\
& \psi\in\scrR(\bA_k^*;B_k^*) 
&\text{if and only if} 
&\qquad&
&\bA_{k-1}^*|_{\ker B_{k-1}^*}\psi=0 \textand \psi\,\bot\,\module^k\Complex.
\end{aligned}
\]
\end{theorem*}
%%%%%%%%

The Hodge-like decompositions yield explicit existence and uniqueness results for the corresponding overdetermined boundary-value problems:

%%%%%%%%%%%%%%%%%%%%%%%%%%%%%
\begin{theorem*}[Overdetermined boundary-value problem]
%\label{thm:ODBVPintro}
Consider the list of data,
\[
\chi\in \Gamma(\E_{k+1}) 
\qquad 
\xi\in \Gamma(\E_{k-1})
\textand
\phi\in \Gamma(\E_k).
\] 
There exists a solution $\psi\in \Gamma(\E_k)$ to the boundary-value problem
\[
\Cases{
\bA_{k-1}^*\psi = \xi
\textand
\bA_k\psi = \chi
& \text{in $M$} \\
B_{k-1}^*\psi = B_{k-1}^*\phi
& \text{on $\dM$},
}
\]
if and only if the following integrability conditions are satisfied:
\[
\begin{gathered}
\bA_{k+1}\chi=0 \textand \bra\chi,\zeta\ket=0 \qquad \text{for every }\zeta\in\module^{k+1}\Complex
\\
\xi - \bA_{k-1}^*\phi\in \scrN(\bA_{k-2}^*,B_{k-2}^*)
\\
\bra \xi,\nu\ket =-\bra B_{k-1}^*\phi,B_{k-1}\nu\ket \qquad  \text{for every }\nu\in\module^{k-1}\Complex.
\end{gathered}
\]
The solution is unique modulo an arbitrary element in the finite-dimensional module $\module^k\Complex$. 
\end{theorem*}
%%%%%%%%%%%%%%%%%%%%%%%

The results extend to $W^{s,p}$-regularity, $s\in\Nzero$, $1<p<\infty$, with a priori Korn-like estimates,
\[
\|\psi\|_{s,p}\lesssim \|\bA_k\psi\|_{s-m_k,p}+ \|\bA_{k-1}^*\psi\|_{s-m_{k-1},p} +
\sum_{i=0}^{m_{k-1}-1} \|B_{i,k-1}^*\psi\|_{s-i-1/2,p} + 
\|\psi\|_{0,p},
\]
for every $s\ge \max(m_{k-1},m_k)$ and $B_{i,k-1}^*$ are the components of the boundary operator, each of order $i$. 
Throughout this work, we use the symbol $\lesssim$ to denote inequalities up to a multiplicative constant.

In analogy with electromagnetism, the existence and uniqueness clauses resolve the gauge freedom in the equation $\bA_{k-1}^*\psi = \xi$. This fact can be exploited in the representation of stress fields by stress potentials  \cite{Pom15,Pom18,KL21b}.

\subsection{\bfseries Application: covariant de Rham complexes} 
The most familiar occurrence of a non-trivial elliptic pre-complex is the sequence of exterior covariant derivatives $\dU:\Omega^{k}(M;\bbU)\rightarrow \Omega^{k+1}(M;\bbU)$ mentioned above. 
Indeed, $\dU$ and its $L^{2}$-adjoint, the covariant codifferential $\deltag:\W^{k+1}(M;\bbU)\rightarrow \W^{k}(M;\bbU)$, satisfy a Green's formula,
\beq
\bra \dU\psi,\eta\ket=\bra \psi,\delta^{\nabU}\eta\ket +\bra \PtD\psi,\PnD\eta\ket,
\label{eq:vector_forms_integration_by_parts}
\eeq
where 
\[
\PtD:\W^k(M;\bbU)\to \W^k(\dM;\jmath^*\bbU) 
\textand 
\PnD:\W^{k+1}(M;\bbU)\to \W^k(\dM;\jmath^*\bbU)
\] 
are the tangential and normal boundary projection operators, $\jmath:\dM\hookrightarrow M$ being the inclusion of the boundary in $M$.

Set  $\Gamma(\E_k) = \W^k(M;\bbU)$, $\Gamma(\bbG_k) = \W^k(\dM;\jmath^{*}\bbU)$, 
$A_k=\dU$, $B_k=\PtD$ and $B_k^*=\PnD$, yielding the following diagram:

\[
\begin{xy}
(-25,0)*+{\cdots}="Em1";
(-25,-20)*+{\cdots}="Gm1";
(0,0)*+{\W^{k-1}(M;\bbU)}="E0";
(35,0)*+{\W^k(M;\bbU)}="E1";
(70,0)*+{\W^{k+1}(M;\bbU)}="E2";
(95,0)*+{\cdots}="E3";
(0,-20)*+{\W^{k-1}(\dM;\jmath^*\bbU)}="G0";
(35,-20)*+{\W^k(\dM;\jmath^*\bbU)}="G1";
(70,-20)*+{\W^{k+1}(\dM;\jmath^*\bbU)}="G2";
{\ar@{->}@/^{1pc}/^{\dU}"E0";"E1"};
{\ar@{->}@/^{1pc}/^{\deltag}"E1";"E0"};
{\ar@{->}@/^{1pc}/^{\dU}"E1";"E2"};
{\ar@{->}@/^{1pc}/^{\deltag}"E2";"E1"};
{\ar@{->}@/^{1pc}/^{}"E2";"E3"};
{\ar@{->}@/^{1pc}/^{}"E3";"E2"};
{\ar@{->}@/^{1pc}/^{}"Em1";"E0"};
{\ar@{->}@/^{1pc}/^{}"E0";"Em1"};
{\ar@{->}@/_{0pc}/^{\PtD}"E0";"G0"};
{\ar@{->}@/_{0pc}/^{\PtD}"E1";"G1"};
{\ar@{->}@/_{0pc}/^{\PtD}"E2";"G2"};
{\ar@{->}@/^{1pc}/^{\PnD}"E1";"G0"};
{\ar@{->}@/^{1pc}/^{\PnD}"E2";"G1"};
{\ar@{->}@/^{1pc}/^{}"E0";"Gm1"};
{\ar@{->}@/^{1pc}/^{}"E3";"G2"};
\end{xy}
\]

The verification that $(\dU)$ is an elliptic pre-complex is straightforward: $\PtD$ and $\PnD$ are surjective; $\dU\dU=R^{\nabU}$ is tensorial, hence its order is lower than that of $\dU$; the fact that the system $(\dU\oplus\deltag,\PnD)$ is overdetermined elliptic follows an elementary calculation, identical to the calculation showing that $(d\oplus \delta,\PnD)$ is overdetermined elliptic \cite{Sch95b}. 

The induced complex is the unique sequences of operators $\mathbf{\dU}:\W^k(M;\bbU)\to \W^{k+1}(M;\bbU)$ satisfying,
\begin{enumerate}[itemsep=0pt,label=(\alph*)]
\item $\mathbf{\dU}\mathbf{\dU} = 0$
\item $\mathbf{\dU} = \dU$ on $\scrN(\pzcdel^{\nabU},\PnD)$,
\end{enumerate}
where $\pzcdel^{\nabU}=(\mathbf{\dU})^*$. In particular, $\mathbf{\dU} = \nabU$ on $\W^0(M;\bbU)$. We call the induced complex $(\mathbf{\dU})$ a \emph{covariant de-Rham complex}. The associated family of Hodge-like decompositions assume the form:
\[
\W^k(M;\bbU) = 
\lefteqn{\overbrace{\phantom{\scrR(\mathbf{\dU}\oplus \module_N^k(M)}}^{\scrN(\mathbf{\dU})}} \scrR(\mathbf{\dU}) \oplus \underbrace{\module_N^k(M)\oplus \scrR(\pzcdel^{\nabU};\PnD)}_{\scrN(\pzcdel^{\nabU},\PnD)},
\]
where
\[
\module^k_N(M)=\ker(\mathbf{\dU}\oplus\pzcdel^{\nabU}\oplus\PnD).
\]
$\module^k_N(M)$ is a covariant version of the Neumann harmonic module, and $\module^0_N(M)=\ker\nabU$ is the space of $\nabU$-parallel sections of $\bbU$. The cohomology theorem for $k=1$ yields,
 \[
\psi=\nabU v \qquad \text{if and only if} \qquad \mathbf{\dU}\psi=0 \textand \psi\,\bot\,\module^1_N(M). 
\] 
If the connection $\nabU$ is flat, then $d^{\nabU}d^{\nabU}=0$, hence $\mathbf{d^{\nabU}}=d^{\nabU}$ by the uniqueness of the induced complex. If, in addition, $M$ is simply-connected, then $\bbU\simeq M\times \bbR^N$ and $\module^1_N(M)=\{0\}$.

\subsection{\bfseries Application: Bianchi complexes}
It was Calabi who first formulated the Saint-Venant problem in the framework of \emph{double forms}, $\Wkm{k}{m}$,  $k,m\in\Nzero$, which are $\Lambda^mT^*M$-valued $k$-forms, i.e., sections of the vector bundles $\Lkm{k}{m}=\Lambda^kT^*M\otimes \Lambda^mT^*M$ \cite{deR84,Cal61,Gra70,Kul72}. 
For introductory reasons, we briefly refer here to properties and constructions pertinent to double forms; full details are presented in the body of the work and in the cited references. 

Double forms constitute a graded algebra equipped with a natural wedge product $\wedge : \Wkm{k}{m}\times \Wkm{\ell}{n}\to \Wkm{k+\ell}{m+n}$ and an involution $(\cdot)^T:\Wkm{k}{m}\to \Wkm{m}{k}$ obtained by interchanging the ``vector" and "form" parts. A $(k,k)$-form $\psi$ is called symmetric if $\psi^T = \psi$. 

Many objects in differential geometry can be cast as double forms. Riemannian metrics, Hessians of scalar functions, and Lie derivatives of the metric can be viewed as symmetric elements of $\Wkm11$; for higher order examples, the Riemannian curvature tensor $\Rm_\g$ can be viewed as a symmetric element of $\Wkm22$. Curvature tensors satisfy additional symmetries---\emph{algebraic Bianchi identities}. Calabi introduced an algebraic symmetry pertinent to all double forms, generalizing the algebraic Bianchi identity. This symmetry is the kernel of a smooth bundle map $\G:\Lkm{k}{m}\to \Lkm{k+1}{m-1}$ for $k\ge m$, and the kernel of its vector counterpart $\GV:\Lkm{k}{m}\to \Lkm{k-1}{m+1}$ for $k\le m$. 
We call double forms satisfying these symmetries \emph{Bianchi forms}, and denote them by $\Ckm{k}{m}$.
We denote by $\calPG:\Wkm{k}{m}\to\Ckm{k}{m}$ the orthogonal projection of double forms onto the space of Bianchi forms (the ``Bianchization" operator). 

The exterior covariant derivative, which is defined like for every other vector-valued form, is a first-order differential  operator
\[
\dg:  \Wkm{k}{m} \to \Wkm{k+1}{m},
\]
with a formal $L^2$-adjoint $\deltag:  \Wkm{k+1}{m} \to \Wkm{k}{m}$. Differential operators can also act on the vector part via  involution. We define
\[
\dgV:  \Wkm{k}{m} \to \Wkm{k}{m+1}
\Textand
\deltagV:  \Wkm{k}{m+1} \to \Wkm{k}{m}
\]
by $\dgV\psi = (\dg\psi^T)^T$ and $\deltagV\psi = (\deltag\psi^T)^T$.

The exterior covariant derivative $\dg$ commutes with the Bianchi sum $\G$, however not with $\GV$, implying that $\dg$ maps Bianchi forms into Bianchi forms for $k\ge m$, but not for $k<m$. 
A first-order differential operator preserving the Bianchi symmetry is
the \emph{Bianchi derivative}, 
\[
\dBianchi:\Ckm{k}{m}\to\Ckm{k+1}{m}
\qquad\text{defined by}\qquad
\dBianchi\psi =  \calPG \dg\psi.
\]
We further define $\dBianchiV:\Ckm{k}{m}\to\Ckm{k}{m+1}$ via the natural involution, and then $\delBianchi:\Ckm{k+1}{m}\to\Ckm{k}{m}$ and $\delBianchiV:\Ckm{k}{m+1}\to\Ckm{k}{m}$ as their formal $L^2$-adjoints. These first-order operators satisfy integration by parts formula,
\[
\bra \dBianchi\psi,\eta\ket = \bra \psi,\delBianchi\eta\ket + \bra B_\calG \psi,B_\calG^*\eta\ket,
\] 
where $B_\calG$ and $B_\calG^*$ are tensorial boundary operators onto corresponding spaces of Bianchi boundary forms. On $\Ckm{k}{m}$ with $k\le m$, $\delBianchi=\deltag$ is the covariant codifferential, and $B_\calG^*=\PnD$ is the contraction by the boundary unit normal; for $k>m$, both $\deltag$ and $\PnD$ are supplemented by a projection onto Bianchi forms.

The operator 
\[
\dBianchi\dBianchi: \Ckm{k}{m}\to \Ckm{k+2}{m} 
\]
turns out to be tensorial for every $k,m$ except for when $k=m-1$. The lack of tensoriality of $\dBianchi\dBianchi:\Ckm{k-1}{k}\to\Ckm{k+1}{k}$ has significant implications, notably in the resolution of the Saint-Venant problem. Indeed, the operator $\dBianchi: \Ckm01\to \Ckm11$ coincides with the Killing operator (up to a multiplicative constant), 
\beq
\dBianchi\omega  = \tfrac12 \calL_{\omega^\#}\g,
\label{eq:Lie_bianchi_intro}
\eeq
where $\omega^\#$ is the vector field corresponding to the 1-form $\omega$. Thus, 
\[
\dBianchi\dBianchi\omega = \tfrac12 \dg\calL_{\omega^{\sharp}}\g,
\] 
which is a second-order operator (and in particular non-zero) even in a Euclidean setting. 
Hence the first-order operator $\dBianchi$ fails to ``detect" the Killing operator, which is why one must resort to a second-order compatibility condition. 

In \cite{KL21a} we rewrote the curl-curl operator and its adjoint as second-order operators acting on Bianchi forms, $H:\Ckm{k}{m}\to \Ckm{k+1}{m+1}$
and $H^*: \Ckm{k+1}{m+1}\to\Ckm{k}{m}$, given by
\[
H=\tfrac12(\dg\dgV+\dgV\dg) 
\Textand 
H^*=\tfrac12(\deltag\deltagV+\deltagV\deltag).
\]
These operators satisfy integration by part formulas involving both tensorial and first-order surjective boundary operators, which we denote by $B_{H}$ and $B_{H}^*$,
\[
\bra H\psi,\eta\ket=\bra \psi,H^*\eta\ket +\bra B_H\psi,B_H^*\eta\ket. 
\]
An explicit calculation shows that $H\dBianchi$ and $\dBianchi H$ are differential operators of order $1$, and that for every $1\le m\le d$, the following sequence, which we break into two lines, constitutes an elliptic pre-complex:
\[
\begin{xy}
(-25,0)*+{0}="Em1";
(0,0)*+{\Ckm{0}{m}}="E0";
(30,0)*+{\Ckm{1}{m}}="E1";
(60,0)*+{\cdots}="E2";
(90,0)*+{\Ckm{m}{m}}="E3";
(-25,-20)*+{}="Gm1";
(0,-20)*+{\plCkm{0}{m}\oplus\plCkm{0}{m-1}}="G0";
(30,-20)*+{\plCkm{1}{m}\oplus\plCkm{1}{m-1}}="G1";
(60,-20)*+{\cdots}="G2";
(90,-20)*+{(\plCkm{m}{m})^2}="G3";
{\ar@{->}@/^{1pc}/^{\dBianchi}"E0";"E1"};
{\ar@{->}@/^{1pc}/^{\delBianchi}"E1";"E0"};
{\ar@{->}@/^{1pc}/^{\dBianchi}"E1";"E2"};
{\ar@{->}@/^{1pc}/^{\delBianchi}"E2";"E1"};
{\ar@{->}@/^{1pc}/^{\dBianchi}"E2";"E3"};
{\ar@{->}@/^{1pc}/^{\delBianchi}"E3";"E2"};
{\ar@{->}@/^{1pc}/^0"Em1";"E0"};
{\ar@{->}@/^{1pc}/^0"E0";"Em1"};
{\ar@{->}@/_{0pc}/^{B_\calG}"E0";"G0"};
{\ar@{->}@/_{0pc}/^{B_\calG}"E1";"G1"};
{\ar@{->}@/_{0pc}/^{B_\calG}"E2";"G2"};
{\ar@{->}@/^{1pc}/^{B_\calG^*}"E1";"G0"};
{\ar@{->}@/^{1pc}/^{B_\calG^*}"E2";"G1"};
{\ar@{->}@/^{1pc}/^{B_\calG^*}"E3";"G2"};
{\ar@{->}@/_{0pc}/^{B_H}"E3";"G3"};
\end{xy}
\]
\[
\begin{xy}
(-30,0)*+{\Ckm{m}{m}}="Em1";
(0,0)*+{\Ckm{m+1}{m+1}}="E0";
(30,0)*+{\cdots}="E1";
(60,0)*+{\Ckm{d}{m+1}}="E2";
(85,0)*+{0}="E3";
(-30,-20)*+{(\plCkm{m}{m})^2}="Gm1";
(0,-20)*+{\plCkm{m+1}{m+1}\oplus\plCkm{m+1}{m}}="G0";
(30,-20)*+{\cdots}="G1";
(60,-20)*+{\plCkm{d}{m+1}\oplus\plCkm{d}{m}}="G2";
(85,-20)*+{}="G3";
{\ar@{->}@/^{1pc}/^{\dBianchi}"E0";"E1"};
{\ar@{->}@/^{1pc}/^{\delBianchi}"E1";"E0"};
{\ar@{->}@/^{1pc}/^{\dBianchi}"E1";"E2"};
{\ar@{->}@/^{1pc}/^{\delBianchi}"E2";"E1"};
{\ar@{->}@/^{1pc}/^0"E2";"E3"};
{\ar@{->}@/^{1pc}/^0"E3";"E2"};
{\ar@{->}@/^{1pc}/^{H}"Em1";"E0"};
{\ar@{->}@/^{1pc}/^{H^*}"E0";"Em1"};
{\ar@{->}@/_{0pc}/^{B_H}"Em1";"Gm1"};
{\ar@{->}@/_{0pc}/^{B_\calG}"E0";"G0"};
{\ar@{->}@/_{0pc}/^{B_\calG}"E1";"G1"};
{\ar@{->}@/_{0pc}/^{B_\calG}"E2";"G2"};
{\ar@{->}@/^{1pc}/^{B_\calG^*}"E1";"G0"};
{\ar@{->}@/^{1pc}/^{B_\calG^*}"E2";"G1"};
{\ar@{->}@/^{1pc}/^{B_H^*}"E0";"Gm1"};
\end{xy}
\]
We name the induced complexes \emph{Bianchi complexes}. 
Specifically, there exists a uniquely determined chains of operators,
\[
\begin{aligned}
&\DBianchi: \Ckm{k}{m} \to \Ckm{k+1}{m}  
&\qquad\qquad
&k=0,\dots,m-1 \\
&\bHg: \Ckm{m}{m} \to \Ckm{m+1}{m+1} \\
&\DBianchi: \Ckm{k}{m+1} \to \Ckm{k+1}{m+1}  
&\qquad\qquad
&k=m+1,\dots,d-1,
\end{aligned}
\]
satisfying
\[
\DBianchi\DBianchi=0 \qquad \DBianchi \bHg=0 \qquad \bHg\DBianchi=0,
\]
such that $\DBianchi-\dBianchi$ and $\bHg-H$ are Green operators of order zero, vanishing on $\Ckm{0}{m}$. We denote the corresponding adjoints by $\DelBianchi$ and $\bHg^*$, with $\DelBianchi-\delBianchi$ and $\bHg^*-H^*$ of order zero as well. Thus, for every $0\leq k,m\leq d$, there corresponds a Hodge-like decomposition \eqref{eq:Hodge_like_intro}, together with integrability conditions and uniqueness clauses for the corresponding overdetermined boundary-value problems. We name the corresponding cohomology groups the \emph{Bianchi cohomology groups}, denoted by $\Bmodule{k}{m}$. This notation emphasizes that $\dim \Bmodule{k}{m}$ is an invariant of the Riemannian structure (by the uniqueness of the induced complex and the fact that all the operators are Riemannian constructions).

We next list some important cases. 

\subsubsection{\bfseries The Hessian complex}
We start with the Bianchi complex associated with $m=0$, naming it the Hessian complex, since the first operator in the chain is the Hessian of functions, $H:\Ckm{0}{0}\to\Ckm{1}{1}$. The first Hodge-like decomposition obtained from this complex decomposes scalar functions,
\[
\Ckm{0}{0} = \Bmodule{0}{0} \oplus \scrR(H^*;B_H^*) 
\quad\text{where}\qquad 
\Bmodule{0}{0}= \ker H.
\]
Thus,
%%%%%%%%%%%%%%%%%%%%%%
\begin{theorem}
\label{thm:hessian0}
For $f\in C^\infty(M)$,
\[
f\in\scrR(H^*;B_H^*) \qquad \text{if and only if \,\,  $f\,\bot\,\Bmodule{0}{0}$}
\]
Sobolev versions for $f\in W^{s,p}(M)$ hold with the required adjustments.  
\end{theorem}
%%%%%%%%%%%%%%%%%%%%%% 

The operator $H^*=\sum_{i,j}i_{E_i}i^V_{E_j}\nabla_{E_i}\nabla_{E_j}$, also known as the div-div operator and sometimes denoted by $\delta\delta$ or $\delta^2$, appears in variation formulas for scalar curvatures \cite{BE69}. In a Euclidean domain, $Hf=0$ implies that $f$ is a linear function, hence $\dim\scrB^{0}_0(\Omega,\frake)=d+1$.  In a general Riemannian geometry, it always holds that $\dim\Bmodule{0}{0}\leq d+1$, since $f\in   \Bmodule{0}{0}$ implies that $df$ is a global parallel form, which greatly restricts the curvature tensor \cite[p.~76]{Pet16}.

Next, still in the context of the Hessian complex, is the decomposition of symmetric $(1,1)$-forms:
\[
\Ckm{1}{1} =
\lefteqn{\overbrace{\phantom{\scrR(H)\oplus \Bmodule{1}{0}}}^{\scrN(\DBianchi)}} 
\scrR(H) \oplus \underbrace{\Bmodule{1}{0} \oplus \scrR(\DelBianchi;B_\calG^*)}_{\scrN(H^*,B_H^*)}.
\]
where
\[
\Bmodule{1}{0}= \ker(H^*\oplus\DBianchi\oplus B_H^*).
\]
This yields the following characterization of Hessians:

%%%%%%%%%%%%%%%%%%%%%%
\begin{theorem}
\label{thm:hessian}
For $\sigma\in\Ckm{1}{1}$,
\[
\sigma\in\scrR(H) \qquad \text{if and only if \,\,  $\DBianchi\sigma=0$\,\, and \,\,  $\sigma\,\bot\,\Bmodule{1}{0}$}
\]
Sobolev versions for $\sigma\in W^{s,p}\Ckm{1}{1}$ hold with the required adjustments.  
\end{theorem}
%%%%%%%%%%%%%%%%%%%%%% 

The characterization of Hessians among symmetric tensors in a general Riemannian manifolds has also been an open question \cite{BE69,Bry13}.  Since the Hessian of a function is in particular in the range of the Killing operator, this problem can seen as a partial instance of the Saint-Venant problem. In simply-connected Euclidean domains, the solution follows from the Poincar\'e Lemma, and amounts to the condition that $\dg\sigma=0$. By the uniqueness clause of the induced complex, $\DBianchi=\dg$ in a Euclidean domain, hence \thmref{thm:hessian} generalizes this condition. This in turn implies that $\scrB^1_0(\Omega,\frake)=\{0\}$ for $(\Omega,\frake)$ simply-connected and Euclidean.

\subsubsection{\bfseries The Calabi complex}
We proceed with the Bianchi complex associated with $m=1$, which brings us to the original motivation of this work; in compliance with the literature, we call the resulting complex the Calabi complex. 

The first decomposition concerns $(0,1)$-forms (i.e., one-forms):
\[
\Ckm{0}{1} = \Bmodule{0}{1}\oplus \scrR(\delBianchi;B^*_\calG) 
\quad\text{where}\qquad 
\Bmodule{0}{1}=\ker\dBianchi.
\]
In this case, $B_\calG^*=\PnD$ and $\delBianchi=\deltag$, hence:

%%%%%%%%%%%%%%%%%%%%%%
\begin{theorem}
\label{thm:Saint_venant0}
For $\xi\in\Ckm{0}{1}$,
\[
\xi\in\scrR(\deltag;\PnD) \qquad \text{if and only if \,\,  $\xi\,\bot\,\Bmodule{0}{1}$}.
\]
Sobolev versions for $\xi\in W^{s,p}\Ckm{1}{1}$ hold with the required adjustments.  
\end{theorem}
%%%%%%%%%%%%%%%%%%%%%% 

By \eqref{eq:Lie_bianchi_intro}, the finite-dimensional module $\Bmodule{0}{1}$ is the space of Killing $1$-forms. 
It is worth noting that \thmref{thm:Saint_venant0} essentially amounts to the integrability conditions for the elliptic differential system $(\delBianchi\dBianchi,B_\calG^*\dBianchi)$, which is well-known in the literature and is central in the theory of linear elasticity (e.g., \cite[pp.~465--466]{Tay11a} and \cite{SS87}). Hence, the machinery of elliptic pre-complexes is not needed to prove it. The theorem also shows that the range of the operator $\deltag$ exhausts $\Ckm{0}{1}$ up to a finite-dimensional module, which is related to the fact that it is underdetermined elliptic \cite{Hin23}.  

The next decomposition concerns $(1,1)$-Bianchi forms, which coincide with the $(1,1)$-symmetric forms:
\[
\Ckm{1}{1}=\lefteqn{\overbrace{\phantom{\scrR(\dBianchi)\oplus \Bmodule{1}{1}}}^{\scrN(\bHg)}} 
\scrR(\dBianchi) \oplus \underbrace{\Bmodule{1}{1}\oplus \scrR(\bHg^*;B_H^*)}_{\scrN(\delBianchi,B_\calG^*)},
\]
where
\[
\Bmodule{1}{1}=\ker(\bHg\oplus\delBianchi\oplus B_\calG^*).
\]

The decomposition of $\Ckm{1}{1}$ refines and generalizes the decomposition obtained in \cite{BE69} for a closed manifold. 
The cohomology groups theorem associated with this decomposition resolves the Saint-Venant problem:

%%%%%%%%%%%%%%%%%%%%%%
\begin{theorem}
\label{thm:Saint_venant}
For $\sigma\in\Ckm{1}{1}$,
\[
\sigma\in \scrR(\dBianchi) \qquad \text{if and only if \,\,  $\bHg\sigma=0$\,\, and \,\,  $\sigma\,\bot\,\Bmodule{1}{1}$}
\]
Sobolev versions for $\sigma\in W^{s,p}\Ckm{1}{1}$ hold with the required adjustments.  
\end{theorem}
%%%%%%%%%%%%%%%%%%%%%% 

As in the Hessian complex, the classical theorem for simply-connected Euclidean domains together with the uniqueness of the complex imply that $\bHg=H$ and $\scrB^1_1(\Omega,\frake)=\{0\}$. Moreover, this same decomposition resolves also the existence of stress potentials with normal boundary conditions:

%%%%%%%%%%%%%%%%%%%%%%
\begin{theorem}
\label{thm:stress}
Let $\sigma\in\Ckm{1}{1}$ satisfy 
\[
\deltag\sigma=0,
\qquad
\PnD\sigma=0
\Textand 
\sigma\bot\,\Bmodule{1}{1}. 
\] 
There exists a $\psi\in \Ckm{2}{2}$ satisfying
\[
\begin{split}
&\sigma=\bHg^*\psi 
\Textand 
B_H^*\psi=0.  
\end{split}
\]
Sobolev versions for $\sigma\in W^{s,p}\Ckm{1}{1}$ holds with the required adjustments. 
\end{theorem}
%%%%%%%%%%%%%%%%%%%%%% 

The next decomposition associated with the Calabi complex concerns $(2,2)$-forms:
\[
\Ckm{2}{2}=
\lefteqn{\overbrace{\phantom{\scrR(\bHg)\oplus \Bmodule{2}{1}}}^{\scrN(\DBianchi)}} 
\scrR(\bHg) \oplus \underbrace{\Bmodule{2}{1} \oplus \scrR(\DelBianchi;B_\calG^*)}_{\scrN(\bHg^*,B_H^*)},
\]
where
\[
\Bmodule{2}{1}= \ker(\bHg^*\oplus\DBianchi\oplus B_H^*).
\]

Using this decomposition, one is able to solve non-homogeneous boundary-value problems, generalizing results obtained in \cite{KL21b} in the context of linearized stress equations: 
\begin{theorem}
\label{thm:linearized_stress_equations}
Consider the data,
\[
\calR\in \Ckm{2}{2}  
\qquad 
\xi\in \Ckm{0}{1}
\Textand
\phi\in\Ckm{1}{1}.
\] 
There exists a solution $\sigma\in \Ckm{1}{1}$ to the boundary-value problem
\[
\Cases{
\deltag\sigma = \xi
\textand
\bHg\sigma = \calR
& \text{in $M$} \\
\PnD\sigma = \PnD\phi
& \text{on $\dM$},}
\]
if and only if the following integrability conditions are satisfied: 
\[
\begin{gathered}
\DBianchi\calR=0 \textand \bra\calR,\zeta\ket=0 \qquad \text{for every }\zeta\in\Bmodule{2}{1} \\
\bra \xi,\nu\ket =-\bra \PnD\phi,\PtD\nu\ket \qquad  \text{for every }\nu\in\Bmodule{0}{1}.
\end{gathered}
\]
The solution is unique modulo an arbitrary element in $\Bmodule{1}{1}$. 
\end{theorem} 
%%%%%%%%%%%%%%%%%%%%%% 

The next level in the Calabi complex concerns the problem, 
\[
\Cases{
\bHg^*\psi = \sigma
\textand
\DBianchi\psi = \chi
& \text{in $M$} \\
B_H^*\psi = B_H^*\theta
& \text{on $\dM$}.
}
\]
This in turn motivates the revisiting of \thmref{thm:stress}, where we find that the existence of stress potentials can be enhanced with a canonical choice of gauge and a uniqueness clause, also generalizing results obtained in \cite{KL21b}: 

%%%%%%%%%%%%%%%%%%%%%%
\begin{theorem}
\label{thm:stressG}
Let $\sigma\in\Ckm{1}{1}$ satisfy 
\[
\deltag\sigma=0,
\qquad
\PnD\sigma=0
\Textand 
\sigma\bot\,\Bmodule{1}{1}. 
\] 
There exists a $\psi\in \Ckm{2}{2}$ satisfying
\[
\begin{split}
&\sigma=\bHg^*\psi 
\Textand 
B_H^*\psi=0.  
\end{split}
\]
Moreover, $\psi$ can be chosen to satisfy the gauge condition, 
\[
\DBianchi\psi=0. 
\]
In this case, $\psi$ is unique up to an element in $\Bmodule{2}{1}$. 
Sobolev versions for $\sigma\in W^{s,p}\Ckm{1}{1}$ hold with the required adjustments. 
\end{theorem}
%%%%%%%%%%%%%%%%%%%%%% 

\subsection{\bfseries Main open question: geometric meaning of the cohomology} 
As noted above, the Hodge-like decompositions identify cohomology groups, which are finite-dimensional modules consisting of smooth sections. These modules generalize the harmonic modules in Hodge theory. As is well-known, the dimensions of the harmonic modules are topological invariants, and in particular, independent of the metric. 
In general, the modules $\module^k\Complex$ cannot be expected to be topological invariants. For example, in the covariant de Rham complex, $\module^{0}_{N}(M)$ is the space of parallel sections, which is connection-dependent; in the Calabi complex, $\Bmodule01$ is the space of Killing forms, which is metric-dependent. Of special interest for applications is to know whether the modules concerning the Saint-Venant problem are topological invariants.

\subsection{\bfseries The structure of this paper}
\secref{sec:pseudo_section} contains a brief review of pseudodifferential operators in the context of boundary-value problems. There is a huge body of literature on this subject; we only review those details that are relevant to the scope of this work, and slightly extend some of them to better suite our framework later on.  \secref{sec:elip_pre_comp} starts with a review on a specialized class of Green operators (\secref{sec:adapting}). We then define in \secref{sec:pre-complex} elliptic pre-complexes. The main theorem regarding the existence of an induced complex is stated in \secref{sec:main_theorem}. In \secref{sec:applications_complex} we present the central consequences of our main theorem, notably the Hodge-like decomposition and the solution of boundary-value problems.
\secref{sec:main_proof} is devoted to the proof of our main theorem, divided into six subsections. 
Finally, the main applications are presented in \secref{sec:bianchi}, with notably the resolution of the Saint-Venant problem in arbitrary geometries.

\subsection{\bfseries Acknowledgments}
We are grateful to Matania Ben-Artzi and Or Hershkovits for various discussions along the preparation of this paper. We are particularly thankful for Gerd Grubb and Cy Maor for their invaluable comments on the manuscript.
This project was funded by ISF Grant 560/22.  RL was funded by an Azrieli Fellowship given by the Azrieli Foundation. 

%%%%%%%%%%%%%%%%%%%%%%%%%%%%%%%%%%%%%%%%%%%%%%%%%%
\section{Preliminary survey: pseudodifferential boundary-value problems}
\label{sec:pseudo_section}

This section contains a brief review of pseudodifferential operators in the context of boundary-value problems. There is a huge body of literature on this subject; we only review those details that are relevant to the scope of this work, and slightly extend some of them to better suit our framework later on. 

%%%%%%%%%%%%%%%%%%%%%%%%%%%%%%%%%%%%%%%%%%%%%%%%%%
\subsection{Pseudodifferential operators having the transmission property} 
\label{sec:psdos}

Let  $(\tM,\tg)$ be a closed $d$-dimensional Riemannian manifold, endowed with a volume form $\Volume\in\W^d(\tM)$; our study can be extended to non-compact manifolds, but  for simplicity, we restrict our attention to compact ones. Let $M \hookrightarrow \tM$ be a compact embedded submanifold of the same dimension having a smooth boundary. Since every compact Riemannian manifold with smooth boundary can be embedded in its closed double \cite[p.~226]{Lee12}, we 
will henceforth view every compact Riemannian manifold with smooth boundary  $M$ as smoothly embedded in a closed ambient Riemannian manifold $\tM$.

Let $\tE,\tF\to\tM$ be Riemannian vector bundles over $\tM$; denote by $\E=\tE|_M$ and  $\bbF=\tF|_M$ the pullback bundles, which are vector bundles over $M$. Let $\bbJ,\bbG\to \dM$ be Riemannian vector bundles over $\dM$. 

For $\psi,\eta\in L^2\Gamma(\tE)$, we denote their $L^2$-inner product by
\[
\bra\psi,\eta\ket = \int_{\tM} (\psi,\eta)_{\tE}\,\Volume.
\]
We use the same notation for the $L^2$-inner product associated with sections over $M$.
Likewise, for $\rho,\tau\in L^2\Gamma(\bbG)$, we denote the induced $L^2$-inner product on the boundary $\dM$ by
\[
\bra\rho,\tau\ket = \int_\dM (\rho,\tau)_{\bbG}\,\, \Volume_{\jmath^*g},
\]
where $\jmath:\dM\hookrightarrow M$ is the inclusion map of the boundary. 

A differential operator  $\Gamma(\tE)\to\Gamma(\tF)$ is a linear map that can be represented as an $\R^{N_1}\to \R^{N_2}$ differential operator in any local trivializations of $\tE$ and $\tF$.
Since this definition is local, it extends to linear maps $\Gamma(\E)\to\Gamma(\bbF)$ and boundary differential operators $\Gamma(\E)\to\Gamma(\bbG)$. 

Differential operators are the prominent example of a larger class of linear operators, known as \emph{pseudodifferential operators}.  
In $\R^d$, pseudodifferential operators are defined through their action on the Fourier transform $\hat{u}(\xi)$ of their argument $u(x)$ via a so-called \emph{symbol matrix} $A(x,\xi)$. On a manifold, their definition is based on their definition in $\R^d$ via the pullback by coordinate charts. A pseudodifferential operator is said to be of \emph{order}  $m\in\bbR$, if its associated symbol matrix satisfies a growth condition with exponent $m$. 
We adopt the notation of \cite{RS82}, and denote by $L^m(\tM,\tE,\tF)$ the space of pseudodifferential operators of order $m$. By definition, a pseudodifferential operator of order $m$ is also of any order greater than $m$.
We denote by
\[
L(\tM,\tE,\tF)=\bigcup_{m\in\bbZ} L^m(\tM,\tE,\tF)
\] 
the space of all pseudodifferential operators and by
\[
L^{-\infty}(\tM,\tE,\tF)=\bigcap_{m\in\bbZ}L^m(\tM,\tE,\tF)
\] 
the space of so-called \emph{smoothing operators}.
We denote by $\ord(A)$ the set of $m\in\bbZ\cup\{-\infty\}$ such that $A\in L^m(\tM,\tE,\tF)$; we say that $\ord(A) < \ord(Q)$ if $\ord(Q) \subsetneq \ord(A)$. 
 
Pseudodifferential operators were introduced as a class of operators, rich enough to encompass both differential operators and singular integral operators arising as inverse operators (parametrices) for elliptic differential systems. We refer the reader to the abundant literature on the subject \cite{Hor03,RS82,WRL95,Gru96,Tay11b,Tay11c}; in the following, we will only list those properties of pseudodifferential operators that are of relevance to the present work.

Every pseudodifferential operator $A\in L^m(\tM,\tE,\tF)$ is associated with a \emph{symbol}, generalizing the principal symbol of a differential operator \cite[pp.~176--178]{Tay11a} or \cite[Sec.~1.2.4.1]{RS82}. On a manifold, unlike in $\R^d$, the symbol is an equivalence class of smooth bundle maps $\sigma_A:T^*\tM\otimes \tE\to \tF$, 
\[
\sigma_A: (x,\xi,v)\mapsto \sigma_A(x,\xi)v
\qquad
\text{for $x\in\tM$, $\xi\in T^*_x\tM$ and $v\in \tE_x$}.
\]
The primary role of symbols is to reduce analytical properties of pseudodifferential operators into algebraic properties of their symbols; notably, it allows for a functional classification of pseudodifferential operators.
Adopting the notation of \cite{RS82}, the space of all symbols of order $m$ is denoted by $S^m(\tM,\tE,\tF)$, where $\sigma_A\in S^m(\tM,\tE,\tF)$ implies $m$-growth conditions with respect to the variable $\xi$. Symbol are defined up to lower-order terms, which is to say that if $A,Q\in L(\tM,\tE,\tF)$ with $\ord(A)>\ord(Q)$, then
\[
\sigma_{A+Q}(x,\xi)=\sigma_A(x,\xi).
\]

An operator $A\in L^m(\tM,\tE,\tF)$ is said to be \emph{homogeneous} if for every $\lambda  > 0$ and $|\xi|$ large enough,
\[
\sigma_A(x,\lambda\xi) = \lambda^m\,\sigma_A(x,\xi).
\] 
Homogeneity holds trivially for differential operators, but does not for general pseudodifferential operators. Operators having symbols that possess locally, as $\xi\to\infty$, an asymptotic expansion  of homogeneous symbols are called \emph{classical}. Sticking with the notation of \cite{RS82}, we denote by $L^m_{\mathrm{cl}}(\tM,\tE,\tF)$ the space of classical pseudodifferential operators of order $m$. The importance of this class is in the homomorphism properties satisfied by their symbols, which is used repeatedly in this work.

A pseudodifferential operator $A\in L(\tM,\tE,\tF)$ is first and foremost a continuous linear map between Fr\'echet spaces,
\[
A:\Gamma(\tE)\to  \Gamma(\tF),
\]
with the topology induced by the uniform convergence of sections along with all their derivatives.
Pseudodifferential operators are closed under composition 
\beq
AQ \in L^{m_A+m_Q}(\tM,\tE,\tF')
\label{eq:compose_AQ}
\eeq
for every $A\in L^{m_A}(\tM,\tF,\tF')$ and $Q\in L^{m_Q}(\tM,\tE,\tF)$.
Moreover, every $A\in L^m(\tM,\tE,\tF)$ admits a \emph{formal adjoint} $A^*\in L^m(\tM,\tF,\tE)$, given by the property that 
\beq
\bra A\psi,\eta\ket = \bra \psi,A^*\eta\ket
\qquad
\text{for every $\psi \in \Gamma(\tE)$ and $\eta\in\Gamma(\tF)$}.
\label{eq:adjoint_neigh_def}
\eeq
The class $L_{\mathrm{cl}}(\tM,\tE,\tF)$ of classical operators is closed under composition and adjointness as well, with the additional property that classical symbols satisfy the homomorphism properties \cite[p.~74]{RS82}, 
\[
\sigma_{AQ}(x,\xi)=\sigma_A(x,\xi)\circ \sigma_Q(x,\xi) 
\Textand 
\sigma_{A^*}(x,\xi)=(\sigma_A(x,\xi))^*. 
\]

This work is concerned with
Sobolev sections of vector bundles, $W^{s,p}\Gamma(\tE)$ defined for $s\in\bbR$ and $1<p<\infty$.
The definition goes through first defining scalar-valued Sobolev functions on $\bbR^d$, then on domains $\Omega\subset\bbR^d$, and then on closed manifolds by means of partitions of unity and coordinate charts. Finally, Sobolev sections of vector bundles over closed manifolds are defined \cite[Sec.~1.2.1.2]{RS82}.

There are several variants of Sobolev spaces. The spaces $H^{s,p}\Gamma(\tE)$ (also known as \emph{Bessel-potential} spaces) are defined for every $s\in\bbR$ and $1<p<\infty$ by means of the Fourier transform
\cite[pp.~42--46]{RS82}, \cite[pp.~291--293]{Gru90} . 
For $s\in\Nzero$, $H^{s,p}\Gamma(\tE)$ is the completion of $\Gamma(\tE)$ with respect to the Sobolev norm,
\[
\|\psi\|_{s,p}=\sum_{\Nzero\ni \alpha\leq s}\|\nabla^{\alpha}\psi\|_{L^{p}}
\]
where $\nabla$ is any connection on $\tE$. 

Our eventual goal is to pass to manifolds with boundary, where trace theorems are being invoked. For $s\in\bbR_{+}\setminus\Nzero$, the spaces $H^{s,p}\Gamma(\tE)$ are insufficient for these theorems to hold. This is where \emph{Besov spaces} $B^{s,p}\Gamma(\tE)$, $s\in\bbR$ and $1<p<\infty$, 
come in \cite[p.~293]{Gru90}, \cite[p.~45--46]{RS82}. As in \cite{Gru90}, we set
\[
W^{s,p}\Gamma(\tE)=\begin{cases}
H^{s,p}\Gamma(\tE) & s\in\bbZ 
\\ B^{s,p}\Gamma(\tE) & s\in \bbR\setminus \bbZ. 
\end{cases}
\] 
Our results will be formulated for $W^{s,p}$-spaces for $s\in\Nzero$, i.e., for ``standard" Sobolev sections. References to non-integer $s$ are needed because of the mapping properties of trace operators. 

Pseudodifferential operators satisfy various mapping properties with respect to these Sobolev spaces. Most prominently, $A\in L^m(\tM;\tE,\tF)$ extends to a continuous linear map \cite[p.~312]{Gru90},
\beq
A:W^{s,p}\Gamma(\tE)\to W^{s-m,p}\Gamma(\tF) 
\label{eq:mapping_without_boundary}
\eeq
for every $s\in \bbR$ and $1< p<\infty$.
In particular, every $A\in L^{-\infty}(\tM,\tE,\tF)$ extends into a map
\[
A: \D'\Gamma(\tE)\to \Gamma(\tF).
\]

A pseudodifferential operator $E\in L(\tM,\tE,\tF)$ is called \emph{elliptic} if $\sigma_E(x,\xi):\tE_x\to \tF_x$ is an isomorphism for every $x\in\tM$ and for every $|\xi|$ large enough.
A \emph{parametrix} (also known as an \emph{approximate inverse}) for $E$ is an operator $P:\Gamma(\tF)\to \Gamma(\tE)$ satisfying
\[
PE-\id\in L^{-\infty}(\tM,\tE,\tE) 
\Textand 
EP-\id\in L^{-\infty}(\tM,\tF,\tF).
\]
Every elliptic $E\in L^m(\tM,\tE,\tF)$ admits a parametrix $P\in L^{-m}(\tM,\tF,\tE)$, which is unique modulo $L^{-\infty}(\tM,\tF,\tE)$ \cite[p.~76]{RS82}. 

Pseudodifferential operators are generally defined on a manifold without boundary. We are interested in a subclass of pseudodifferential operators over $\tM$ that truncate ``nicely" to $M$, while retaining the closure of the calculus to adjoints, compositions and parametrices. Such operators were introduced by H\"ormander \cite[p.~105]{Hor03}; our exposition is based on a combination of \cite[p.~23]{Gru96}, \cite[Sec.~2.3]{RS82} and \cite[p.~512]{WRL95}. 

Let $r_+:\D'\Gamma(\tF)\to \D'\Gamma(\bbF)$ be the restriction operator,
\[
r_+\psi=\psi|_M, 
\]
i.e., the restriction of $\psi$ acting on test functions with support in $M$,
and let $e_+:\Gamma(\E)\to \D'\Gamma(\tE)$ be the extension-by-zero operator, 
\[
e_+\psi=\begin{cases}
\psi & \text{in $M$} \\
0 & \text{in $\tM\setminus M$}.
\end{cases}
\]
A pseudodifferential operator $A\in L(\tM,\tE,\tF)$ is said to have the \emph{transmission property} with respect to $M$ when its \emph{truncation},
\[
A_+=r_+Ae_+:\Gamma(\E)\to \D'\Gamma(\bbF),
\]
is a continuous map $A_+:\Gamma(\E)\to \Gamma(\bbF)$.

We adopt the notation of \cite{RS82} and denote the space of 
all classical operators of order  $m$ over $\tM$ having the transmission property with respect to $M$ by $\OPAm(\tE,\tF)$, or by $\OPAm$ when there is no ambiguity,
and let
\[
\OPA=\bigcup_{m\in\bbZ} \OPAm.
\]

There is ample discussion in the cited literature on sufficient conditions for a pseudodifferential operator to have the transmission property. For our purposes, we will only mention that every differential operator is in $\OPA$, and that $\OPA$ is closed under adjoints, compositions and parametrices \cite[p.~136]{RS82} (Prop.~2 requires the elements to be properly-supported \cite[p.~86]{Hor03}, but every pseudodifferential operator is properly-supported in a compact manifold). 

Truncations of operators in $\OPAm$ satisfy mapping properties as well, which requires to define Sobolev spaces on manifolds with boundaries. 
For $s<0$, we note that the spaces $W^{s,p}\Gamma(\tE)$ consists of distributions. We define \cite[pp.~294--297]{Gru90},
\[
\begin{split}
&W^{s,p}\Gamma(\E) = W^{s,p}\Gamma(\tE)  / \{\omega \in W^{s,p}\Gamma(\tE) ~:~ \supp \omega \subseteq \overline{\tM\setminus M}\},
\end{split}
\]
and
\[
\begin{split}
W_0^{s,p}\Gamma(\E) = \{\omega\in W^{s,p}\Gamma(\tE) ~:~ \supp\omega \subseteq M\}.
\end{split}
\]
For $1/p+1/q=1$, 
\[
(W^{s,q}_0\Gamma(\bbE))^*=W^{-s,p}\Gamma(\bbE).
\]
The mapping properties of $A\in\OPAm$ are given in \cite[p.~312]{Gru90},   
\beq
A_+:W^{s,p}\Gamma(\E)\to W^{s-m,p}\Gamma(\bbF)
\label{eq:Sobolev_with_boundary}
\eeq
for every $\bbZ\ni s\ge 1/p-1$ and $1< p< \infty$.
Henceforth, we remove the ``$+$" subscript from the truncation of a differential operator. Since these always act locally, this should cause no confusion.

%%%%%%%%%%%%%%%%%%%%%%%%%%%%%%%%%%%%%%%%%%%%%%%%%%
\subsection{Integration by parts, trace operators and normal conditions} 
\label{sec:integration_by_parts}

As mentioned in the previous section, the space $\OPA$ of classical operators having the transmission property  is closed under adjoints, i.e., if $A\in \OPA(\tE,\tF)$ then $A^*\in \OPA(\tF,\tE)$, where $A^*$ is defined by \eqref{eq:adjoint_neigh_def}. Consider the truncation of the adjoint,
\[
(A^*)_+=r_+A^*e_+.
\]
The question is whether the truncation $(A^*)_+$ of $A^*$ is in some sense adjoint to the truncation $A_+$ of $A$; for example, does it hold that
\[
\bra A_+\psi,\eta\ket = \bra \psi,(A^*)_+\eta\ket
\qquad
\text{for every $\psi \in \Gamma(\E)$ and $\eta\in\Gamma(\bbF)$}\, ?
\]
\cite[p.~36]{Gru96} shows that every $A\in \OPAm$ can be written as a sum
\[
A = D + Q
\]
where $D\in \OPAm$ is a differential operator and $Q\in \OPAm$ satisfies 
\[
\bra Q_+\psi,\eta\ket = \bra \psi,(Q^*)_+\eta\ket
\qquad
\text{for every $\psi \in \Gamma(\E)$ and $\eta\in\Gamma(\bbF)$}.
\]
Since $D$ is a differential operator, 
\[
\bra D\psi,\eta\ket =\bra \psi, D^*\eta\ket
\qquad
\text{for every $\psi\in\Gamma(\E)$ and $\eta\in\Gamma_c(\bbF)$},
\]
from which follows that
\[
\bra A_+\psi,\eta\ket =\bra \psi, (A^*)_+\eta\ket
\qquad
\text{for every $\psi\in\Gamma(\E)$  and $\eta\in\Gamma_c(\bbF)$}.
\]
This formula holds also for non-compactly-supported sections when $\ord(A) \leq 0$, since in this case $A_+:L^2\Gamma(\E)\to L^2\Gamma(\bbF)$ continuously, hence admits an $L^2$-adjoint, $(A_+)^*$, and by the uniqueness of the adjoint, $(A_+)^* = (A^*)_+$. We henceforth denote $(A^*)_+$ simply by $A^*_+$, recalling that the adjointness property is only with respect to compactly-supported sections.

\cite[pp.~37--38]{Gru96} denotes by $\rho_N:\Gamma(\E)\to (\Gamma(\jmath^*\E))^N$ the \emph{Cauchy-boundary operator}, 
\[
\rho_N\psi=(D_\frakn^0\psi,D_\frakn\psi,...,D_\frakn^{N-1}\psi),
\]
where $D_\frakn$ is the normal covariant derivative, (which is well-defined in a collar neighborhood of $\dM$, hence can be iterated) evaluated at the boundary, and $D_\frakn^0$ is the trace on the boundary; the choice of connection on $\E$ is immaterial. Given $A\in\OPA$, there exists a unique matrix of tangential differential operators $U_A=(U_\alpha^\beta)_{\alpha,\beta=0,...,m-1}$ of orders $\le m-\alpha-1$ such that the following \emph{Green's formula} holds \cite[pp.~37--38]{Gru96},
\beq
\bra A_+\psi,\eta\ket =\bra \psi, A^*_+\eta\ket+\bra U_A\rho_m\psi,\rho_m\eta\ket
\label{eq:integration_by_parts_rudimentary}
\eeq
for every $\psi\in \Gamma(\E)$ and $ \eta\in \Gamma(\bbF)$.

In the sequel, we will encounter integration by parts formulas such as \eqref{eq:integration_by_parts_rudimentary} where the operator $A$ belongs to a class of operators larger than $\OPA$, which requires the expansion of the class of differential boundary operators. 
A \emph{trace operator} $T$ of \emph{order} $m \in\R$ and \emph{class} $r\in\Nzero$ is a linear map $T:\Gamma(\E)\to\Gamma(\bbG)$ of the form
\beq
T = \sum_{j=0}^{r-1} S_jD_\frakn^j + \jmath^*Q_+,
\label{eq:TraceOp}
\eeq 
where $S_j\in L^{m-j}_{\mathrm{cl}}(\dM,\jmath^*\E,\bbG)$ is a pseudodifferential operator on the boundary (which is a closed manifold) and $Q\in\OPAm$ \cite[pp.~27--28, 33]{Gru96}. The operator $\rho_m$ is an instance of a trace operator of order $m-1$ and class $m$ with $\bbG = (\jmath^*\E)^m$, $S_j(\xi) = (0,\dots,0,\xi,0,\dots,0)$ and $Q=0$.

The order of a trace operator is an extension of the order of a pseudodifferential operator (by \eqref{eq:compose_AQ}, $\ord(S_jD_\frakn^j) \le m$ for every $j$), whereas its class retains (one more than) the number of normal derivatives.  
We denote the set of trace operators of order $m\in \bbR$ and class $r\in\Nzero$ by $\OPTmr$, as in \cite{RS82}. 
The class of trace operators can be extended to negative values \cite[pp.~309--311]{Gru90}. In simple terms, $T\in \OP(\frakT^{m,-r})$ if $T\in \OP(\frakT^{m,0})$ and $T D_\frakn^r\in \OP(\frakT^{m,0})$.
We denote the union of all $\OPTmr$ of order $m\in\bbR$ and class $r\in\bbZ$ by $\OPT$. Trace operators have well-defined symbols, much like pseudodifferential operators. 
However, unlike operators with the transmission property, the mapping properties of trace operators depend on both the order and the class; for every $T\in\OPTmr$ \cite[p.~312]{Gru90}, 
\beq
T:W^{s,p}\Gamma(\E)\to W^{s-m-1/p,p}\Gamma(\bbG)
\label{eq:Sobolev_trace}
\eeq   
for every $\bbZ\ni s>r+1/p-1$ and $1<p<\infty$.
Note how the class $r$ limits the mapping properties: a trace operator of order $m$ reduces the regularity of a $W^{s,p}$-section by $m+1/p$, as expected, but only for $s$ large enough. On the other hand, a negative class allows mapping between negative Sobolev spaces.

We next specify a particular class of trace operators: A \emph{system of trace operators associated with order}  $m$ is a trace operator of the form $T=T_0\oplus T_1\oplus\cdots\oplus T_{m-1}\in\OP(\frakT^{m-1,m})$, where $T_i:\Gamma(\E)\to\Gamma(\bbJ_i)$ is in $\OP(\frakT^{i,m})$ and $\bbJ_i\to \dM$ is a vector bundle \cite[pp.~45--46]{Gru96}. As stated above, every component $T_i$ can be written as
\beq
T_i = \sum_{j=0}^{m-1} S_{ij}D_\frakn^j + \jmath^*(Q_i)_+,
\label{eq:components_normal}
\eeq
where $S_{ij}\in L^{i-j}_{\mathrm{cl}}(\dM,\jmath^*\E,\bbJ_i)$ and $Q_i\in\OPAi$. 
Since $T_i\in \OP(\frakT^{i,m})$, and since the mapping property \eqref{eq:Sobolev_trace} applies to each $T_i$ separately,
systems of trace operators associated with order $m$ satisfy the compound mapping property
\[
T:W^{s,p}\Gamma(\E)\to \bigoplus_{i=0}^{m-1} W^{s-i-1/p,p}\Gamma(\bbJ_i) 
\]
for every $\bbZ\ni s> m+1/p-1$. 

%%%%%%%%%%%%%%%%%%%%%%%% 
\begin{definition}
A system of trace operators $T_0\oplus T_1\oplus\cdots\oplus T_{m-1}$ associated with order $m$ is said to be \emph{normal} if each $T_i$ of the form \eqref{eq:components_normal} satisfies that $S_{ii}:\Gamma(\jmath^*\E)\to \Gamma(\bbJ_i)$ is surjective.
\end{definition}
%%%%%%%%%%%%%%%%%%%%%%

The normality of a system of trace operators implies surjectivity \cite[p.~80]{Gru96}:

%%%%%%%%%%%%%%%%%%%%%
\begin{proposition}
Let $T\in\OP(\frakT^{m-1,m})$ be a normal system of trace operators associated with order $m$. Then $T:\Gamma(\E)\to \Gamma(\bbG)$ and $T:W^{m,2}\Gamma(\E)\to \bigoplus_{i=0}^{m-1} W^{m-i-1/2,2}\Gamma(\bbJ_i)$ are surjective.
\end{proposition}
%%%%%%%%%%%%%%%%%%%%%

Consider the integration by parts formula \eqref{eq:integration_by_parts_rudimentary} for $A\in\OPAm$. 
We are interested in a setting where there exist \emph{differential} operators, $B_A:\Gamma(\E)\to\Gamma(\bbG)$ and  $B_{A^*}:\Gamma(\bbF)\to\Gamma(\bbG)$, which are normal systems of trace operators associated with order $m$, such that 
\beq
\bra A_+\psi,\eta\ket = \bra\psi,A_+^*\eta\ket + \bra B_A\psi, B_{A^*}\eta\ket
\label{eq:integration_by_parts_convention}
\eeq 
for every $\psi\in \Gamma(\E)$ and $\eta\in \Gamma(\bbF)$.
Not every $A\in\OPA$ has this property.
In our work, however, formulas such as \eqref{eq:integration_by_parts_convention} emerge naturally, hence we omit this discussion. 

%%%%%%%%%%%%%%%%%%%%%%%%%%%%%%%%%%%%%%%%%%%%%%%%%%
\subsection{Green operators and elliptic boundary-value problems} 

Let $E\in \OPAm$ be elliptic and let $T\in\OP(\frakT^{m-1,r})$. Consider the problem of finding a linear map $R:\Gamma(\bbF)\to \Gamma(\E)$ satisfying 
\beq
\begin{cases}
E_+R=\id  & \text{in $M$} \\
T R=0 & \text{on $\dM$}, 
\end{cases}
\label{eq:psdbvp1}
\eeq
i.e., a solution operator for a \emph{pseudodifferential boundary-value problem}
\[
\begin{cases}
E_+ \psi= \rho  & \text{in $M$} \\
T \psi= 0 & \text{on $\dM$},
\end{cases}
\] 
with $\rho\in\Gamma(\bbF)$. Since $E\in\OPAm$ is elliptic, it has a  parametrix $P\in\OP(\frakA^{-m})$. However, its truncation $P_+$ is generally not useful for finding $R$ for two reasons: first, the boundary condition $T R=0$ has to be taken into account; second, for general $A,Q\in\OPA$,
\beq
(AQ)_+-A_+Q_+\ne0,
\label{eq:truncation_commutation}
\eeq
and in particular, $E_+P_+\ne \id $ (modulo a smoothing operator). In fact, the expression \eqref{eq:truncation_commutation} is not necessarily the truncation of an element in $\OPA$, hence is not subject to the theory surveyed in \secref{sec:integration_by_parts}.

This motivates the introduction of an even larger class of operators, which allows among other things to classify operators such as $A_+Q_+$, and solution operators for \eqref{eq:psdbvp1}. This new class retains some of the desirable properties of pseudodifferential operators. The idea, originating in work by Boutet de Monvel \cite{Bou71}, is to construct a class of operators representing boundary-value problems, which is closed under its own algebra (the so-called \emph{Boutet de Monvel algebra}).

A \emph{Green operator} of order $m\in\bbR$ and class $r\in\bbZ$ is a system of operators $\bA$, which can be written in matrix form as
\beq
\bA=\begin{pmatrix}
A_++G & K_1 \\
T & K_2
\end{pmatrix}: 
\mymat{\Gamma(\E) \\ \Gamma(\bbJ)}
\longrightarrow  
\mymat{\Gamma(\bbF)  \\ \Gamma(\bbG)}.
\label{eq:full_Green_operator}
\eeq
Here $A\in \OPAm$, $T\in \OP(\frakT^{m-1,r})$ and $K_2\in L^m_{\mathrm{cl}}(\dM,\bbJ,\bbG)$, which all belong to classes of operators that have already been introduced. The operator $K_1$ is known as a \emph{potential operator} (of order $m$); it maps boundary sections into interior sections. The operator $G$ is known as a \emph{singular Green operator}. 
Singular Green operators are non-pseudodifferential operators, which are associated with a principal symbol much like pseudodifferential operators \cite[pp.~30--32]{Gru96}. They can also be characterized as classical in the sense of possessing an asymptotic expansion of homogeneous terms. 
Just like trace operators, they possess both an order and a class.
They were introduced in order to obtain good composition rules, e.g., to rectify elements such as $A_+Q_+$ and possibly their approximate inverses. Specifically, if $A\in\OP(\frakA^{m_A})$ and $Q\in\OP(\frakA^{m_Q})$, then
$(AQ)_+ - A_+Q_+$ is a singular Green operator of order $m_A + m_Q - 1$ and class $m_Q$ \cite[p.~152]{RS82}.

The singular Green operator $G$ in \eqref{eq:full_Green_operator} is assumed to be of order $m-1$ and class $r\in\bbZ$,
in which case \cite[p~312]{Gru90},
\beq
G:W^{s,p}\Gamma(\E)\to W^{s-m,p}\Gamma(\bbF)
\label{eq:Sobolev_sgs}
\eeq
for every $\bbZ\ni s > r +1/p-1$ and $1<p<\infty$.
If $r=0$, since a singular Green operator is $L^p$-continuous, it has an adjoint of the same order \cite[p.~32]{Gru96}. For $r>0$, this is however not true, so we have to keep track of the class of singular Green operators as we compose them with other operators.  

Green operators of the form \eqref{eq:full_Green_operator} satisfy the following mapping properties:
If $\bA$ is of order $m$ and class $r$, then
\beq
\bA:
\mymat{W^{s,p}\Gamma(\E)  \\ W^{s,p}\Gamma(\bbJ)}
\longrightarrow  
\mymat{W^{s-m,p}\Gamma(\bbF) \\ W^{s-m+1-1/p,p}\Gamma(\bbG)}
\label{eq:Green_mapping_proprety}
\eeq
for every $\bbZ\ni s> r+1/p-1$. 
Green operators are associated with a pair of symbols,
\beq
\sigma(\bA) = \sigma_M(\bA)\oplus \sigma_\dM(\bA),
\label{eq:symbol_Green_operator} 
\eeq
where $\sigma_M(\bA)(x,\xi)=\sigma_A(x,\xi):\E_x\to \bbF_x$, which is defined for every $x\in M$ and $\xi\in T_x^*M$,  is the interior symbol of $A\in\OPA$, and $\sigma_\dM(\bA)(x,\xi')$, which is defined for every $x\in\dM$ and $\xi'\in T_x^*\dM$, is the \emph{boundary symbol} of $\bA$;  the latter is a continuous linear map
\beq
\begin{split}
&\sigma_\dM(\bA)(x,\xi'):\mymat{\scrS(\overbar{\R}_+;\bbC\otimes\E_x) \\ \bbC\otimes \bbJ_x}\longrightarrow \mymat{\scrS(\overbar{\R}_+;\bbC\otimes\bbF_x) \\ \bbC\otimes \bbG_x},
\end{split}
\label{eq:boundary_symbol}
\eeq 
where for a vector bundle $\bbU\to M$,  $\scrS(\overbar{\R}_+;\bbC\otimes\bbU_x)$ denotes the space of $\bbC\otimes \bbU_x$-valued Schwartz functions on the half line $\overbar{\R}_+=\BRK{s\in\R~:~s\geq 0}$. We shall elaborate below upon how one obtains the map \eqref{eq:boundary_symbol} from the Green operator $\calA$ when both $K_1$ and $K_2$ are zero, in which case its domain is just $\scrS(\overbar{\R}_+;\bbC\otimes\E_x)$. A general definition is found in \cite[pp.~23--34]{Gru96}.

Green operators form an algebra closed under composition, with their symbols satisfying a homomorphism property
\cite[p.~175]{RS82}:

%%%%%%%%%%%%%%%%%%%%%
\begin{theorem}
\label{thm:full_Green_compositon_rules}
Let $\bA,\bQ$ be Green operators of orders $m_A$, $m_Q$ and classes $r_A$, $r_Q$. Then $\bQ\bA$ is a Green operator of order $m_A+m_Q$ and class $\max(m_A+r_Q,r_A)$. The symbol of $\bQ\bA$ is given by
\[
\sigma(\bQ\bA)=\sigma(\bQ)\circ\sigma(\bA) = (\sigma_M(\bQ)\circ\sigma_M(\bA)) \oplus (\sigma_\dM(\bQ)\circ\sigma_\dM(\bA)). 
\]
Moreover, if $\calA$ is a Green operator of order $m$ and $\calQ$ is a Green operator of order $<m$, then, 
\beq
\sigma(\calA+\calQ)=\sigma(\calA).
\label{eq:lower_order_symbol}
\eeq
\end{theorem}
%%%%%%%%%%%%%%%%%%%%%

A Green operator $\calA$ is called elliptic when $\sigma(\calA)$ is invertible.
It should be noted that the notation $\sigma_M(\bA)\oplus \sigma_\dM(\bA)$ is formal; the invertibility of the symbol amount to the separate invertibility of each component. 
Generalizing elliptic pseudodifferential operators on a closed manifold, an elliptic Green operator $\bA$  benefits from the existence of a parametrix in the calculus, such that $\bA\calP-\id$ and $\calP\bA-\id$ are both Green operators of order $-\infty$. If $\bA$ is of order $m$ and class $r$, then its parametrix $\calP$ is of order $-m$ and class $r-m$ \cite[pp.~335--336]{Gru90}.

Property \eqref{eq:lower_order_symbol} of the symbol raises a problem when considering systems of trace operators  $T=T_0\oplus T_1\oplus\dots\oplus T_{m-1}$  associated with order $m$, since by the definition of the symbol, the only contribution to $\sigma(T)$ is that of $\sigma(T_{m-1})$. 
The notion of ellipticity can be extended to encapsulate operators decomposing into direct sums $\calE=\oplus_i \calE_i$ and $T=\oplus_i T_i$  of operators having different orders (such systems are known as Douglis-Nirenberg boundary-value problems). The inclusion of such systems within the elliptic theory is justified by an order-reduction argument, which will be elaborated below.

Consider the upper-left term, $E_++G$, of the Green operator.
We denote all operators of this form where $E\in\OPAm$ and $G$ is a singular Green operator of order $m-1$ and class $r$  by $\OPSm$. We further introduce the class,
\[
\OPS=\bigcup_{m\in\bbZ, r\geq 0}\OPSm.
\]
We set $\OPSi=\bigcap_{m\in\bbZ} \OPSm$ as the operator class of smoothing operators of class $r$ and set $\OPSii=\bigcup_{r\geq 0}\OPSi$ \cite[p.~171]{RS82}.  As stated in the last reference, mappings in these classes map distributive sections into smooth ones, and as such are always compact. When it comes to composition, since an $\OPS$ operator can be viewed as a Green operator with all other terms equal zero, \thmref{thm:full_Green_compositon_rules} implies:

%%%%%%%%%%%%%%%%%
\begin{proposition}[Composition rules]
\label{prop:composition_Green_operators} 
Let $\calE\in\OP(\frakS^{m_E,r_E})$, $\bQ\in\OP(\frakS^{m_Q,r_Q})$ and $T\in \OP(\frakT^{m_T,r_T})$. Then, the following composition rules hold:
\begin{enumerate}[itemsep=0pt,label=(\alph*)]
\item $\bQ\calE\in\OPS$ is of order $m_E+m_Q$ and class $\max{(m_E+r_Q,r_E)}$.
\item $T\calE\in\OPT$ is of order $m_E+m_T$ and class $\max{(m_E+r_T,r_E)}$.
\end{enumerate}
\end{proposition}
%%%%%%%%%%%%%%%%%%

Operators in $\OPS$ benefit from Sobolev mapping properties with respect to their order and class, as inherited from \eqref{eq:Green_mapping_proprety}. In particular, the mapping properties of operators in $\OPS$ are limited by their class. Most importantly, for $r>0$, elements in $\OP(\frakS^{m,r})$ are not $L^p$-continuous, hence do not admit adjoints. Since $\OPA$ operators are $L^p$-continuous, this failure is due to the singular Green part. In fact (see the sharpness comment in \cite[p.~312]{Gru90}):

%%%%%%%%%%%%%%
\begin{proposition}
\label{prop:L2_continuity}
For every $m\in\bbZ$,
$\calE\in\OP(\frakS^{m,r})$ is $L^p\rightarrow W^{-m,p}$ continuous if and only if $r\leq 0$. 
\end{proposition}
%%%%%%%%%%%%%%%%%%%%%%

%%%%%%%%%%%%%%%%%%%%%%
\begin{proof}
The $L^p$-continuity for $r\le 0$ follows from the mapping property \eqref{eq:Green_mapping_proprety}. 
In the other direction, suppose that $\calE=E_++G\in\OP(\frakS^{m,r})$ is $L^p$-continuous. Since $E_+$ is $L^p$-continuous, it follows that $G$ is $L^p$-continuous. 
By \cite[p.~306]{Gru90}, every singular Green operator of order $m$ and class $r$ can be written as 
\[
G=\sum_{j=0}^{r-1} K_jD^j_\frakn+G',
\]
where $G'$ is a singular Green operator of order $m$ and class $\leq 0$
and $K_j$ are potential operators of order $m-j$
(the precise definition of $K_j$ is immaterial here). 

Let $\psi$ be smooth, and let $\psi_n$ be a sequence of smooth, compactly-supported sections converging to $\psi$ in $L^p$.
Since $G$ is $L^p$-continuous, and since $D^j_\frakn\psi_n=0$ for every $j$ and $n$, it follows that
\[
 G'\psi_n = G\psi_n \to G\psi \qquad\text{in $W^{-m,p}$}.
\]
Since $G'$ has class zero, it follows from the mapping property \eqref{eq:Green_mapping_proprety} that
\[
G'\psi_n \to G'\psi \qquad\text{in $W^{-m,p}$},
\]
from which we conclude that $G = G'$, hence $r\leq 0$.
\end{proof}
%%%%%%%%%%%%%%%%%%%%%%%%%%%%

A very important case is $\calE\in\OP(\frakS^{0,0})$, in which case it is $L^p$-continuous and as such has an adjoint $\calE^*$ of the same order and class \cite[pp.~175--176]{RS82}. We introduce the notation,
\[
\frakG^0=\OP(\frakS^{0,0}).
\]
By \propref{prop:composition_Green_operators}, $\frakG^0$ is also closed under compositions. 

Consider now Green operators of the form
\[
\mymat{\calE & 0 \\ T & 0},
\]
with $\calE = E_+ + G$, which for typographical reasons we denote $(\calE,T)$.
As stated above, the symbol of $(\calE,T)$ decomposes into 
\[
\sigma(\calE,T) = \sigma_M(\calE,T) \oplus \sigma_\dM(\calE,T),
\]
where $\sigma_M(\calE,T)(x,\xi)=\sigma_E(x,\xi)$. More specifically, we consider the case where:
\begin{enumerate}[itemsep=0pt,label=(\alph*)]
\item The order and the class of $G$ are strictly less than $m-1$.
\item The order of the $Q_+$ component of $T$ in \eqref{eq:TraceOp} is strictly less than $m$.
\end{enumerate}
In this case, $G$ and $Q_+$ do not contribute to the boundary symbol $\sigma_\dM(\calE,T)$, which is only determined by the pseudodifferential operators $E$ and $S_j$:
Let $x\in\dM$,
and write $\xi\in T^*_xM$ in the form $\xi=\xi'+\xi_d\,dr$, where $\xi'\in T_x^*\dM$ and $dr$ is the unit covector normal to the boundary, so that $\xi_d\in\R$ is the normal component of $\xi$.
Consider the map
\[
\mymat{\sigma_E(x,\xi'+\xi_d\,dr) \\
\sigma_T(x,\xi'+\xi_d\,dr)}:\E_x\longrightarrow \mymat{\bbF_x \\ \bbG_x},
\]
where $\sigma_T(x,\xi'+\xi_ddr)$ is obtained from \eqref{eq:TraceOp} by \cite[p.~27]{Gru96}
\[
\sigma_T(x,\xi'+\xi_ddr) =
\sum_{0\leq j<m} \xi_d^j\, \sigma_{S_j}(x,\xi').
\]
(Since $S_j\in L^{m-j}_{\mathrm{cl}}(\dM,\jmath^*\E,\bbG)$, it follows that $\sigma_{S_j}(x,\xi'):\E_x\to\bbG_x$ for $x\in\dM$ and $\xi'\in T_x^*\dM$.)
If one considers $\xi_d\in\R$ as an independent variable, then this map can be extended to operate on complexified vector-valued functions, 
\[
F: \operatorname{Func}(\R;\bbC\otimes\E_x) \to \operatorname{Func}\brk{\R;\mymat{\bbC\otimes\bbF_x \\ \bbC\otimes\bbG_x}},
\]
given by
\[
F(\psi)(t) =
\mymat{\sigma_E(x,\xi'+t\,dr)\psi(t) \\
\sigma_T(x,\xi'+t\,dr)\psi(t)}. 
\]
We then perform, formally, a one-dimensional Fourier transform, replacing $t \mapsto \iota\partial_s$. 
This yields a differential map, $\hat{F}$, given by
\[
\hat{F}(\psi)(s) = 
\mymat{\sigma_E(x,\xi'+\imath\partial_s\,dr)\psi(s) \\
\sigma_T(x,\xi'+\imath\partial_s\,dr)\psi(s)}. 
\]
This map can be restricted to one-sided Schwartz functions, yielding a map
\[
\hat{F}  : \scrS(\overline{\R}_+;\bbC\otimes\E_x) \to \scrS\brk{\overline{\R}_+;\mymat{\bbC\otimes\bbF_x \\ \bbC\otimes\bbG_x}}.
\]
The boundary symbol of $(\calE,T)$ is the map
\[
\sigma_\dM(\calE,T)(x,\xi'):\scrS(\overbar{\R}_+;\bbC\otimes\E_x)\longrightarrow \mymat{\scrS(\overbar{\R}_+;\bbC\otimes\bbF_x)\\\bbC\otimes\bbG_x},
\]
given by
\[
\sigma_\dM(\calE,T)(x,\xi') \psi =
\mymat{\{s\mapsto \sigma_E(x,\xi'+\iota\partial_s\,dr)\psi(s)\} \\
\sigma_T(x,\xi'+\iota\partial_s\,dr)\psi(0)}. 
\]

The notion of ellipticity for Green operators $(\calE,T)$ reduces to two ingredients: 
the ellipticity of $E\in \OPA$ as a pseudodifferential operator over $\tM$, supplemented by the requirement that $\sigma_\dM(\calE,T)(x,\xi')$ be a bijection \cite[p.~34]{Gru96}. 
This condition generalizes the classical Lopatinskii-Shapiro condition, which is a sufficient condition for differential systems. 
Indeed, when $\calE$ and $T$ are differential operators, then the boundary symbol $\sigma_\dM(\calE,T)(x,\xi')$ 
can be viewed as the restriction of an ordinary differential operator,
\[
\sigma_E(x,\xi'+\iota\partial_s\,dr):C^{\infty}(\overbar{\R}_+;\bbC\otimes\E_x)\to C^{\infty}(\overbar{\R}_+;\bbC\otimes\bbF_x),
\]
to Schwartz functions,
supplemented by an initial condition map
\[
\Xi_{x,\xi'} = 
\sigma_T(x,\xi'+\iota\partial_s\,dr)|_{s=0}:C^{\infty}(\overbar{\R}_+;\bbC\otimes\E_x)\to \bbC\otimes\bbG_x. 
\]
These mappings coincide with the ones defined in \cite[pp.~233--234]{Hor03} in the statement of the classical Lopatinskii-Shapiro condition. The following proposition demonstrates how the invertibility of $\sigma_\dM(\calE,T)(x,\xi')$ extends this condition:

%%%%%%%%%%%%%%%%%%%%%%
\begin{proposition}
\label{prop:lopatnskii_shapiro}
Consider a system $(\calE,T)$, where $E$ and $T$ are differential operators.
Let $x\in\dM$. 
Suppose that $\sigma_E(x,\xi):\E_x\to \bbF_x$ is injective for every $\xi\in T^*_xM\setminus\{0\}$. 
For $\xi'\in T_x^*\dM$, let $\bbM_{x,\xi'}^+\subset C^{\infty}(\overbar{\R}_+;\bbC\otimes\E_x)$ denote the space of  decaying solutions of the linear $\bbC\otimes\E_x$-valued ordinary differential equation
\beq
\sigma_E(x,\xi'+\iota\partial_s\,dr)\psi(s)=0. 
\label{eq:ODE_lop}
\eeq
Then,
the boundary symbol $\sigma_\dM(\calE,T)(x,\xi')$ is injective if and only if the restriction of $\Xi_{x,\xi'}$ to $\bbM_{x,\xi'}^+$ is injective. 
If in addition $\dim\E_x=\dim\bbF_x$, then
the boundary symbol $\sigma_\dM(\calE,T)(x,\xi')$ is a bijection if and only if  $\dim_{\bbC}\bbM_{x,\xi'}^+=\dim\bbG_x$.  
\end{proposition}
%%%%%%%%%%%%%%%%%%%%%%

%%%%%%%%%%%%%%%%%%%%%%
\begin{proof}
We start by relating the Schwartz functions in the definition of $\sigma_\dM(\calE,T)(x,\xi')$  with the space $\bbM_{x,\xi'}^+$ in the Lopatinskii-Shapiro condition.
The solutions of the ordinary differential system \eqref{eq:ODE_lop} are spanned by elements of the form $\psi(s) = s^m\, e^{\lambda s}\, \psi_{m,\lambda}$. Hence, $\bbM_{x,\xi'}^+$ is spanned by all those elements for which the real part of $\lambda$ is negative, which implies that $\bbM_{x,\xi'}^+\subset \scrS(\overbar{\R}_+;\bbC\otimes\E_x)$.

Suppose that the boundary symbol $\sigma_\dM(\calE,T)(x,\xi')$ is injective, and let $\psi\in \bbM_{x,\xi'}^+$ be in the kernel of $\Xi_{x,\xi'}$. By definition $\psi\in \ker \sigma_E(\xi' + \imath \partial_s\, dr)$, hence 
$\psi\in \ker \sigma_\dM(\calE,T)(x,\xi')$, which implies that $\psi=0$, thus proving that the restriction of $\Xi_{x,\xi'}$ to $\bbM_{x,\xi'}^+$ is injective.

Conversely, suppose that the restriction of $\Xi_{x,\xi'}$ to $\bbM_{x,\xi'}^+$ is injective, and let $\psi$ be in the kernel of 
$\sigma_\dM(\calE,T)(x,\xi')$. Since $\psi$ is a Schwarz function solving \eqref{eq:ODE_lop}, it is in $\bbM_{x,\xi'}^+$, and since moreover $\Xi_{x,\xi'}\psi=0$, it follows that $\psi=0$, thus completing the proof of the first part. 
The second clause is the classical Lopatinskii-Shapiro condition.
\end{proof}
%%%%%%%%%%%%%%%%%%%%%%%

%%%%%%%%%%%%%%%%%%%%%%%%%%%%%%%%%%%%%%%%%%%%%%
\subsection{Overdetermined elliptic systems}

In the sequel we invoke a degenerate form of ellipticity of Green operators \cite[p.~237]{RS82}, \cite[p.~315]{Gru90}:

%%%%%%%%%%%%%%%%%%%%%%%
\begin{definition}
A Green operator $\bA$ is called \emph{overdetermined (OD) elliptic} if its symbol $\sigma(\bA)$ is injective. 
\end{definition}
%%%%%%%%%%%%%%%%%%%%

By \eqref{eq:symbol_Green_operator}, the OD ellipticity of $\bA$ amounts to the injectivity of the interior symbol $\sigma_A(x,\xi)$ for every $x\in M$ and $\xi\in T_x^*M\setminus\{0\}$, and the injectivity of $\sigma_\dM(\bA)(x,\xi')$ for every $x\in\dM$ and $\xi'\in T_x^*\dM\setminus\{0\}$. By \propref{prop:lopatnskii_shapiro}:

%%%%%%%%%%%%%%%%%%%%
\begin{corollary}
\label{corrr:OD_elliptic_differntial}
In the notations of \propref{prop:lopatnskii_shapiro}, a Green operator of the form $(\calE,T)$, where $E$ and $T$ are  differential operator, is OD elliptic if and only if $\sigma_E(x,\xi):\E_x\to\bbF_x$ and $\Xi_{x,\xi'}:\bbM^+_{x,\xi'}\to \bbC\otimes\bbG_x$ are injective. 
\end{corollary}
%%%%%%%%%%%%%%%%%%% 

OD elliptic Green operators of the form $(\calE,T)$ imply a priori Sobolev estimates and the existence of left-inverses \cite[pp.~335--336]{Gru90}:

%%%%%%%%%%%%%%%%%%%%%
\begin{proposition}
\label{prop:OD_tools}
Let $(\calE,T)$ be OD elliptic of order $m\in\bbZ$ and class $r\in\bbZ$. Then there exists an a priori estimate,
\beq
\|\psi\|_{s,p}\lesssim\|\calE\psi\|_{s-m,p} + \|T\psi\|_{s-m+1-1/1p,p}+ \|\psi\|_{0,p}
\label{eq:overdetermined_ellipticity_a_priori1}
\eeq 
for every $\bbZ\ni s > r +1-1/p$. 
In particular, $\ker(\calE,T)\subseteq W^{s,p}\Gamma(\E)$ is finite-dimensional, independent of $s,p $ and consists of smooth sections. If furthermore $(\calE,T)$ is injective, then it has a left-inverse of order $-m$ and class $r-m$ within the calculus of Green operators. 
\end{proposition}
%%%%%%%%%%%%%%%%%%%%%%

Note that OD ellipticity is defined as the injectivity of the symbol, which only implies the existence of an \emph{approximate} left-inverse, yielding the a priori estimate. The injectivity of the operator is a stronger requirement. 

Since $\ker(\calE,T)\subseteq L^2\Gamma(\bbE_k)$, it admits an $L^2$-orthogonal projection onto it, which we denote by $\bS:L^2\Gamma(\bbE)\to L^2\Gamma(\bbE)$. Since $\ker(\calE,T)$ consists solely of smooth functions, $\bS:\Gamma(\bbE)\to \Gamma(\bbE)$ continuously as an integral operator with a smooth integral kernel. As such, $\bS\in\OPSii$ \cite[p.~196]{RS82}, and by continuity, the projection $\bS:W^{s,p}\Gamma(\bbE)\to  W^{s,p}\Gamma(\bbE)$ is a compact operator with a finite-dimensional range $\ker(\calE,T)\subseteq W^{s,p}\Gamma(\bbE)$.  A standard procedure, using the finite-dimensionality of $\ker(\calE,T)$ and the Rellich compact embedding theorem \cite[p.~51]{Bre11} yields a refinement of the estimate \eqref{eq:overdetermined_ellipticity_a_priori1},
\[
\|\psi\|_{s,p} \lesssim 
\|\calE\psi\|_{s-m,p} + \|T\psi\|_{s-m+1-1/p,p} + 
\|\bS\psi\|_{0,p}
\]
for every $\bbZ\ni s > r +1- 1/p$. 

We next focus on
systems $\calE=\oplus_{i=1}^l\calE_i$ and $T=\oplus_{j=0}^{m-1}T_j$, where $\calE_i$ and $T_j$ are of varying orders. 
This extension uses the so-called ``simple reduction of order"  \cite[p.~208, pp.~234--235]{RS82}; we provide all the details for self-containment.

We start with the  boundary operators $T_j$: Since $\dM$ is a closed Riemannian manifold, there exists for every vector bundle $\bbS\to\dM$ and every $t\in\bbZ$ an invertible pseudodifferential operator $\calL_{\bbS}^t\in L^t_{\mathrm{cl}}(\dM,\bbS,\bbS)$, with inverse within the calculus (e.g., $(\id + \nabla^*\nabla)^{t/2}$, where $\nabla$ is any Riemannian connection on $\bbS$). Due to the mapping property \eqref{eq:mapping_without_boundary} This operator  extends to an isomorphism
\[
\calL_{\bbS}^t:W^{s,p}\Gamma(\bbS)\to W^{s-t,p}\Gamma(\bbS)
\]
for every $s\in\bbZ$ and $1<p<\infty$.
On a vector bundle $\bbU\to M$  over a compact manifold with boundary, the existence of such an isomorphism is not trivial. This fact is proved in \cite[Secs.~4,5]{Gru90}: for every $m>0$ there exists an $\OP(\frakS^{m,0})$ operator $\calL_\bbU^m$, which extends to an isomorphism  
$W^{s,p}\Gamma(\bbU)\to W^{s-m,p}\Gamma(\bbU)$ for every $s>1/p-1$. Its inverse is an $\OP(\frakS^{-m,0})$ operator.

With these noted, the following is an adaptation of the construction in \cite[pp.~234--235]{RS82} for elliptic systems of varying orders to OD systems of varying orders.  This is essentially what is done in \cite[pp.~331--338]{Gru90}, restated here to better suit the framework we develop later:

%%%%%%%%%%%%%%%%%%%%%%%%
\begin{definition}
\label{def:OD_varying_order}
Let $(\calE,T)$ be a Green operator, with $\bbF=\oplus_i^l\bbF_i$ and $\calE=\oplus_{i=1}^l\calE_i$, where $\calE_i:\Gamma(\E)\to\Gamma(\bbF_i)$ are $\OP(\frakS^{m_i,r_i})$, and $T=\oplus_{j=0}^qT_j$, where $T_j\in\OP(\frakT^{\gamma_j,\gamma_j+1})$. Set  $m=\max_i{m_i}$ and $r=\max_{i,j}\BRK{r_i,\gamma_j+1}$. The system $(\calE,T)$ is called \emph{OD elliptic of varying orders} (we later usually omit the suffix ``of varying orders") if:
\begin{enumerate}[itemsep=0pt,label=(\alph*)]
\item For every $x\in M$ and $\xi\in T_x^*M\setminus\{0\}$, the map 
\[
\bigoplus_{i=1}^l \sigma_{E_i}(x,\xi) : \E_x\to \bbF_x
\]
is injective.
\item For every $x\in\dM$ and $\xi'\in T_x^*\dM\setminus\{0\}$, the map 
\[
\mymat{
\bigoplus_{i=1}^l  \sigma_{E_i}(x,\xi'+ \imath\partial_s\, dr) \\ \bigoplus_{j=0}^q  \sigma_{T_j}(x,\xi'+ \imath\partial_s\, dr)}:
\scrS(\overbar{\R}_+;\bbC\otimes\E_x)\longrightarrow \mymat{\scrS(\overbar{\R}_+;\bbC\otimes\bbF_x)\\\bbC\otimes\bbG_x}
\]
is injective.
\end{enumerate} 
\end{definition}
%%%%%%%%%%%%%%%%%%%%%%%%%%%%

%%%%%%%%%%%%%%%%%%%%%%%%%%%%
\begin{proposition}
\label{prop:OD_varying_orders}
Let $(\calE,T)$ be OD elliptic of varying orders. Then there is an a priori estimate,
\beq
\|\psi\|_{s,p}\lesssim\sum_{i=1}^l\|\calE_i\psi\|_{s-m_i,p}+\sum_{j=0}^q\|T_j\psi\|_{s-\gamma_j-1/p,p} + \|\calS\psi\|_{0,p},
\label{eq:OD_a_priori}
\eeq
for every $\bbZ\ni s > r + 1/p-1$,  and $1<p<\infty$. Here, $\calS\in\OPSii$ is the $L^2$-orthogonal projection onto the finite-dimensional space $\ker(\calE,T)\subseteq W^{s,p}\Gamma(\E)$, which is independent of $s,p$ and consists of smooth sections.  If furthermore $(\calE,T)$ is injective, then it has a left-inverse of order $-m$ and class $r-m$ within the calculus of Green operators. Conversely, if $(\calE,T)$ has a left-inverse of order $-m$ and class $r-m$, then it is injective and OD elliptic of varying orders, with maximal order $m$ and maximal class $r$.
\end{proposition}
%%%%%%%%%%%%%%%%%%%%%%%%%%

%%%%%%%%%%%%%%%%%%%%%%%%%%
\begin{proof}
Using the order-reduction operators $\calL_{\bbF_i}^a$, $\calL_{\E}^a$ and $\calL_{\bbJ_i}^a$, consider the modified system:
\[
\begin{split}
&\tilde{\calE} = 
\bigoplus_{i=1}^l \calL^{m-m_i}_{\bbF_i}\calE_i \calL_{\E}^{-m} \\
&  \tilde{T}=
\bigoplus_{j=0}^q \calL^{m-\gamma_j-1}_{\bbJ_j}T_j  \calL_{\E}^{-m}. 
\end{split} 
\]
By the homomorphism property of the symbols of Green operators 
\[
\begin{split}
\sigma(\tilde{\calE},\tilde{T}) &= \mymat{\bigoplus_{i=1}^l \sigma(\calL^{m-m_i}_{\bbF_i}\calE_i \calL_{\E}^{-m},0) \\
\bigoplus_{j=0}^{m-1} \sigma(0,\calL^{m-j-1}_{\bbJ_j}T_j  \calL_{\E}^{-m})} \\
&= \mymat{\bigoplus_{i=1}^l \sigma(\calL^{m-m_i}_{\bbF_i}) \circ \sigma(\calE_i) \\
\bigoplus_{j=0}^{m-1} \sigma(\calL^{m-j-1}_{\bbJ_j}) \circ \sigma(T_j ) }\circ \sigma(\calL_{\E}^{-m}) .
\end{split}
\]
It is a straightforward, yet tedious calculation to show that the fact that symbols of the order-reduction operators are isomorphisms, and the fact that the direct sum of the symbols $\sigma(\calE_i,0)$ and $\sigma(0,T_j)$ is injective, implies that $\sigma(\tilde{\calE},\tilde{T})$ is injective. We conduct it only for the interior symbol, as the argument for the boundary symbol follows the same lines. 

Let $x\in M$, $\xi\in T_x^*M\setminus\{0\}$, and suppose that $\psi\in \ker\sigma_M(\tilde{\calE},\tilde{T})$, i.e.,
\[
\bigoplus_{i=1}^l \sigma_M(\calL^{m-m_i}_{\bbF_i})(x,\xi) \circ \sigma_M(\calE_i)(x,\xi) \circ \sigma_M(\calL_{\E}^{-m})(x,\xi) \psi = 0.
\]

By the assumptions, $\sigma_M(\calE_i) = \sigma_{E_i}$,
hence the above identity for $\psi$ reads 
\[
\bigoplus_{i=1}^l 
\sigma_M(\calL^{m-m_i}_{\bbF_i})(x,\xi) \circ \sigma_{E_i}(x,\xi) \circ \sigma_M(\calL^{-m}_{\E})(x,\xi) \psi=0. 
\]
This implies,
\[
\sigma_M(\calL^{m-m_i}_{\bbF_i})(x,\xi) \circ \sigma_{E_i}(x,\xi) \brk{\sigma_M(\calL^{-m}_{\E})(x,\xi) \psi} = 0
\qquad
\text{for every $i$}.
\]
Since $\sigma_M(\calL^{m-m_i}_{\bbF_i})(x,\xi)$ is bijective,
\[
\sigma_{E_i}(x,\xi) \brk{\sigma_M(\calL^{-m}_{\E})(x,\xi) \psi} = 0
\qquad
\text{for every $i$}.
\]
Since $\bigoplus_{i=1}^l \sigma_{E_i}(x,\xi)$ is injective, it follows that
\[
\sigma_M(\calL^{-m}_{\E})(x,\xi) \psi = 0,
\]
and since $\sigma_M(\calL^{-m}_{\E})(x,\xi)$ is bijective, $\psi=0$, thus proving that $\sigma_{\tilde{E}}(x,\xi)$ is injective.

The fact that the kernel of the problem is finite-dimensional and admits a left-inverse follows from \propref{prop:OD_tools}, by a simple composition with isomorphisms. We turn to the a priori estimate. Since $(\tilde{\calE},\tilde{T})$ is of order zero and class $r$,  the a priori estimate \eqref{eq:overdetermined_ellipticity_a_priori1} assumes the form 
\[
\|\psi\|_{s,p}\lesssim\sum_{i=1}^l \|\calL^{m-m_i}_{\bbF_i}\calE_i \calL_{\E}^{-m}\psi\|_{s,p}
+\sum_{j=0}^q\|\calL^{m-\gamma_j-1}_{\bbJ_j}T_j  \calL_{\E}^{-m}\psi\|_{s+1-1/p,p}+\|\tilde{\calS}\psi\|_{0,p}
\]
for every $\bbZ\ni s> r + 1/p-1$,
where $\tilde{\calS}\in\OPSii$ is the projection onto $\ker(\tilde{\calE},\tilde{T})$. 
By the isomorphism property of $\calL_{\E}^{-m}$,
\[
\begin{split}
\|\psi\|_{s,p} &\lesssim \|\calL_{\E}^m\psi\|_{s-m,p} \\
& \lesssim
\sum_{i=1}^l \|\calL^{m-m_i}_{\bbF_i}\calE_i\psi\|_{s-m,p}
+\sum_{j=0}^q\|\calL^{m-\gamma_j-1}_{\bbJ_j}T_j \psi\|_{s-m+1-1/p,p}+\|\tilde{\calS}\calL_{\E}^m \psi\|_{0,p} \\
&\lesssim
\sum_{i=1}^l \|\calE_i\psi\|_{s-m_i,p}
+\sum_{j=0}^q\|T_j \psi\|_{s-\gamma_j-1/p,p}+\|\calS\psi\|_{m,p}.
\end{split}
\]
The replacement of the last term by $\|\psi\|_{0,p}$ follows from the equivalence of norms on a finite-dimensional vector space. 

In the other direction, suppose that $(\calE,T)$ has a left-inverse of order $-m$ and class $r-m$. It is therefore injective. By using the order-reducing operators, we may assume that all the $\calE_i$ and $T_j$ are of the same order. By the homomorphism property of the symbols, $\sigma(\calE,T)$ is invertible, hence injective. The order and the classes are inferred by reverting the order-reduction. 
\end{proof}
%%%%%%%%%%%%%%%%%%%%%%%%%%%%%%%%%%%%%%%%%

We next combine \propref{prop:lopatnskii_shapiro} and \propref{prop:OD_varying_orders}: 

%%%%%%%%%%%%%%%%%%%%%%%%%%%%%%%%%%%%%%%%%
\begin{theorem}
\label{thm:adatped_OD_ellipticity_estiamte}
Let $(\calE,T) = (E,T)$ be a differential system of varying orders, where $E = \oplus_{i=1}^l E_i$ and $T = \oplus_{j=0}^{m-1} T_j$ . It is OD elliptic if 
\begin{enumerate}[itemsep=0pt,label=(\alph*)]
\item For every $x\in M$ and $\xi\in T_x^*M\setminus\{0\}$, the following map 
\[
\bigoplus_{i=1}^l \sigma_{E_i}(x,\xi):\E_x\to \bbF_x
\]
is injective.
\item For every $x\in\dM$ and $\xi'\in T_x^*\dM\setminus\{0\}$, the following map
\[
\Xi_{x,\xi'}= \bigoplus_{j=0}^{m-1} \sigma_{T_j}(x,\xi'+\iota\partial_s\,dr)|_{s=0}
\]
is injective when restricted to $\bbM_{x,\xi'}^+$, the space of decaying solutions for the ordinary-differential system,
\[
\bigoplus_{i=1}^l \sigma_{E_i}(x,\xi'+\iota\partial_s\,dr) \psi(s)=0, \qquad \psi\in C^{\infty}(\overbar{\R}_+;\bbC\otimes\E_x). 
\]
\end{enumerate}
\end{theorem}
%%%%%%%%%%%%%%%%%%%%%%%%%%%%%%%%%%%%%%%%%%

%%%%%%%%%%%%%%%%%%%%%%%%%%%%%%%%%%%%%%%%%%
\begin{proof}
We argue that Items (a),(b) imply that $(E,T)$ satisfies the requirements of \defref{def:OD_varying_order}, i.e., that
\[
\bigoplus_{i=1}^l \sigma_{E_i}(x,\xi) : \E_x\to \bbF_x
\]
and
\[
\mymat{
\bigoplus_{i=1}^l  \sigma_{E_i}(x,\xi'+ \imath\partial_s\, dr) \\ \bigoplus_{j=0}^{m-1}  \sigma_{T_j}(x,\xi'+ \imath\partial_s\, dr)|_{s=0}}:
\scrS(\overbar{\R}_+;\bbC\otimes\E_x)\longrightarrow \mymat{\scrS(\overbar{\R}_+;\bbC\otimes\bbF_x)\\\bbC\otimes\bbG_x}.
\]
are injective.
The first requirement is Item~(a). The second requirement follows from Item~(b), along with \corrref{corrr:OD_elliptic_differntial}, which extends to operators in the form of direct sums.  
\end{proof}
%%%%%%%%%%%%%%%%%%%%%%%%%%%%%%%%%%%%%%%%%

%%%%%%%%%%%%%%%%%%%%%%%%%%%%%%%%%%%%%%%%%%%%%%%%%%
\section{Elliptic pre-complexes}
\label{sec:elip_pre_comp}

%%%%%%%%%%%%%%%%%%%%%%%%%%%%%%%%%%%%%%%%%%%%%%%%%%
\subsection{Adapted Green operators and auxiliary decompositions} 
\label{sec:adapting}

In the sequel, we study the splitting of spaces of sections into ranges and kernels of $\OPS$ operators, reminiscent of a Fredholm alternative, which occurs in many elliptic boundary-value problems.
Such operators satisfy integration by parts formula with surjective boundary operators, due to the non-characteristic property satisfied by elliptic problems (\cite[p.~470]{Tay11a} and \cite[Sec.~1.4]{Gru96}). This motivates the following definition:

%%%%%%%%%%%%%%%%%
\begin{definition}
\label{def:adapting_operator}
We call a map $\bA:\Gamma(\E)\to\Gamma(\bbF)$ an \emph{adapted Green operator} of order $m\in\Nzero$ if:
\begin{enumerate}[itemsep=0pt,label=(\alph*)]
\item $\bA = A+G \in \OP(\frakS^{m,0})$, with $A$ a \emph{differential} operator and $G\in\frakG^0$.
\item The differential operator $A$ satisfies an integration by parts formula \eqref{eq:integration_by_parts_convention}, with differential boundary operators $B_A$ and $B_{A^*}$, both normal systems of trace operators associated with order $m$. 
\end{enumerate}
If $G=0$, we say that $\bA=A$ is an \emph{adapted differential operator}. 
\end{definition}
%%%%%%%%%%%%%%%%%%%%%%

We remark that if $\bA=A$ is a differential operator, it obviously satisfies Item~(a) in \defref{def:adapting_operator}, but not necessarily Item~(b). 

Due to the closure of both $\frakG^0$ and the class of differential operators to adjoints, and the symmetry between the requirements on $B_A$ and $B_{A^*}$, $\bA$ ia an adapted Green operator if and only if $\bA^*$ is an adapted Green operator. In particular, the mapping properties of Green operators in $\frakG^0$ implies that $G:L^2\Gamma(\E)\to L^2\Gamma(\bbF)$, and 
\[
\bra G\psi,\eta\ket =\bra \psi,G^*\eta\ket
\]
for every $\psi\in \Gamma(\E)$ and $\eta\in \Gamma(\bbF)$.
Thus, $\bA$ and its adjoint $\bA^*$ inherit the integration by parts formula \eqref{eq:integration_by_parts_convention},
\beq
\bra \bA\psi,\eta\ket = \bra\psi,\bA^*\eta\ket + \bra B_A\psi, B_{A^*}\eta\ket.
\label{eq:integration_by_parts_adapted}
\eeq

We introduce several definitions associated with adapted Green operators:
Let $\bA\in\OP(\frakS^{m,0})$ be an adapted Green operator.
The Banach dual of
$\bA:W^{m,q}\Gamma(\E)\to L^q\Gamma(\bbF)$ is the operator $\bA_p':L^p\Gamma(\bbF)\to (W^{m,q}\Gamma(\E))^*$ given by the pairing
\[
\bA_p'\eta(\psi) = \bra \eta, \bA\psi\ket 
\]
for every $\eta\in L^p\Gamma(\bbF)$ and $\psi\in W^{m,q}\Gamma(\E)$. 
The kernel of $\bA_p'$ is the closed subspace of $L^p\Gamma(\bbF)$ consisting of sections $\eta$ satisfying 
\[
\bra\eta, \bA\psi \ket = 0 
\]
for every $\psi\in W^{m,q}\Gamma(\E)$.
For $\eta\in\ker \bA_p' \cap W^{m,p}\Gamma(\bbF)$, it follows from the integration by parts formula \eqref{eq:integration_by_parts_adapted} that 
\[
\bra \bA^*\eta,\psi\ket+\bra B_A\psi, B_{A^*}\eta\ket=0
\]
for every $\psi\in W^{m,q}\Gamma(\E)$.
Taking $\psi$ compactly-supported, using the fact that $B_A$ is a differential operator,  hence $B_A\psi=0$, yields
\[
\bra \bA^*\eta,\psi\ket = 0
\]
for every $\psi\in W_0^{m,q}\Gamma(\E)$.
Since $W_0^{m,q}\Gamma(\E)$ is dense in $L^q\Gamma(\E)$, it follows that $\bA^*\eta=0$.
Using the surjectivity of $B_A$ to prescribe $B_A\psi$ arbitrary yields $B_{A^*}\eta=0$. Thus, $\ker \bA_p'\cap W^{m_A,p}\Gamma(\bbF)$ coincides with the intersection of the kernels of $\bA^*$ and $B_{A^*}$. This motivates the following notation for the kernel of $\bA_p'$, 
\[
\scrN^{0,p}(\bA^*,B_{A^*})=\ker \bA_p'. 
\]

In a similar fashion, consider the Sobolev space, 
\[
W^{s,p}_A\Gamma(\E)= W^{s,p}\Gamma(\E)\cap\ker B_A,
\] 
for $s\ge m$ and $1<p<\infty$.
By \eqref{eq:integration_by_parts_adapted}, 
\[
\bra \bA\psi,\eta\ket=\bra\psi,\bA^*\eta\ket
\]
for every $\psi\in W^{m,p}_A\Gamma(\E)$ and $\eta\in W^{m,q}\Gamma(\bbF)$,
where $1/p+1/q=1$. 
Consider the restriction of $\bA$ to $\ker B_A$,
\[
\bA|_{\ker B_A}: W^{m,q}_A\Gamma(\E) \to L^q\Gamma(\bbF).
\]
By definition, its Banach adjoint 
\[
\bA'_{p,A} :  L^p\Gamma(\bbF) \to (W^{m,q}_A\Gamma(\E))^*,
\]
is given by the pairing 
\[
\bA'_{p,A}\eta(\psi) = \bra \eta, \bA\psi\ket \qquad
\text{for every $\eta\in L^p\Gamma(\bbF)$ and $\psi\in W^{m,q}_A\Gamma(\E)$}, 
\]
where we used again the isomorphism $L^p\Gamma(\E)\simeq (L^q\Gamma(\bbF))^*$. 
The kernel of $\bA'_{p,A} $ is the closed subspace of $L^p\Gamma(\bbF)$ consisting of sections $\eta$ satisfying 
\[
\bra\eta, \bA\psi \ket = 0 
\]
for every $\psi\in W^{m,q}_A\Gamma(\E)$.
We denote this space by 
\beq
\scrN^{0,p}(\bA^*)=\ker \bA'_{p,A}. 
\label{eq:kerA*}
\eeq
Comparing with \eqref{eq:integration_by_parts_adapted},
$\scrN^{0,p}(\bA^*)\cap W^{m,p}\Gamma(\bbF)$ coincides with the classical kernel of $\bA^*: W^{m,p}\Gamma(\bbF)\to L^p\Gamma(\E)$. 

For the next definitions, we recall the mapping properties \eqref{eq:Sobolev_with_boundary} and \eqref{eq:Sobolev_sgs}. 

%%%%%%%%%%%%%%%% 
\begin{definition}
Let $s\in\Nzero$ and $1<p<\infty$, and let $\bA\in\OP(\frakS^{m,0})$ be an adapted Green operator $\Gamma(\E)\to\Gamma(\bbF)$. 
We define the following subspaces of $W^{s,p}\Gamma(\E)$ and $W^{s,p}\Gamma(\bbF)$,
\[
\begin{aligned}
&\scrR^{s,p}{(\bA)} = \bA({W^{s+m,p}\Gamma(\E)}) 
&\qquad
&\scrN^{s,p}(\bA) = W^{s,p}\Gamma(\E) \cap \scrN^{0,p}(\bA), \\
&\scrR^{s,p}{(\bA;B_A)} = \bA({W^{s+m,p}_A\Gamma(\E)}) 
&\qquad
&\scrN^{s,p}(\bA,B_A) = W^{s,p}\Gamma(\E) \cap \scrN^{0,p}(\bA,B_A),
\end{aligned}
\]
along with their smooth versions,
\[
\begin{aligned}
&\scrR{(\bA)} = \bA(\Gamma(\E)) 
&\qquad
&\scrN(\bA) = \Gamma(\E) \cap \ker \bA \\
&\scrR{(\bA;B_A)} = \bA(\Gamma(\E)\cap\ker B_A) 
&\qquad
&\scrN(\bA,B_A) = \Gamma(\E) \cap \ker(\bA \oplus B_A).
\end{aligned}
\]
\end{definition}
%%%%%%%%%%%%

The closed range theorem asserts that for a bounded linear map $T:V\to W$ between Banach spaces, $\ker T' = (T(V))^\bot$, where $\bot$ is the Banach annihilator functor \cite[p. 575]{Tay11a}. Applying this theorem to 
$\bA:W^{m,q}\Gamma(\E)\to L^q\Gamma(\bbF)$ and its Banach dual $\bA_p':L^p\Gamma(\bbF)\to (W^{m,q}\Gamma(\E))^*$ yields
\beq
(\scrR^{0,q}(\bA))^{\bot}=\scrN^{0,p}(\bA^*,B_{A^*}).
\label{eq:range_null_bot1}
\eeq
Similarly, applying to closed range theorem to
$\bA: W^{m,q}_A\Gamma(\E) \to L^q\Gamma(\bbF)$ and its Banach adjoint 
$\bA'_{p,A} :  L^p\Gamma(\bbF) \to (W^{m,q}_A\Gamma(\E))^*$ yields.
\beq
(\scrR^{0,q}(\bA;B_A))^{\bot}=\scrN^{0,p}(\bA^*)
\label{eq:range_null_bot2}
\eeq

Adapted operators were introduced for the purpose of the following definition:

%%%%%%%%%%%%%%%%%%%%%%%%%
\begin{definition}
\label{def:aux_decomposition}
Let $\bA\in\OP(\frakS^{m,0})$ be an adapted Green operator. We say that $\bA$ induces an \emph{auxiliary decomposition} if the following holds: 
\begin{enumerate}[itemsep=0pt,label=(\alph*)]
\item There exists an $L^2$-orthogonal, topologically-direct decomposition of Fréchet spaces, 
\beq
\Gamma(\bbF)=\scrR(\bA)\oplus \scrN(\bA^*,B_{A^*}).
\label{eq:aux_smooth}
\eeq
\item There exists an operator $\calP_{\bA} \in\OP(\frakS^{-m,0}):\Gamma(\bbF)\to\Gamma(\E)$, such that the operator $\bA\calP_{\bA}:\Gamma(\bbF)\to\Gamma(\bbF)$ is the projection onto $\scrR(\bA)$. 
\end{enumerate}
\end{definition}
%%%%%%%%%%%%%

Eq.~\eqref{eq:aux_smooth} can be viewed as a Fredholm alternative induced by the (generally non-elliptic) operator $\bA$, supplemented with a clause establishing its connection with the calculus of Green operators. In this setting, the composition rules in \propref{prop:composition_Green_operators} implies that $\bA\calP_{\bA}\in\frakG^0$, hence 
\[
\bA\calP_{\bA}:W^{s,p}\Gamma(\bbF)\to W^{s,p}\Gamma(\bbF)
\]
continuously for every $s\in\Nzero$ and $1<p<\infty$.
Since $\Gamma(\bbF)$ is dense in $W^{s,p}\Gamma(\bbF)$ and $\bA\calP_{\bA}$ is a projection, by an approximation/continuity argument we deduce the existence of a $W^{s,p}$-direct decomposition,
\beq
W^{s,p}\Gamma(\bbF)=\scrR^{s,p}(\bA)\oplus \scrN^{s,p}(\bA^*,B_{A^*}),
\label{eq:Wsp_aux_decomp}
\eeq
for every such $s,p$.  
Conversely, if \eqref{eq:Wsp_aux_decomp} holds for every $s,p$, then the smooth version \eqref{eq:aux_smooth} also holds.

The existence of an auxiliary decomposition implies in particular that $\scrR^{s,p}(\bA)$  and $\scrN^{s,p}(\bA^*,B_{A^*})$ are closed subspaces of $W^{s,p}\Gamma(\bbF)$. 
For $s=0$ and $p=2$, the closedness of the ranges suffices for the existence of an $L^2$-direct decomposition, but this is not true for general $s,p$:

%%%%%%%%%%%%%%%%%%%%%%%
\begin{proposition}
\label{prop:L2_aux_adapted_pair}
Let $\bA\in\OP(\frakS^{m,0})$ be an adapted Green operator. There exist $L^2$-orthogonal decompositions
\[
\begin{split}
&L^2\Gamma(\bbF) =  \overline{\scrR^{0,2}(\bA)} \oplus \scrN^{0,2}(\bA^*,B_{A^*}) \\
& L^2\Gamma(\bbF) =  \overline{\scrR^{0,2}(\bA;B_A)} \oplus \scrN^{0,2}(\bA^*).
\end{split}
\]
where the overline stands for the closure in the $L^2$-norm. 
\end{proposition}
%%%%%%%%%%%%%%%%%

%%%%%%%%%%%%%%%%%
\begin{proof}
In view of the isomorphism $L^2\Gamma(\bbF)\simeq (L^2\Gamma(\bbF))^*$, 
the Banach annihilator of a closed subspace coincides with its orthogonal complement. 
Thus, \eqref{eq:range_null_bot1} reads
\[
\scrN^{0,2}(\bA^*,B_{A^*}) = (\scrR^{0,2}(\bA))^\bot. 
\]
Since every closed subspace of a Hilbert space induces an orthogonal decomposition, we obtain
\[
L^2\Gamma(\bbF) = (\scrR^{0,2}(\bA))^\bot \oplus ((\scrR^{0,2}(\bA))^\bot)^\bot = \scrN^{0,2}(\bA^*,B_{A^*}) \oplus \overline{\scrR^{0,2}(\bA)}.
\]
The proof of the second clause follows similar lines, starting with \eqref{eq:range_null_bot2}.
\end{proof}
%%%%%%%%%%%%%%%%%%%%%

A decomposition of a Banach space is topologically-direct if it is algebraically direct and both subspaces are closed. 
In a Banach space, unlike in a Hilbert space, a closed subspace may fail to induce a direct decomposition. It induces a direct decomposition if and only if the closed subspace admits a continuous projection onto it. 
As will be seen below, auxiliary decompositions are a first step towards a more refined Hodge-like decomposition, whence the adjective \emph{auxiliary}.

%%%%%%%%%%%%%%%%%%%%%%%%%%%%%%%%%%%%%%%%%%%%%%%%%%
\subsection{Elliptic pre-complexes}
\label{sec:pre-complex}

We consider chains of operators between spaces of sections, as depicted below:
\[
\begin{xy}
(-20,0)*+{0}="Em1";
(-20,-20)*+{0}="Gm1";
(0,0)*+{\Gamma(\E_0)}="E0";
(30,0)*+{\Gamma(\E_1)}="E1";
(60,0)*+{\Gamma(\E_2)}="E2";
(90,0)*+{\Gamma(\E_3)}="E3";
(100,0)*+{\cdots}="E4";
(0,-20)*+{\Gamma(\bbG_0)}="G0";
(30,-20)*+{\Gamma(\bbG_1)}="G1";
(60,-20)*+{\Gamma(\bbG_2)}="G2";
(90,-20)*+{\Gamma(\bbG_3)}="G3";
(100,-20)*+{\cdots}="G4";
{\ar@{->}@/^{1pc}/^{A_0}"E0";"E1"};
{\ar@{->}@/^{1pc}/^{A_0^*}"E1";"E0"};
{\ar@{->}@/^{1pc}/^{A_1}"E1";"E2"};
{\ar@{->}@/^{1pc}/^{A_1^*}"E2";"E1"};
{\ar@{->}@/^{1pc}/^{A_2}"E2";"E3"};
{\ar@{->}@/^{1pc}/^{A_2^*}"E3";"E2"};
{\ar@{->}@/^{1pc}/^0"Em1";"E0"};
{\ar@{->}@/^{1pc}/^0"E0";"Em1"};
{\ar@{->}@/_{0pc}/^{B_0}"E0";"G0"};
{\ar@{->}@/_{0pc}/^{B_1}"E1";"G1"};
{\ar@{->}@/_{0pc}/^{B_2}"E2";"G2"};
{\ar@{->}@/^{1pc}/^{B_0^*}"E1";"G0"};
{\ar@{->}@/^{1pc}/^{B_1^*}"E2";"G1"};
{\ar@{->}@/^{1pc}/^0"E0";"Gm1"};
{\ar@{->}@/^{1pc}/^{B_2^*}"E3";"G2"};
{\ar@{->}@/_{0pc}/^{B_3}"E3";"G3"};
\end{xy}
\]
where $\PseudoComplex = (A_k)_{k\in\Nzero}$ is a sequence of adapted differential operators $A_k:\Gamma(\E_k)\to \Gamma(\E_{k+1})$ of orders $m_k$. We denote the corresponding normal systems of trace operators by $B_k:\Gamma(\E_k)\to \Gamma(\bbG_k)$ and $B_k^*:\Gamma(\E_{k+1})\to \Gamma(\bbG_k)$, 
namely,
\[
\bra A_k\psi,\eta\ket = \bra\psi,A_k^*\eta\ket + \bra B_k\psi, B_k^*\eta\ket
\]
for every $\psi\in \Gamma(\E_k)$ and $\eta\in \Gamma(\E_{k+1})$.

%%%%%%%%%%%%%%%%%%%%%%%%
\begin{definition}[Elliptic pre-complex]
A sequence $\PseudoComplex$ is called an \emph{elliptic pre-complex} if:
\begin{enumerate}[itemsep=0pt,label=(\alph*)]
 
\item $(A_{k-1}^* \oplus A_k,B_{k-1}^*)$ is OD elliptic, generally of varying orders. 

\item $\ord(A_k A_{k-1}) \le \ord(A_{k-1})$, i.e., the minimal order of $A_k A_{k-1}$ is at most $m_{k-1}$. 
\end{enumerate}
\end{definition}
%%%%%%%%%%%%%%%%%%%%%%%%%%%

There are several distinctions 
between this definition and the classical notion of an elliptic complex. 
Most prominently, elliptic complexes are algebraic complexes, i.e., $A_k A_{k-1}=0$. In many applications, however, $A_k A_{k-1}$ does not vanish, satisfying instead Condition (b). This order-reduction enables the ``correction" of $\PseudoComplex$  by lower-order terms into a complex, whence the terminology of a pre-complex.

Secondly, in classical elliptic complexes, ellipticity is usually a property of the ``Laplacian" system (\cite[pp.~460--465]{Tay11b} and \cite{SS19}),
\[
\mymat{A_{k-1}A_{k-1}^* + A_k^* A_k \\
B_{k-1}^*\oplus B_k^*A_k}.
\]
The notion of OD ellipticity of varying orders replaces this ellipticity, and enables the consideration of chains in which the operators are of variable order, verifiable by a relatively simple criterion (\thmref{thm:adatped_OD_ellipticity_estiamte}). 
Elliptic complexes of variable orders have been considered in the literature, in particular in closed manifolds where an order-reduction argument can be applied with relative ease \cite[p.~279-280]{RS82}. In manifolds with a boundary, however, the picture is more involved, and this is where \thmref{thm:adatped_OD_ellipticity_estiamte} provides a significant simplification. 

%%%%%%%%%%%%%%%%%%%%%%%%%%%%%%%%%%%%%%%%%%%%%%%%%%
\subsection{The induced elliptic complex}
\label{sec:main_theorem}

The next theorem is our main result concerning elliptic pre-complexes.  

%%%%%%%%%%%%%%%% 
\begin{theorem}[Induced elliptic complex]
\label{thm:corrected_complex}
Every elliptic pre-complex $\PseudoComplex$ induces a complex of adapted Green operators $\Complex$, with $\bA_k:\Gamma(\E_k)\to\Gamma(\E_{k+1})$, uniquely characterized by the following properties: 
\begin{enumerate}[itemsep=0pt,label=(\alph*)]
\item $\bA_{k+1}\bA_k=0$.
\item
$\bA_{k+1} = A_{k+1}$ on $\scrN(\bA_k^*,B_k^*)$.
\end{enumerate}
We call $\Complex$ the elliptic complex induced by $\PseudoComplex$. 
\end{theorem}
%%%%%%%%%%%%%%%%  

Note that by convention, $\scrN(\bA_{-1}^*,B_{-1}^*)=\Gamma(\bbE_0)$, which forces $\bA_0=A_0$. Moreover,

%%%%%%%%%%%%%%%%%%%%%
\begin{proposition}
\label{prop:complex_aux_decomp}
In the setting of \thmref{thm:corrected_complex}, each $\bA_k$ induces an auxiliary decomposition,
\beq
\Gamma(\E_{k+1})=\scrR(\bA_k)\oplus \scrN(\bA_k^*,B_k^*).
\label{eq:Neumann_aux_dec}
\eeq
We denote by $\calP_k = \calP_{\bA_k}: \Gamma(\E_{k+1})\to\Gamma(\E_k)$ the operator associated with this decomposition, i.e., $\calP_k\in\OP(\frakS^{-m_k,0})$ and $\bA_k\calP_k\in\frakG^0$ is the projection onto $\scrR(\bA_k)$.
\end{proposition}
%%%%%%%%%%%%%%%%%%%%%%

The following proposition
shows that the induced elliptic complex $\Complex$ is a ``correction" of the elliptic pre-complex $\PseudoComplex$  by lower-order terms, and gives an explicit formula for the difference:

%%%%%%%%%%%%%%%%
\begin{proposition}
\label{prop:correction}
In the setting of \thmref{thm:corrected_complex}, the induced elliptic complex $\Complex$ satisfies
\[
G_k= \bA_k-A_k\in\frakG^0 \qquad \text{for every $k$},
\] 
with $G_k$ given by the recursive formula,
\beq
\begin{split}
&G_0=0 \\
&G_k=-A_kA_{k-1}\calP_{k-1}-A_kG_{k-1}\calP_{k-1}.
\end{split}
\label{eq:recursive_correction}  
\eeq
\end{proposition}
%%%%%%%%%%%%%%%%%%%%

It follows from \propref{prop:correction} that $\bA_k$ and $\bA_k^*$ satisfy the same integration by parts formula as $A_k$ and $A_k^*$,
\beq
\bra \bA_k\psi,\eta\ket = \bra\psi,\bA_k^*\eta\ket + \bra B_k\psi, B_k^*\eta\ket.
\label{eq:integration_by_parts_corrected_k}
\eeq

\thmref{thm:corrected_complex},  \propref{prop:complex_aux_decomp} and \propref{prop:correction} are proved simultaneously by induction on $k$ in \secref{sec:main_proof},
Moreover, since $\bA_k-A_k\in\frakG^0$, the systems $(\bA_{k-1}^*\oplus\bA_k,B_{k-1}^*)$ are also OD elliptic. 

%%%%%%%%%%%%%%%%%%%%%%%%%%%%%%%%%%%%%%%%%%%%%%%%%%
\subsection{Applications of the induced elliptic complex}
\label{sec:applications_complex}

The defining properties of the induced elliptic complex $\Complex$ imply the following additional properties:

%%%%%%%%
\begin{lemma}
\label{lem:useful_neumann}
In the setting of \thmref{thm:corrected_complex} and every $s\in\Nzero$
\begin{enumerate}[itemsep=0pt,label=(\alph*)]
\item For $p\ge 2$, the subspaces $\scrN^{s,p}(\bA_k)$ and $\scrR^{s,p}(\bA_k^*;B_k^*)$ are $L^2$-orthogonal, hence intersect trivially.
\item $\scrR^{s,p}(\bA_{k-1}) \subset \scrN^{s,p}(\bA_k)$.
\item  For $p\ge 2$, the subspaces $\scrR^{s,p}(\bA_{k-1})$ and $\scrR^{s,p}(\bA_k^*;B_k^*)$ are $L^2$-orthogonal, hence intersect trivially.
%\item $\scrR^{s,p}(\bA_k^*;B_k^*) \subset \scrN^{s,p}(\bA_{k-1}^*,B_{k-1}^*)$.
\end{enumerate}
\end{lemma}
%%%%%%%

%%%%%%%
\begin{proof}
The spaces $\scrN^{0,p}(\bA_k)$ and $\scrR^{0,q}(\bA_k^*;B_k^*) \supset \scrR^{0,p}(\bA_k^*;B_k^*)$ are $L^2$-orthogonal by the very definition \eqref{eq:kerA*} of $\scrN^{0,p}(\bA_k)$, with $\bA = \bA_k^*$, hence Item~(a) holds for every $s\in\Nzero$.
For $s\ge m_{k-1}+ m_k$, the second item follows from the property $\bA_k\bA_{k-1}=0$; the extension to every $s\in\Nzero$ follows from an approximation argument. The third item is an immediate consequence of the first two items.
\end{proof}
%%%%%%

The auxiliary decomposition refines into a Hodge-like decomposition:   

%%%%%%%%%%%%%%%%%%%%%%%%%%%%
\begin{theorem}[Hodge-like decomposition]
\label{thm:hodge_like_corrected_complex}
In the setting of \thmref{thm:corrected_complex},  there exists a $W^{s,p}$-direct  decomposition,
\beq
W^{s,p}\Gamma(\E_k) = \scrR^{s,p}(\bA_{k-1})\oplus\scrR^{s,p}(\bA_k^*;B_k^*)\oplus \module^k\Complex
\label{eq:Hodgelikesmooth}
\eeq
for every $s\in\Nzero$ and $1<p<\infty$, where the subspace
\[
\module^k\Complex=\ker(\bA_k\oplus\bA_{k-1}^*\oplus B_{k-1}^*)
\]
is finite-dimensional, independent of $s$ and $p$, and consists of smooth sections. 
In particular, comparing with the auxiliary decomposition \eqref{eq:Neumann_aux_dec},
\beq
\scrN^{s,p}(\bA_{k-1}^*,B_{k-1}^*)=\scrR^{s,p}(\bA_k^*;B_k^*)\oplus \module^k\Complex.
\label{eq:aux_refinmenetSmooth}
\eeq
\end{theorem}
%%%%%%%%%%%%%%%%%%%%%%%%%%%%

The smooth version of this decomposition follows immediately:
\[
\Gamma(\E_k)=\scrR(\bA_{k-1})\oplus\scrR(\bA_k^*;B_k^*)\oplus\module^k\Complex.
\]

The proof of this theorem relies on some of the constructs developed in \secref{sec:main_proof}, hence we present it in that same section, after the construction of the induced elliptic complex.

The refinement of the auxiliary decomposition into a Hodge-like decomposition identifies $\module^k\Complex$ as the cohomology groups of the complex $\Complex$:

%%%%%%%%%%%%%%%%%%%%%%%
\begin{theorem}[Cohomology groups]
\label{thm:cohomology}
Let $\psi\in W^{s,p}\Gamma(\E_k)$, with $s\in\Nzero$ and $1<p<\infty$. Then, 
\[
\begin{gathered}
\psi\in \scrR^{s,p}(\bA_{k-1}) \\
\text{if and only if} \\
\psi\in \scrN^{s,p}(\bA_k) \textand \bra\psi,\zeta\ket=0 \qquad \text{for every }\zeta\in\module^k\Complex. 
\end{gathered}
\]
Equivalently,
\beq
\scrN^{s,p}(\bA_k)=\scrR^{s,p}(\bA_{k-1})\oplus\module^k\Complex,
\label{eq:cohomology_spaces}
\eeq
or in the smooth case,
\[
\scrN(\bA_k)=\scrR(\bA_{k-1})\oplus\module^k\Complex.
\]
\end{theorem}
%%%%%%%%%%%%%%%%%%%%%%%

%%%%%%%%%%%%%%%%%%%%%%%
\begin{proof}
First, note that since elements in $\module^k\Complex$ are smooth, the coupling $\bra \psi,\zeta\ket$ is well-defined for every $\psi\in \scrR^{s,p}(\bA_{k-1})$ and $\zeta\in \module^k\Complex$. 
Let $\psi=\bA_{k-1}\omega \in \scrR^{s,p}(\bA_{k-1})$. Then, $\psi\in\scrN^{s,p}(\bA_k)$ by \lemref{lem:useful_neumann}(b). Its $L^2$-orthogonality to $\module^k\Complex$ follows from the Hodge-like decomposition \eqref{eq:Hodgelikesmooth}.

In the other direction, let $\psi\in \scrN^{s,p}(\bA_k)$ be $L^2$-orthogonal to $\module^k\Complex$.
Decompose $\psi$ according to \eqref{eq:Hodgelikesmooth}. Its $\module^k\Complex$ component vanishes, 
whereas by \lemref{lem:useful_neumann}(a), its $\scrR^{s,p}(\bA_k^*;B_k^*)$ component vanishes as well, remaining with 
$\psi\in \scrR^{s,p}(\bA_{k-1})$ (for $p<2$, an additional approximation argument is needed, using the closedness of the subspaces).

To prove \eqref{eq:cohomology_spaces}, we note that the inclusion 
\[
\scrR^{s,p}(\bA_{k-1})\oplus\module^k\Complex \subseteq \scrN^{s,p}(\bA_k)
\]
is trivial. The reverse inclusion is an immediate corollary of the decomposition \eqref{eq:Hodgelikesmooth} and \lemref{lem:useful_neumann}(a) (see above comment for $p<2$).
\end{proof}
%%%%%%%%%%%%%%%%%%%

Combining Theorems~\ref{thm:hodge_like_corrected_complex} and \ref{thm:cohomology}, we obtain the following compound decompositions:
\[
W^{s,p}\Gamma(\E_k) = 
\lefteqn{\overbrace{\phantom{\scrR^{s,p}(\bA_{k-1})\oplus \module^k\Complex}}^{\scrN^{s,p}(\bA_k)}} \scrR^{s,p}(\bA_{k-1}) \oplus \underbrace{\module^k\Complex\oplus \scrR^{s,p}(\bA_k^*;B_k^*)}_{\scrN^{s,p}(\bA_{k-1}^*,B_{k-1}^*)}
\]

Generalizing the technique introduced in \cite{Sch95b}, Hodge-like decompositions bestow us with the ability to solve non-homogeneous, OD elliptic boundary-value problems:

%%%%%%%%%%%%%%%%%%%%%
\begin{theorem}[Overdetermined boundary-value problem]
Given an elliptic pre-complex $\PseudoComplex$, consider the list of data,
\[
\chi\in W^{s-m_k,p}\Gamma(\E_{k+1}) 
\qquad 
\xi\in W^{s-m_{k-1},p}(\E_{k-1})
\textand
\phi\in W^{s,p}\Gamma(\E_k),
\] 
where $s \ge m  = \max(m_k,m_{k-1})$.
There exists a solution $\psi\in W^{s,p}\Gamma(\E_k)$ to the boundary-value problem
\beq
(\bA_k\oplus \bA_{k-1}^* \oplus B_{k-1}^*) \psi = (\chi,\xi,B_{k-1}^*\phi),
\label{eq:boundary_value_problem_complex}
\eeq
if and only if the following integrability conditions are satisfied:
\begin{subequations}
\begin{gather}
\chi\in\scrN^{0,p}(\bA_{k+1}) \textand \bra\chi,\zeta\ket=0 \qquad \text{for every }\zeta\in\module^{k+1}\Complex
\label{eq:integrability_condition_1} \\
\xi - \bA_{k-1}^*\phi\in \scrN^{0,p}(\bA_{k-2}^*,B_{k-2}^*)
\label{eq:integrability_condition_2} \\
\bra \xi,\nu\ket =-\bra B_{k-1}^*\phi,B_{k-1}\nu\ket \qquad  \text{for every }\nu\in\module^{k-1}\Complex.
\label{eq:integrability_condition_3}
\end{gather}
\end{subequations}
The solution is unique up to an arbitrary $\lambda\in\module^k\Complex$. Moreover, there is an a priori estimate
\beq
\begin{split}
\|\psi\|_{s,p}&\lesssim \|\bA_k\psi\|_{s-m_k,p}+ \|\bA_{k-1}^*\psi\|_{s-m_{k-1},p} \\
&\quad+
\sum_{i=0}^{m_{k-1}-1} \|B_{i,k-1}^*\psi\|_{s-i-1/p,p} + 
\|\lambda_\psi\|_{0,p},
\end{split}
\label{eq:Bvp_estimate1}
\eeq
where $B_{i,k-1}^*$ are the components of the normal system of trace operators $B_{k-1}^*$, and 
$\lambda_\psi$ is the projection of $\psi$ onto $\module^k\Complex$.
\end{theorem}
%%%%%%%%%%%%%%%%%%%

%%%%%%%%%%%%%%%%%%%
\begin{proof}
We start by noting that for every $\nu\in\module^{k-1}\Complex$,
\[
\bra B_{k-1}^*\phi,B_{k-1}\nu\ket =  \bra \phi, \bA_{k-1}\nu\ket - \bra \bA_{k-1}^*\phi,\nu\ket = - \bra \bA_{k-1}^*\phi,\nu\ket ,
\]
hence the integrability condition \eqref{eq:integrability_condition_3} may take the alternative form
\beq
\bra \xi- \bA_{k-1}^*\phi,\nu\ket = 0 \qquad  \text{for every }\nu\in\module^{k-1}\Complex.
\tag{\ref{eq:integrability_condition_3}-2}
\label{eq:integrability_condition_3b}
\eeq

We first verify the necessity of the integrability conditions. 
Let $\psi\in W^{s,p}\Gamma(\E_k)$ be a solution to \eqref{eq:boundary_value_problem_complex}.
Since $\chi\in\scrR^{s-m_k,p}(\bA_k) \subset \scrR^{0,p}(\bA_k)$, then \eqref{eq:integrability_condition_1} follows from \thmref{thm:cohomology}.
Since $\psi - \phi\in\ker B_{k-1}^*$, then
\[
\xi - \bA_{k-1}^* \phi = \bA_{k-1}^*(\psi - \phi) \in \scrR^{s-m_{k-1},p}(\bA_{k-1}^*;B_{k-1}^*) \subset \scrR^{0,p}(\bA_{k-1}^*;B_{k-1}^*),
\]
which by \eqref{eq:aux_refinmenetSmooth} implies both \eqref{eq:integrability_condition_2} and  \eqref{eq:integrability_condition_3b}.

To prove sufficiency, it is enough to do it for smooth data, as the same claim in Sobolev regularity follows from the continuity of all the operators, along with an approximation argument.

The second integrability condition \eqref{eq:integrability_condition_2} asserts that 
\[
\xi - \bA_{k-1}^*\phi \in \scrN(\bA_{k-2}^*,B_{k-2}^*) = \scrR(\bA_{k-1}^*;B_{k-1}^*) \oplus \module^{k-1}\Complex,
\] 
where the equality follows from \eqref{eq:aux_refinmenetSmooth},
whereas the third integrability condition, in its form \eqref{eq:integrability_condition_3b} implies that $\xi - \bA_{k-1}^*\phi$ is $L^2$-orthogonal to $\module^{k-1}\Complex$, hence
\[
\xi - \bA_{k-1}^*\phi \in  \scrR(\bA_{k-1}^*;B_{k-1}^*).
\]
Let 
\[
\xi - \bA_{k-1}^*\phi =  \bA_{k-1}^*\alpha,
\qquad
\text{where $\alpha\in\ker B_{k-1}^*$}. 
\]
By \thmref{thm:cohomology}, the first integrability condition
\eqref{eq:integrability_condition_1} implies that $\chi\in\scrR(\bA_k)$.
By the Hodge-like decomposition for $\Gamma(\E_k)$, we may write
\[
\chi -  \bA_k(\alpha + \phi) = \bA_k \omega,
\qquad
\text{where $\omega\in \scrR(\bA_k^*; B_k^*) \subset \scrN(\bA_{k-1}^*,B_{k-1}^*)$},
\]
and the inclusion follows once again from \eqref{eq:aux_refinmenetSmooth}.
Let $\psi = \omega + \alpha + \phi$. Then,
\[
\begin{gathered}
\bA_k\psi = \bA_k \omega + \bA_k(\alpha + \phi) = \chi \\
\bA_{k-1}^*\psi = \bA_{k-1}^*\omega + \bA_{k-1}^*(\alpha + \phi)  = \xi \\
B_{k-1}^*\psi = B_{k-1}^*\omega + B_{k-1}^*(\alpha + \phi) = B_{k-1}^*\phi,
\end{gathered}
\]
i.e., $\psi$ is a solution to \eqref{eq:boundary_value_problem_complex}.
The uniqueness clause is immediate. 

Finally, the system $(\bA_{k-1}^*\oplus\bA_k,B_{k-1}^*)$ is OD elliptic, hence yields an a priori estimate  \eqref{eq:OD_a_priori}, which takes the form \eqref{eq:Bvp_estimate1}. 
\end{proof}
%%%%%%%%%%%%%%%%%%%

%%%%%%%%%%%%%%%%%%%%%%%%%%%%%%%%%%%%%%%%%%%%%%%%%%
\section{Construction of the induced elliptic complex}
\label{sec:main_proof}

In this section we prove jointly \thmref{thm:corrected_complex}, \propref{prop:complex_aux_decomp} and \propref{prop:correction} by induction on $k$. The proof is partitioned into five stages. In the last subsection we prove the Hodge-like decomposition (\thmref{thm:hodge_like_corrected_complex}).

%%%%%%%%%%%%%%%%%%%%%%%%%%%%%%%%%%%%%%%%%%%%%%%%%%
\subsection{Stage 1: Base and setup of induction step}
\label{sec:stage1}

For the base of the induction, it is convenient to append an extra level to the sequence $\PseudoComplex$, as in the diagram in \secref{sec:pre-complex}, setting $\E_{-1}= M \times \{0\}$,  $\bbG_{-1}= \dM \times \{0\}$ and $A_{-1}=\bA_{-1}=0$. The base of the induction requires that:
\begin{enumerate}[itemsep=0pt,label=(\alph*)]
\item $\bA_{-1}$ induces an auxiliary decomposition,
\[
\Gamma(\E_0) = \scrR(\bA_{-1})\oplus \scrN(\bA_{-1}^*,B_{-1}^*).
\]
\item $\bA_0\bA_{-1}=0$.
\item $\bA_0 = A_0$ on $\scrN(\bA_{-1}^*,B_{-1}^*)$.
\item $\bA_0 = A_0 + G_0$ is an adapted Green operators with $G_0$ satisfying the properties specified in \propref{prop:correction}.
\end{enumerate}
This is satisfied trivially by observing that $\scrR(\bA_{-1})=\{0\}$ and $\Gamma(\E_0)=\scrN(\bA_{-1}^*;B_{-1}^*)$, hence Conditions (a) and (b) are satisfied. Condition (c) determines $\bA_0 = A_0$ uniquely. 
Finally, Condition (d) is satisfied as $\bA_0$ has a Green part $G_0=0$. 

Induction step: we assume that $\bA_k$ and $\bA_{k-1}$ have been defined for some $k>0$, such that:

\begin{enumerate}[itemsep=0pt,label=(\alph*)]
\item $\bA_{k-1}$ induces an auxiliary decomposition, i.e., there is a map $\calP_{k-1}\in\OP(\frakS^{-m_{k-1},0})$, such that $\bA_{k-1}\calP_{k-1}\in\frakG^0$ is the projection onto $\scrR(\bA_{k-1})$ in the topologically-direct decomposition,
\beq
\Gamma(\E_k) = \scrR(\bA_{k-1})\oplus \scrN(\bA_{k-1}^*,B_{k-1}^*).
\label{eq:aux_induction}
\eeq
\item $\bA_k\bA_{k-1}=0$.
\item $\bA_k = A_k$ on $\scrN(\bA_{k-1}^*,B_{k-1}^*)$.
\item $\bA_k = A_k + G_k$ and $\bA_{k-1} = A_{k-1} + G_{k-1}$ are adapted Green operators with $G_k$ and $G_{k-1}$ in $\frakG^0$, satisfying the recursive formula
\[
G_k = -A_k A_{k-1}\calP_{k-1} - A_k G_{k-1}\calP_{k-1}.
\]
\end{enumerate}

%%%%%%%%%%%%%%%%%%%%%%%%%%%%%%%%%%%%%%%%%%%%%%%%%%%%
\subsection{Stage 2: Additional elliptic estimates}

Since the system $(A_{k-1}^*\oplus A_k, B_{k-1}^*)$ is OD elliptic and  $\bA_{k-1}^*-\bA_{k-1}^*,\bA_k-A_k\in\frakG^0$, it follows that the system $(\bA_{k-1}^*\oplus\bA_k, B_{k-1}^*)$ is also OD elliptic (we use here the closure of $\frakG^0$ to adjoints). By \thmref{thm:adatped_OD_ellipticity_estiamte}, this implies the finite-dimensionality of the space
\[
\module^k\Complex=\ker(\bA_{k-1}^*\oplus\bA_k\oplus B_{k-1}^*),
\]
along with the a priori estimate,
\beq
\begin{split}
\|\psi\|_{s,p} &\lesssim 
\|\bA_{k-1}^*\psi\|_{s-m_{k-1},p} + \|\bA_k\psi\|_{s-m_k,p} \\
&\quad+ \sum_{i=0}^{m_{k-1}-1} \|B^*_{i,k-1}\psi\|_{s-i-1/p,p} + 
\|\bS_k\psi\|_{0,p},
\end{split}
\label{eq:OD_elliptic_estimate_sharp}
\eeq 
for every $\Nzero\ni s\ge m$, where $\bS_k\in \OPSii$ is the $L^2$-orthogonal projection onto $\module^k\Complex$, and
$m=\max{(m_{k-1},m_k)}$. As a consequence of \eqref{eq:OD_elliptic_estimate_sharp}, the system $(\bA_{k-1}^*\oplus\bA_k\oplus\calS_k, B_{k-1}^*)$ is OD elliptic and injective, hence, by \propref{prop:OD_varying_orders}, admits a left-inverse or order $-m$ and class $m_{k-1}-m$ within the calculus of Green operators.

In the sequel, we need to estimate $\|\psi\|_{m_k,p}$ in terms of $\|\bA_k\psi\|_{0,p}$ when $\psi$ is restricted to 
the subspace $\ker(\bA_{k-1}^*\oplus B_{k-1}^*\oplus \calS_k)$. While this seems to follow from  \eqref{eq:OD_elliptic_estimate_sharp},this estimate cannot be used when $m_k<m_{k-1}$, since in this case the 
left-inverse has positive class. To overcome this difficulty, we derive an additional elliptic estimate, which uses the inductive assumption regarding the auxiliary decomposition induced by $\bA_{k-1}$:

%%%%%%%%%%%%%
\begin{proposition}
\label{prop:new_estimate1}
The following estimate holds,
\beq
\|\psi\|_{s,p} \lesssim 
\|\bA_{k-1}\calP_{k-1}\psi\|_{s,p} + 
\|\bA_k\psi\|_{s-m_k,p} + 
\|\bS_k\psi\|_{0,p},
\label{eq:Ak_OD_elliptic_estimate}
\eeq 
valid for every $\bbZ\ni s\geq 1/p-1$ and $1<p<\infty$. (Note the absence of a boundary term, which is embodied in the projection $\calP_{k-1}$.) 
\end{proposition}
%%%%%%%%%%%%%

%%%%%%%%%%%%%
\begin{proof}
By \propref{prop:OD_varying_orders}, 
if $(\calE,T)$ has a left-inverse of order $-m$ and class $r-m$ within the calculus of Green operators, then it is OD elliptic (of varying orders) and injective, with (maximal) order $m$ and class $r$.

The auxiliary decomposition \eqref{eq:aux_induction} induced by $\bA_{k-1}$ and the property of $\calP_k$ imply that $(\id-\bA_{k-1}\calP_{k-1})$ is the projection onto $\scrN(\bA_{k-1}^*,B_{k-1}^*)$, which implies that for every $\psi\in \Gamma(\bbE_k)$,
\[
\bA_{k-1}^*\bA_{k-1}\calP_{k-1}\psi=\bA_{k-1}^*\psi 
\Textand 
B_{k-1}^*\bA_{k-1}\calP_{k-1}\psi=B_{k-1}^*\psi.
\]
Hence,
\[
(\bA_{k-1}^*\oplus \bA_k\oplus \calS_k,B_{k-1}^*)=(\bA_{k-1}^*\bA_{k-1}\calP_{k-1}\oplus\bA_k\oplus \calS_k,B_{k-1}^*\bA_{k-1}\calP_{k-1}).
\]
As stated above,
the system on the left-hand side is injective and admits a left-inverse or order $-m$ and class $m_{k-1}-m$.
The system on the right-hand side can be rewritten in the form
\[
\mymat{\bA_{k-1}^*\bA_{k-1}\calP_{k-1}\oplus\bA_k\oplus \calS_k \\ B_{k-1}^*\bA_{k-1}\calP_{k-1}} = 
\mymat{\bA_{k-1}^*\oplus \id\oplus \id \\ B_{k-1}^*+ 0+  0} (\bA_{k-1}\calP_{k-1}\oplus\bA_k\oplus\calS_k),
\]
where we revert to the matrix notation for typographical reasons; here, $(A\oplus B)(C\oplus D) = AC \oplus BD$, and $(A+ B)(C\oplus D) = AC  + BD$.
Hence, the system $(\bA_{k-1}\calP_{k-1}\oplus\bA_k\oplus\calS_k)$ has a left-inverse, and is therefore OD elliptic. 
Since $\bA_{k-1}\calP_{k-1}\in\frakG^0$, it is
of order $m_k$ and class zero, yielding the elliptic estimate \eqref{eq:Ak_OD_elliptic_estimate}.
\end{proof}
%%%%%%%%%%%%%

We identify another OD elliptic system (of varying orders), which will be used to obtain yet another a priori estimate:

%%%%%%%%%%%%%%%%%%%%
\begin{proposition}
\label{prop:AstarA_OD_system}
The system $(\bA_{k-1}^* \oplus \bA_k^*\bA_k , B_{k-1}^* \oplus B_k^*\bA_k)$ is OD elliptic of order and class 
$\max(m_{k-1},2m_k)$. 
\end{proposition}
%%%%%%%%%%%%%%%%%%%

%%%%%%%%%%%%%%%%%%%
\begin{proof}
In the notation of \defref{def:OD_varying_order}, 
\[
\calE_1 = \bA_{k-1}^*
\qquad
\calE_2 = \bA_k^*\bA_k
\qquad
T_0 = B_{k-1}^*
\Textand
T_1 = B_k^*\bA_k.
\]
By the composition rules (\propref{prop:composition_Green_operators}), $\calE_1\in\OP(\frakS^{m_{k-1},0})$, $\calE_2\in\OP(\frakS^{2m_k,m_k})$, $T_0\in\OP(\frakT^{m_{k-1}-1,m_{k-1}})$ and $T_1\in\OP(\frakT^{2m_k-1,2m_k})$, thus satisfying the preamble to \defref{def:OD_varying_order}, with $m = r = \max(m_{k-1},2m_k)$. Moreover,  the leading $\OPA$ parts of $\bA_k^*\bA_k$ and $\bA_{k-1}^*$ are $A_k^*A_k$ and $A_{k-1}^*$ respectively, and the leading differential part of $B_k^*\bA_k$ is $B_k^*A_k$. 

It remains to show that Conditions (a) and (b) of \defref{def:OD_varying_order} are satisfied. We only verify Condition (a): given $x\in M$ and $\xi\in T_x^*M\setminus\{0\}$, we prove that the map
\[
\sigma_{A_{k-1}^*}(x,\xi) \oplus \sigma_{A_k^*A_k}(x,\xi) : (\E_k)_x \to (\E_{k-1})_x \oplus (\E_k)_x
\]
is injective. The verification of Condition (b) follows the same lines. 

Since the order of $A_{k+1}A_k$ is strictly less than the order of $A_k$, it follows from \eqref{eq:lower_order_symbol} that  
\beq
\sigma_{A_{k+1}}(x,\xi) \sigma_{A_k}(x,\xi)=0.
\label{eq:sigmasigma=0}
\eeq
Let $\psi\in (\bbE_k)_x$ satisfy
\[
\psi\in\ker(\sigma_{A_{k-1}^*}(x,\xi) \oplus \sigma_{A_k^*A_k}(x,\xi)),
\]
which implies, combining \eqref{eq:sigmasigma=0} and the homomorphism property of the symbols, that
\[
\psi\in\ker(\sigma_{A_{k+1}}(x,\xi)\sigma_{A_k}(x,\xi) \oplus \sigma_{A_k^*}(x,\xi)\sigma_{A_k}(x,\xi)),
\]
i.e.,
\[
\sigma_{A_k}(x,\xi) \psi\in\ker(\sigma_{A_{k+1}}(x,\xi) \oplus \sigma_{A_k^*}(x,\xi)).
\]
Since, by assumption, $\sigma_{A_{k+1}}(x,\xi) \oplus \sigma_{A_k^*}(x,\xi)$ is injective, it follows that $\sigma_{A_k}(x,\xi) \psi=0$,
hence,
\[
\psi\in\ker(\sigma_{A_{k-1}^*}(x,\xi) \oplus \sigma_{A_k}(x,\xi)).
\]
Since $\sigma_{A_{k-1}^*}(x,\xi) \oplus \sigma_{A_k}(x,\xi)$ is injective,  it follows that $\psi=0$, as required.
\end{proof}
%%%%%%%%%%%%%%%%%%%%%%%%%%%%%%%%%%%%

%%%%%%%%%%%%%%%%%%%%
\begin{proposition}
\label{prop:same_kernels}
The kernel of the OD elliptic system $(\bA_{k-1}^*\oplus \bA_k^*\bA_k ,B_{k-1}^*\oplus B_k^*\bA_k)$ coincides with the kernel $\module^k\Complex$ of the OD elliptic system $(\bA_{k-1}^*\oplus \bA_k, B_{k-1}^*)$. 
Moreover, the system  $(\bA_{k-1}^* \oplus \bA_k^*\bA_k \oplus \calS_k , B_{k-1}^* \oplus B_k^*\bA_k)$ is OD elliptic and injective, hence admits a left-inverse of order $-\max(m_{k-1},2m_k)$ and class $0$.
\end{proposition}
%%%%%%%%%%%%%%%%%%

%%%%%%%%%%%%%%%%%%
\begin{proof}
The inclusion 
\[
\module^k\Complex \subset \ker(\bA_{k-1}^*\oplus \bA_k^*\bA_k \oplus B_{k-1}^*\oplus B_k^*\bA_k)
\]
is immediate. In the reverse direction, let $\psi\in \ker(\bA_{k-1}^*\oplus \bA_k^*\bA_k \oplus B_{k-1}^*\oplus B_k^*\bA_k)$.
Applying
\eqref{eq:integration_by_parts_corrected_k},
\[
\bra \bA_k\psi,\bA_k\psi\ket   = \bra \bA_k^*\bA_k\psi,\psi\ket + \bra B_k^*\bA_k\psi,B_k\psi\ket = 0,
\]
which implies that $\psi\in\ker(\bA_k)$, hence $\psi\in \module^k\Complex$. The second clause is immediate once we established that $\calS_k$ is the projection onto $\ker(\bA_{k-1}^*\oplus \bA_k^*\bA_k \oplus B_{k-1}^*\oplus B_k^*\bA_k)$.
\end{proof}
%%%%%%%%%%%%%%%%%%%%%%%%%%%%%%%%%%%%%

Similarly to \propref{prop:new_estimate1}: 

%%%%%%%%%%%%%
\begin{proposition}
\label{prop:new_estimate2}
The following estimate holds,
\beq
\begin{split}
\|\psi\|_{s,p} &\lesssim 
\|\bA_{k-1}\calP_{k-1}\psi\|_{s,p} +
\|\bA_k^*\bA_k\psi\|_{s-2m_k,p} \\ 
&\quad +
 \sum_{j=0}^{m_k-1}\|B_{j,k}^*\bA_k\psi \|_{s-m_k-j-1/p,p}+ 
\|\bS_k\psi\|_{0,p},
\end{split}
\label{eq:AkstarAk_OD_elliptic_estimate}
\eeq 
valid for every $\bbZ\ni s\ge 2m_k$ and $1<p<\infty$.  
\end{proposition}
%%%%%%%%%%%%%

%%%%%%%%%%%%%
\begin{proof}
We write 
\[
\mymat{\bA_{k-1}^* \oplus \bA_k^*\bA_k \oplus \calS_k \\ B_{k-1}^* \oplus B_k^*\bA_k} = 
\mymat{\bA_{k-1}^* \oplus \id \oplus \id & 0 \\ B_{k-1}^*+ 0+ 0 & \id}
\mymat{\bA_{k-1}\calP_{k-1} \oplus \bA_k^* \bA_k \oplus \calS_k \\ B_k^*\bA_k}.
\]
Since the left-hand side has a left-inverse, it follows
that the system $(\bA_{k-1}\calP_{k-1}\oplus \bA_k^*\bA_k\oplus \calS_k, B_k^*\bA_k)$ has a left-inverse, hence is  OD elliptic. By the composition rules, it has order and class $2m_k$, yielding the elliptic estimate \eqref{eq:AkstarAk_OD_elliptic_estimate}.
\end{proof}
%%%%%%%%%%%%%

%%%%%%%%%%%%%%%%%%%%%%%%%%%%%%%%%%%%%%%%%%%%%%%%%%%%
\subsection{Stage 3: Closed range argument and a priori estimates for $\bA_k$}

The space $\module^k\Complex$ is a finite-dimensional subspace of $\scrN^{s,p}(\bA_{k-1}^*,B_{k-1}^*)$. Writing $\id=(\id-\bS_k)+\bS_k$ when restricted to $\scrN^{s,p}(\bA_{k-1}^*,B_{k-1}^*)$, yields for every $s\in\Nzero$ and $1<p<\infty$ the topologically-direct splitting,
\[
\scrN^{s,p}(\bA_{k-1}^*,B_{k-1}^*)=\scrN_\bot^{s,p}(\bA_{k-1}^*,B_{k-1}^*)\oplus \module^k\Complex,
\]
which in turn yields a direct decomposition of Fréchet  spaces,
\[
\scrN(\bA_{k-1}^*,B_{k-1}^*)=\scrN_\bot(\bA_{k-1}^*,B_{k-1}^*)\oplus \module^k\Complex,
\]
where $\scrN_\bot(\bA_{k-1}^*,B_{k-1}^*)$ is the complement of $\module^k\Complex$ in $\scrN(\bA_{k-1}^*,B_{k-1}^*)$ with respect to the $L^2$-orthogonal  projection $\bS_k$. 
Combined with the auxiliary decomposition \eqref{eq:aux_induction} induced by $\bA_{k-1}$, we obtain the topologically-direct decomposition,
\beq
\Gamma(\bbE_k)=\scrR(\bA_{k-1})\oplus \scrN_\bot(\bA_{k-1}^*,B_{k-1}^*)\oplus \module^k\Complex .
\label{eq:N0spdecomp}
\eeq
The projection onto $\scrN_\bot(\bA_{k-1}^*,B_{k-1}^*)$ is the map $\calP_\bot = (\id-\bS_k)(\id-\calA_{k-1}\bP_{k-1})$. By the composition rules,   $\calP_\bot\in\frakG^0$. Since all projections on the closed subspaces in \eqref{eq:N0spdecomp} are in $\frakG^0$,  it follows from a density/continuity argument that
\beq
W^{s,p}\Gamma(\bbE_k)=\scrR^{s,p}(\bA_{k-1})\oplus\scrN_\bot^{s,p}(\bA_{k-1}^*,B_{k-1}^*)\oplus\module^k\Complex. 
\label{eq:WspN0}
\eeq
By definition, for $s\in \Nzero$,
\[
\scrN(\bA_{k-1}^*,B_{k-1}^*)\hookrightarrow \scrN^{s,p}(\bA_{k-1}^*,B_{k-1}^*).
\]
Applying the projection $\id-\bS_k$ on both sides,
\[
\scrN_\bot(\bA_{k-1}^*,B_{k-1}^*)\hookrightarrow \scrN^{s,p}_\bot(\bA_{k-1}^*,B_{k-1}^*).
\]

%%%%%%%%%%%%%%%%%%%%%%%%%
\begin{lemma}
\label{lem:density_Nsp}
For all $1<p<\infty$ and $s\in\Nzero$, the continuous inclusion,
\[
\scrN_\bot(\bA_{k-1}^*,B_{k-1}^*)\hookrightarrow \scrN^{s,p}_\bot(\bA_{k-1}^*,B_{k-1}^*)
\]
is dense. 
\end{lemma}
%%%%%%%%%%%%%%%%%%%%%%%%%

%%%%%%%%%%%%%%%%%%%%%%%%%
\begin{proof}
Let $\mu\in \scrN_\bot^{s,p}(\bA_{k-1}^*,B_{k-1}^*)$ be given. 
Since $\Gamma(\E_k)$ is dense in $W^{s,p}\Gamma(\E_k)$, there exists a sequence $\psi_n\in\Gamma(\E_k)$, such that
\[
\psi_n \to \mu \qquad\text{in $W^{s,p}$}.
\]
Since $\calP_\bot \in\frakG^0$, it extends to a continuous map $\calP_\bot:W^{s,p}\Gamma(\bbE_k)\to W^{s,p}\Gamma(\bbE_k)$ for every $s$ and $p$. Hence,   
\[
\scrN_\bot(\bA_{k-1}^*,B_{k-1}^*) \ni \calP_\bot \psi_n \to \calP_\bot  \mu = \mu \qquad\text{in $W^{s,p}$},
\]
which completes the proof.
\end{proof}
%%%%%%%%%%%%%%%%%%%%%%%%%

%%%%%%%%%%%%%%%%%%%%%
\begin{proposition}
\label{prop:Ak_closed_range}
For every $s\in\Nzero$ and $1<p<\infty$, $\scrR^{s,p}(\bA_k)$ is a closed subspace of $W^{s,p}\Gamma(\bbE_{k+1})$. Moreover,
\beq
\|\psi\|_{s+m_k,p} \lesssim 
\|\bA_k\psi\|_{s,p},
\label{eq:estimateNsp}
\eeq
for every $\psi\in \scrN^{s+m_k,p}_\bot(\bA_{k-1}^*,B_{k-1}^*)$.
\end{proposition}
%%%%%%%%%%%%%%%%%%%%

%%%%%%%%%%%%%%%%%%%%
\begin{proof}
For $s\in\Nzero$,
\[
\scrR^{s,p}(\bA_k) = \bA_k(W^{s+m_k,p}\Gamma(\bbE_k))=\bA_k(\scrN_\bot^{s+m_k,p}(\bA_{k-1}^*,B_{k-1}^*)),
\]
where in the last passage we substituted the 
$W^{s+m_k,p}$ version of \eqref{eq:WspN0} and the fact that $\bA_k\bA_{k-1}=0$.
For $\psi\in \scrN_\bot^{s+m_k,p}(\bA_{k-1}^*,B_{k-1}^*)$, the elliptic estimate \eqref{eq:Ak_OD_elliptic_estimate} reduces to \eqref{eq:estimateNsp}. This in turn implies that $\bA_k(\scrN_\bot^{s+m_k,p}(\bA_{k-1}^*,B_{k-1}^*))$ is a closed subspace due to \cite[Prop.~6.7, p.~583]{Tay11a}. 
\end{proof}
%%%%%%%%%%%%%%%%%%%%%%%%%%%%%%%%%%

Since $\bA_k\in \OP(\frakS^{m_k,0})$, it operates continuously as $\bA_k:L^p\Gamma(\bbE_k)\to W^{-m_k,p}\Gamma(\bbE_{k+1})$. The following proposition addresses the case where
$\psi\in \scrN_{\bot}^{0,p}(\bA_{k-1}^*,B_{k-1}^*)$, however $\bA_k\psi$ is of higher regularity than guaranteed by this mapping property:

%%%%%%%%%%%%%%%%%%%%%%%%%%%%%%%
\begin{proposition}
\label{prop:aprioriAk}
Let $\psi\in\scrN^{0,p}_\bot(\bA_{k-1}^*,B_{k-1}^*)$ for $1<p<\infty$. For every $s\in\Nzero$,
\[
\bA_k\psi \in W^{s,p}\Gamma(\E_{k+1})
\qquad\text{implies}\qquad
\psi\in \scrN^{s+m_k,p}_\bot(\bA_{k-1}^*,B_{k-1}^*).
\]
\end{proposition}
%%%%%%%%%%%%%%%%%%%%%

%%%%%%%%%%%%%%%%%%%%%
\begin{proof}
Consider the operator $\nabla^s:\Gamma(\bbE_{k+1})\to\Gamma(\otimes_{i=1}^sT^*M\otimes\bbE_{k+1})$, which is a differential operator of order $s$, hence has an adjoint $\nabla^{s*}:\Gamma(\otimes_{i=1}^sT^*M\otimes\bbE_{k+1})\to \Gamma(\bbE_{k+1})$ which is also a differential operator of order $s$. 

By iterating the integration by parts formulas for $\nabla^s$ and $\bA_k$, using the $L^p$--$L^q$ duality we may extend the operation of $\nabla^s\bA_k$ to a continuous map,
\[
\nabla^s\bA_k:L^p\Gamma(\bbE_k)\to W^{-s-m_k,p}\Gamma(\otimes_{i=1}^sT^*M\otimes\bbE_{k+1}),
\] 
defined by
\[
\bra \nabla^s\bA_k\lambda,\eta\ket =\bra \lambda,\bA_k^*\nabla^{s*}\eta\ket \qquad \eta\in W^{s+m_k,q}_0\Gamma(\otimes_{i=1}^sT^*M\otimes\bbE_{k+1}).
\]
Since $M$ is compact, the $W^{s,p}$ norm is equivalent to the sum of the $L^p$ norm and the $L^p$ norm of the $s$-derivative,
\beq
\|\nabla^s\bA_k\psi\|_{0,p} + \|\bA_k\psi\|_{0,p}\lesssim\|\bA_k\psi\|_{s,p}\lesssim\|\nabla^s\bA_k\psi\|_{0,p}+\|\bA_k\psi\|_{0,p}.
\label{eq:eq_norms}
\eeq

Let $\psi\in\scrN^{0,p}_\bot(\bA_{k-1}^*,B_{k-1}^*)$ satisfy $\bA_k\psi\in W^{s,p}\Gamma(\bbE_{k+1})$. Using \lemref{lem:density_Nsp}, let $\psi_n\in \scrN_\bot(\bA_{k-1}^*,B_{k-1}^*)$ converge to $\psi$ in $L^p$. For all $\eta\in W^{s+m_k,q}_0\Gamma(\otimes_{i=1}^sT^*M\otimes\bbE_{k+1})$,  
\[
\bra \nabla^s\bA_k\psi_n,\eta\ket =\bra \psi_n,\bA_k^*\nabla^{s*}\eta\ket\to \bra \psi,\bA_k^*\nabla^{s*}\eta\ket = \bra \nabla^s\bA_k\psi,\eta\ket.
\]
By assumption, $\nabla^s\bA_k\psi\in L^p\Gamma(\otimes_{i=1}^sT^*M\otimes\bbE_{k+1})$, hence 
$\nabla^s\bA_k\psi_n$ weakly converges to $\nabla^s\bA_k\psi$ in $L^p$. It follows that $\nabla^s\bA_k\psi_n$ is $L^p$-bounded, and from \eqref{eq:eq_norms}, $\bA_k\psi_n$ is $W^{s,p}$-bounded. It follows from \eqref{eq:estimateNsp} that $\psi_n$ is a bounded sequence in $W^{s+m_k,p}$, hence has a weakly converging subsequence. By the uniqueness of the limit, $\psi\in W^{s+m_k,p}\Gamma(\E_k)\cap\scrN^{0,p}_\bot(\bA_{k-1}^*,B_{k-1}^*)$, i.e., $\psi\in \scrN^{s+m_k,p}_\bot(\bA_{k-1}^*,B_{k-1}^*)$. 
\end{proof}
%%%%%%%%%%%%%%%%%%%%%%%%%%%%%%%%

By invoking a slightly modified procedure for the OD elliptic system $(\bA_k^*\bA_k\oplus \bA_{k-1}\calP_{k-1}\oplus\calS_k, B_k^*\bA_k)$ and using the estimate \eqref{eq:AkstarAk_OD_elliptic_estimate},  \propref{prop:AstarA_OD_system} yields the following analog of \propref{prop:aprioriAk}:

%%%%%%%%%%%%%%%%%%%%%%%%%%%%%%%
\begin{proposition}
\label{prop:AprioriAstarAk} 
Let  $\psi\in\scrN^{0,p}_\bot(\bA_{k-1}^*,B_{k-1}^*)$, $1<p<\infty$, and let $s\in\Nzero$.
Suppose that  there exists an $\eta \in W^{s+m_k,p}\Gamma(\bbE_{k+1})$ such that $\bA_k\psi-\eta\in \scrN^{0,p}(\bA_k^*,B_k^*)$. Then $\psi\in \scrN^{s+2m_k,p}_\bot(\bA_{k-1}^*,B_{k-1}^*)$, and 
\beq
\|\psi\|_{s+2m_k,p} \lesssim 
\|\bA_k^*\bA_k\psi\|_{s,p} + 
\sum_{j=0}^{m_{k-1}}\|B_{j,k}^*\bA_k\psi\|_{s+m_k-j-1/p,p}.
\label{eq:estimateN0p2}
\eeq
\end{proposition}
%%%%%%%%%%%%%%%%%%%%%%%%%%%%%%%%%%%%

%%%%%%%%%%%%%%%%%%%%%%%%%%%%%%%%%%%%
\begin{proof}
First, it follows from \propref{prop:aprioriAk} that $\psi\in\scrN^{m_k,p}_\bot(\bA_{k-1}^*,B_{k-1}^*)$.
It suffices to show that if
\beq
\bra \bA_k\psi - \eta, \bA_k\lambda\ket = 0
\qquad
\text{for every $\lambda\in W^{m_k,q}\Gamma(\E_k)$},
\label{eq:condition_Ak*Ak}
\eeq
then for every $W^{m_k,p}$-approximating sequence $\psi_n\in\Gamma(\E_k)$ for $\psi$,
\[
\sup_n \|\bA_k^* \bA_k\psi_n\|_{s,p} < \infty
\Textand
\sup_n \max_j \|B_{j,k}^*\bA_k\psi_n\|_{s + m_k - j - 1/p,p} < \infty,
\]
as \eqref{eq:AkstarAk_OD_elliptic_estimate} and the density of $\scrN_\bot(\bA_{k-1}^*,B_{k-1}^*)$ in $\scrN^{m_k,p}_\bot(\bA_{k-1}^*,B_{k-1}^*)$ then imply that $\psi\in W^{s + 2m_k,p}\Gamma(\E_k)$ along with the estimate \eqref{eq:estimateN0p2}. The uniform boundedness of the boundary sections can be replaced by the equivalent
\[
\sup_n \|\calL_k  B_k^*\bA_k\psi_n\|_{0,p} < \infty,
\]
where
\[
\calL_k : \bigoplus_{j=0}^{m_k-1} W^{s + m_k - j - 1/p,p}\Gamma(\bbJ_{j,k}) \to L^p\Gamma(\bbG_k)
\]
is the isomorphism given by
\[
\calL_k = \bigoplus_{j=0}^{m_k-1} \calL_{\bbJ_{j,k}}^{s + m_k - j - 1/p},
\]
and $\calL_{\bbJ_{j,k}}^t$ are the boundary order-reduction operators.

Thus, assume that \eqref{eq:condition_Ak*Ak} holds and let $\psi_n$ be an $W^{m_k,p}$-approximating sequence for $\psi$.
Recall that $\bA_k^*:L^p\Gamma(\bbE_{k+1})\to W^{-m_k,p}\Gamma(\bbE_k)$ continuously as elements in $\OP(\frakS^{m_k,0})$ by
\[
\bra \bA_k^*\omega, \lambda \ket = \bra \omega, \bA_k \lambda \ket
\qquad
\text{for every $\lambda\in W_0^{m_k,q}\Gamma(\E_k)$}.
\]
Comparing with \eqref{eq:condition_Ak*Ak}, 
\[
\bra \bA_k^*\bA_k\psi - \bA_k^* \eta, \lambda\ket = 0
\qquad
\text{for every $\lambda\in W_0^{m_k,q}\Gamma(\E_k)$},
\]
which implies that
\beq
\bA_k^*\bA_k\psi = \bA_k^* \eta \in W^{s,p}\Gamma(\E_k).
\label{eq:A*Apsi=A*eta}
\eeq

As in the proof of \propref{prop:aprioriAk}, the continuous map
\[
\nabla^s\bA_k^*\bA_k:W^{m_k,p}\Gamma(\bbE_k)\to W^{-m_k-s,p}\Gamma(\otimes_{i=1}^sT^*M\otimes\bbE_k)
\] 
is defined by 
\[
\bra \nabla^s\bA_k^*\bA_k\psi,\eta\ket =\bra \bA_k\psi,\bA_k\nabla^{s*}\eta\ket \qquad \forall\eta\in W^{s+m_k,q}_0\Gamma(\otimes_{i=1}^sT^*M\otimes\bbE_k).
\]
By the same argument as in \propref{prop:aprioriAk}, $\nabla^s\bA_k^*\bA_k\psi_n$ weakly converges in $L^p$ to $\nabla^s\bA_k^*\bA_k\psi$. Thus, $\nabla^s\bA_k^*\bA_k\psi_n$ is $L^p$-bounded, which by an equivalence analogous to \eqref{eq:eq_norms} translates into the $W^{s,p}$-boundedness of $\bA_k^*\bA_k\psi_n$. 

Next, for every $\lambda\in \Gamma(\bbG_k)$, since $B_k$ is surjective, there exists a $\xi\in\Gamma(\E_k)$, such that
\[
B_k \xi = (\calL_k)^* \lambda,
\]
where $(\calL_k)^*$ is the adjoint of the pseudodifferential operator $\calL_k\in L_{\mathrm{cl}}(\dM;\bbG_k,\bbG_k)$. Then,
\[
\begin{split}
\bra \calL_k B_k^*\bA_k\psi_n,\lambda\ket &=
\bra B_k^*\bA_k\psi_n, (\calL_k)^* \lambda\ket 
= \bra \bA_k\psi_n, \bA_k \xi\ket - \bra \bA_k^*\bA_k\psi_n, \xi\ket .
\end{split}
\]
Since $\bA_k\psi_n$ converges to $\bA_k\psi$ in $L^p$ and $ \bA_k^*\bA_k\psi_n$ weakly converges in $L^p$ to $\bA_k^*\bA_k\psi$, it follows that
\[
\begin{split}
\limn \bra \calL_k B_k^*\bA_k\psi_n,\lambda\ket &= \bra \bA_k\psi, \bA_k \xi\ket - \bra \bA_k^*\bA_k\psi, \xi\ket \\
&=  \bra \eta, \bA_k \xi\ket - \bra \bA_k^*\bA_k\psi, \xi\ket \\
&=  \bra \bA_k^* \eta,  \xi\ket + \bra B_k^* \eta, B_k \xi\ket - \bra \bA_k^*\bA_k\psi, \xi\ket \\
&=  \bra \calL_k B_k^* \eta,   \lambda \ket - \bra \bA_k^*(\bA_k\psi - \eta) , \xi\ket.
\end{split} 
\]
The second term on the right-hand side vanishes by \eqref{eq:A*Apsi=A*eta}, hence $\calL_k B_k^*\bA_k\psi_n$ converges weakly in $L^p$ to $\calL_k B_k^* \eta$, hence is uniformly $L^p$-bounded. This completes the proof.
\end{proof}
%%%%%%%%%%%%%%%%%%%%%%%%%%%%%%%%%%%%

%%%%%%%%%%%%%%%%%%%%%%%%%%%%%%%%
\subsection{Stage 4: Auxiliary decomposition induced by $\bA_k$}

We first establish Condition (a) in \defref{def:aux_decomposition} of an auxiliary decomposition: 

%%%%%%%%%%%%%%%%%%%%%%%%%%%%
\begin{proposition}
\label{prop:aux_first_step}
There exists a topologically-direct decomposition of Fr\'echet spaces, 
\beq
\Gamma(\bbE_{k+1})=\scrR(\bA_k)\oplus\scrN(\bA_k^*,B_k^*).
\label{eq:aux_decompstionAksmooth} 
\eeq
The decomposition is $L^2$-orthogonal, and the smooth projection onto $\scrR(\bA_k)$ continuously extends to the $L^2$-orthogonal projection onto $\scrR^{0,2}(\bA_k)$.
\end{proposition}
%%%%%%%%%%%%%%%%%%%%%%%%%%%

%%%%%%%%%%%%%%%%%%%%%%%%%%%
\begin{proof}
\propref{prop:Ak_closed_range} implies, in particular, that $\scrR^{0,2}(\bA_k)$ is a closed subspace of $L^2\Gamma(\bbE_{k+1})$. By \propref{prop:L2_aux_adapted_pair} 
\beq
L^2\Gamma(\E_{k+1}) = \scrR^{0,2}(\bA_k) \oplus \scrN^{0,2}(\bA_k^*,B_k^*)
\label{eq:L2_splitting_proof}
\eeq 
is an $L^2$-orthogonal decomposition.

Let $s\in\Nzero$.  On the one hand, $\scrR^{s+m_k,2}(\bA_k)$ 
is a closed subspace of $W^{s+m_k,2}\Gamma(\E_{k+1})$ by \propref{prop:Ak_closed_range}. On the other hand, $\scrN^{s+m_k,2}(\bA_k^*,B_k^*)$ is a closed subspace as the kernel of the  continuous linear operator $\bA_k^*\oplus B_k^*$. By \eqref{eq:L2_splitting_proof}, these two closed subspaces intersect trivially and are mutually $L^2$-orthogonal. Thus, in order to prove that,
\beq
W^{s+m_k,2}\Gamma(\bbE_{k+1})=\scrR^{s+m_k,2}(\bA_k)\oplus \scrN^{s+m_k,2}(\bA_k^*,B_k^*),
\label{eq:this_holds}
\eeq
it remains to prove that the sum $\scrR^{s+m_k,2}(\bA_k)+ \scrN^{s+m_k,2}(\bA_k^*,B_k^*)$ exhausts the whole of $W^{s+m_k,2}\Gamma(\bbE_{k+1})$. If \eqref{eq:this_holds} holds for every $s\in\Nzero$, then \eqref{eq:aux_decompstionAksmooth} holds. 

Let $\eta \in W^{s+m_k,2}\Gamma(\bbE_{k+1})$ be given. Decompose it as an element in $L^2\Gamma(\bbE_{k+1})$ according to \eqref{eq:L2_splitting_proof},
\[
\eta= \bA_k\psi+\mu,
\]
where $\bA_k\psi\in \scrR^{0,2}(\bA_k)$ and $\mu\in \scrN^{0,2}(\bA_k^*,B_k^*)$. Since $\bA_k\bA_{k-1}=0$, it follows from \eqref{eq:N0spdecomp} that
$\psi\in W^{m_k,2}\Gamma(\bbE_k)$ can be chosen in $\scrN^{m_k,2}_{\bot}(\bA_{k-1}^*,B_{k-1}^*)$.

Since $\eta-\bA_k\psi\in\scrN^{0,2}(\bA_k^*,B_k^*)$, \propref{prop:AprioriAstarAk} implies that $\psi\in \scrN^{s+2m_k,2}_{\bot}(\bA_{k-1}^*,B_{k-1}^*)$, hence $\bA_k\psi\in \scrR^{s+m_k,2}(\bA_k)$. Thus, $\mu\in \scrN^{0,2}(\bA_k^*;B_k^*)\cap W^{s+m_k,2}\Gamma(\bbE_k)=\scrN^{s+m_k,2}(\bA_k^*;B_k^*)$. This completes the proof. 
The $L^2$-continuity clause regarding the projection is apparent from the construction. 
\end{proof}
%%%%%%%%%%%%%%%%%%%%%%%%%%%

%%%%%%%%%%%%%%%%%%%%
\begin{theorem}
$\bA_k$ induces an auxiliary decomposition. 
\end{theorem}
%%%%%%%%%%%%%%%%%%%%%%

%%%%%%%%%%%%%%%%%%%%%%
\begin{proof} 
In view of \propref{prop:aux_first_step}, it remains to show that
there exists an operator $\calP_k \in \OP(\frakS^{-m_k,0}) : \Gamma(\E_{k+1}) \to \Gamma(\E_k)$, such that $\bA_k \calP_k : \Gamma(\E_{k+1}) \to \Gamma(\E_{k+1})$ is the projection onto $\scrR(\bA_k)$.

The decomposition \eqref{eq:aux_decompstionAksmooth} implies the existence of a projection  $\tilde{\bP}_k:\Gamma(\bbE_{k+1})\to \Gamma(\bbE_{k+1})$ onto $\scrR(\bA_k)$, which is continuous in the Fr\'echet  topology. By 
\propref{prop:aux_first_step}, it continuously extends into $\tilde{\calP}_k:L^2\Gamma(\bbE_{k+1})\to L^2\Gamma(\bbE_k)$. 
Since $\bA_k\bA_{k-1}=0$, and in view of \eqref{eq:N0spdecomp},
\[
\scrR^{s,p}(\bA_k)=\bA_k(\scrN_{\bot}^{s+m_k,p}(\bA_{k-1}^*,B_{k-1}^*)).
\]
The estimate \eqref{eq:estimateNsp} implies that the continuous map $\bA_k:\scrN_{\bot}^{s+m_k,p}(\bA_{k-1}^*,B_{k-1}^*)\to \scrR^{s,p}(\bA_k)$ is a bijection. By the open mapping theorem it is an isomorphism of Banach spaces \cite[p.~574]{Tay11a}. Since this is true for every $s\in\Nzero$ and $1<p<\infty$, this implies that  $\bA_k:\scrN_{\bot}(\bA_{k-1}^*,B_{k-1}^*)\to \scrR(\bA_k)$ is an isomorphism of Fréchet spaces. Let $(\bA_k)^{-1}:\scrR(\bA_k)\to \scrN_{\bot}(\bA_{k-1}^*,B_{k-1}^*)$ be the continuous inverse of this isomorphism, and define $\calP_k:\Gamma(\bbE_{k+1})\to \Gamma(\bbE_k)$ by
\[
\calP_k\psi=(\bA_k)^{-1}\tilde{\calP}_k\psi,
\]
which is a continuous map as the composition of continuous maps. Also $\bA_k\calP_k\psi=\tilde{\calP}_k\psi$, hence $\bA_k\calP_k$ is indeed the smooth projection onto $\scrR(\bA_k)$ as required by Condition (b) in \defref{def:aux_decomposition}. It remains to prove that $\calP_k\in\OP(\frakS^{-m_k,0})$. 

Since $(\bA_k)^{-1}:\scrR^{0,2}(\bA_k)\to \scrN_{\bot}^{m_k,2}(\bA_{k-1}^*,B_{k-1}^*)$, and since $\tilde{\calP}_k$ extends to an $L^2$-continuous map, then $\calP_k:L^2\Gamma(\bbE_k)\to W^{m_k,2}\Gamma(\bbE_k)$ continuously. 
By the decomposition \eqref{eq:aux_decompstionAksmooth} and the fact that $\bA_k\calP_k$ is the projection onto $\scrR(\bA_k)$,
\[
\bA_k^*\bA_k\calP_k=\bA_k^* 
\Textand 
B_k^*\bA_k\calP_k=B_k^*.
\]
On the other hand, since $\calP_k$ takes its values in $\scrN_{\bot}(\bA_{k-1}^*,B_{k-1}^*)$,
\[
\bA_{k-1}\calP_{k-1}\calP_k=0 
\Textand 
\bS_k\calP_k=0.
\]
Summarizing,
\[
\mymat{\bA_{k-1}\calP_{k-1}\oplus\bA_k^*\bA_k\oplus\calS_k\\B_k^*\bA_k}\calP_k=\mymat{0\oplus \bA_k^*\oplus 0 \\ B_k^*}. 
\]
By \propref{prop:new_estimate2}, the system
$(\bA_{k-1}\calP_{k-1}\oplus\bA_k^*\bA_k\oplus\calS_k,B_k^*\bA_k)$ is OD elliptic and injective of order and class $2m_k$. Hence, it admits a left-inverse of order $-2m_k$ and class $0$ within the calculus. 
Thus, $\calP_k$ is the composition of a Green operator of class $-2m_k$ and class $0$, and a Green operator of order and class $m_k$. By the composition rules, $\calP_k$ is a Green operator of order $-m_k$ and class at most $m_k$. Finally, since $\calP_k$ is in particular $L^2\rightarrow W^{m_k,2}$ continuous, it follows from \propref{prop:L2_continuity} that it is of class zero. This completes the proof.
\end{proof}
%%%%%%%%%%%%%%%%%%%%%%

%%%%%%%%%%%%%%%%%%%%%%%%%%%%%%%%%%%%%%%%%%%%%%%%%%
\subsection{Stage 5: Construction of $\bA_{k+1}$} 

We complete the induction step by proving that there exists a unique operator $\bA_{k+1}$ with the properties outlined in \thmref{thm:corrected_complex} and \propref{prop:correction}. We define $\bA_{k+1}:\Gamma(\E_{k+1})\to \Gamma(\E_{k+2})$ by
\[
\bA_{k+1} = A_{k+1}(\id - \bA_k\calP_k).
\]
We need to show that
\begin{enumerate}[itemsep=0pt,label=(\alph*)]
\item $\bA_{k+1}\bA_k = 0$.
\item The restriction of $\bA_{k+1}$ to $\scrN(\bA_k^*,B_k^*)$ acts as $A_{k+1}$.
\item The operator $G_{k+1} = \bA_{k+1} - A_{k+1}$ belongs to $\frakG^0$ and is given by
\[
G_{k+1} = -A_{k+1} A_k \calP_k - A_{k+1} G_k \calP_k. 
\]
\end{enumerate}

The first two items follow from the auxiliary decomposition
\[
\Gamma(\E_{k+1})=\scrR(\bA_k)\oplus\scrN(\bA_k^*,B_k^*),
\]
and the fact that $\bA_k\calP_k$ is the projection onto $\scrR(\bA_k)$, or equivalently, $\id - \bA_k\calP_k$ is the projection onto $\scrN(\bA_k^*,B_k^*)$. The uniqueness of $\bA_{k+1}$ as an operator satisfying these two properties is evident from the direct decomposition.  

For the third item, by the definition of $\bA_{k+1}$, 
\[
G_{k+1} = \bA_{k+1} - A_{k+1} = -A_{k+1}\bA_k\calP_k=-A_{k+1}A_k\calP_k-A_{k+1}G_k\calP_k.
\]
By the definition of an elliptic pre-complex, $A_{k+1}A_k$ is a differential operator of order $\leq m_k$. Since $\calP_k$ is of order $-m_k$ and class $0$, by the composition rules, $A_{k+1}A_k\calP_k\in\frakG^0$. In the same way, by the induction hypothesis \eqref{eq:recursive_correction} on $G_k$, 
\[
A_{k+1}G_k\calP_k =-
\underbrace{A_{k+1}A_k}_{\le m_k} \underbrace{A_{k-1}\calP_{k-1}}_0 \underbrace{\calP_k}_{-m_k} -
\underbrace{A_{k+1}A_k}_{\le m_k}
\underbrace{G_{k-1}}_{\le 0}
\underbrace{\calP_{k-1}}_{-m_{k-1}}
\underbrace{\calP_k}_{-m_k},
\]
where the expressions under the braces represent the orders the operators, whereas all the classes are zero. 
By the composition rules, $A_{k+1}G_k\calP_k\in\frakG^0$, hence so is $G_{k+1}$.

%%%%%%%%%%%%%%%%%%%%%%%%%%%%%%%%%%%%%%%%%%%%%%%%%%%%%%
\subsection{The Hodge-like decomposition}

In view of \eqref{eq:WspN0}, \thmref{thm:hodge_like_corrected_complex} asserts that
\[
\scrN^{s,p}_{\bot}(\bA_{k-1}^*,B_{k-1}^*)=\scrR^{s,p}(\bA_k^*;B_{k-1}^*)
\]
for every $s\in\Nzero$ and $1<p<\infty$.
We prove it in several steps.

%%%%%%%%%%%%%%%%%%%%%%%%%%%
\begin{proposition}
\label{prop:correction_final_dual_new}
For every $s\in\Nzero$ and $1<p<\infty$,
\[
\scrR^{s,p}{(\bA_k^*;B_k^*)}\subseteq\scrN_{\bot}^{s,p}(\bA_{k-1}^*,B_{k-1}^*).
\]
\end{proposition}
%%%%%%%%%%%%%%%%%%%%%%

%%%%%%%%%%%%%
\begin{proof}
Let $\psi\in\Gamma(\E_{k+1})$ satisfy $B_k^*\psi=0$ and let $\eta\in\Gamma(\E_{k-1})$. Iterating twice the integration by parts formula \eqref{eq:integration_by_parts_corrected_k}, 
using the fact that $\bA_k\bA_{k-1}=0$,
\[
\bra \bA_{k-1}^* \bA_k^* \psi,\eta\ket =-\bra B_{k-1}^* \bA_k^* \psi,B_{k-1}\eta\ket. 
\] 
Taking $B_{k-1}\eta=0$, using the density of such elements in $L^2\Gamma(\E_{k-1})$, it follows that $\bA_{k-1}^* \bA_k^* \psi=0$. Prescribing $B_{k-1}\eta$ arbitrarily, as $B_{k-1}$ is surjective, it follows that $B_{k-1}^* \bA_k^* \psi=0$. Hence, $\scrR(\bA_k^*;B_k^*)\subseteq \scrN(\bA_{k-1}^*,B_{k-1}^*)$. 

To show that $\scrR(\bA_k^*;B_k^*)\,\bot\, \module^k\Complex$, let $\bA_k^*\psi\in \scrR(\bA_k^*;B_k^*)$ and $\eta\in \module^k\Complex$. Integrating by parts,
\[
\bra \bA_k^*\psi,\eta\ket = \bra \psi,\bA_k \eta\ket - \bra B_k^*\psi, B_k\eta\ket = 0,
\]
where we used the fact that $\module^k\Complex \subset \ker\bA_k$ and $B_k^*\psi=0$.
The proof easily generalizes to arbitrary $\psi\in W^{s,p}_{A_k^*}\Gamma(\E_{k+1})$.
\end{proof}
%%%%%%%%%%%%%

%%%%%%%%%%%%%%%%%%%%%%%%%%%
\begin{proposition}
For all $s\in\Nzero$ and $1<p<\infty$, the space $\scrR^{s,p}{(\bA_k^*;B_k^*)}$ is a closed subspace of $W^{s,p}\Gamma(\E_k)$.
Moreover,
\[
\scrR^{s,p}(\bA_k^*;B_k^*) = \scrR^{0,p}(\bA_k^*;B_k^*)\cap W^{s,p}\Gamma(\bbE_k). 
\]
\end{proposition}
%%%%%%%%%%%%%%%%%%%%%%

%%%%%%%%%%%%%
\begin{proof}
By definition,
\[
\scrR^{s,p}{(\bA_k^*;B_k^*)} = \{\bA_k^*\eta ~:~ \eta\in W^{s+m_k,p}\Gamma(\E_{k+1}), \,\, B_k^*\eta=0\}.
\]
In view of the decomposition \eqref{eq:WspN0},
\[
\begin{aligned}
& W^{s+m_k,p}\Gamma(\bbE_{k+1}) = \scrR^{s+m_k,p}(\bA_k)\oplus\scrN_\bot^{s+m_k,p}(\bA_k^*,B_k^*)\oplus\module^{k+1}\Complex \\ 
& W^{s+2m_k,p}\Gamma(\bbE_k) = \scrR^{s+2m_k,p}(\bA_{k-1})\oplus\scrN_\bot^{s+2m_k,p}(\bA_{k-1}^*,B_{k-1}^*)\oplus\module^k\Complex,
\end{aligned}
\]
and the fact that $\bA_k \bA_{k-1}=0$, it follows that 
\[
\begin{split}
\scrR^{s,p}{(\bA_k^*;B_k^*)} &= \{\bA_k^* \bA_k \psi ~:~ \psi\in W^{s+2m_k,p}\Gamma(\E_k), \,\, B_k^* \bA_k \psi=0\} \\
&= \{\bA_k^* \bA_k \psi ~:~ \psi\in \scrN_\perp^{s+2m_k,p}(\bA_{k-1}^*,B_{k-1}^*), \,\, B_k^* \bA_k \psi=0\} .
\end{split}
\]
By the estimate \eqref{eq:AkstarAk_OD_elliptic_estimate}, for every $\psi\in \scrN_\perp^{s+2m_k,p}(\bA_{k-1}^*,B_{k-1}^*)$ satisfying $B_k^* \bA_k \psi=0$,
\[
\|\psi\|_{s + 2m_k,p} \lesssim \|\bA_k^* \bA_k \psi\|_{s,p}, 
\]
which by \cite[p.~583]{Tay11a}, considering the map
\[
\bA_k^* \bA_k : \scrN_\perp^{s+2m_k,p}(\bA_{k-1}^*,B_{k-1}^*) \cap \ker B_k^* \bA_k \to W^{s,p}\Gamma(\E_k),
\]
 implies that $\scrR^{s,p}{(\bA_k^*;B_k^*)}$ is a closed subspace of $W^{s,p}\Gamma(\E_k)$.

For the second clause, we need to show that
\[
\psi\in  \scrN_\perp^{2m_k,p}(\bA_k^*,B_k^*)
\qquad 
B_k^* \bA_k \psi=0
\Textand
\bA_k^* \bA_k \psi\in  W^{s,p}\Gamma(\E_k)
\]
implies that
\[
\psi\in  \scrN_\perp^{s+2m_k,p}(\bA_k^*,B_k^*),
\]
which follows from that same inequality.
\end{proof}
%%%%%%%%%%%%%

%%%%%%%%%%%%%
\begin{proposition}
For every $s\in\Nzero$ and $1<p<\infty$,
\[
\scrR^{s,p}{(\bA_k^*;B_k^*)} = \scrN_{\bot}^{s,p}(\bA_{k-1}^*,B_{k-1}^*).
\]
\end{proposition}
%%%%%%%%%%%%%

%%%%%%%%%%%%%
\begin{proof}
By \propref{prop:L2_aux_adapted_pair}, with $(\bA,B_A) = (\bA_k^*;B_k^*)$, and since $\scrR^{0,2}(\bA_k^*;B_k^*)$ is closed,
\[
L^2\Gamma(\bbE_k) =  \scrR^{0,2}(\bA_k^*;B_k^*) \oplus \scrN^{0,2}(\bA_k),
\]
whereas from the auxiliary decomposition
\[
L^2\Gamma(\bbE_k) =  \scrR^{0,2}(\bA_{k-1}) \oplus \scrN_\bot^{0,2}(\bA_{k-1}^*,B_{k-1}^*) \oplus \module^k\Complex.
\]
Both decompositions are $L^2$-orthogonal. We note that from \propref{prop:aprioriAk} and \eqref{eq:estimateNsp},
\[
\scrN^{0,2}(\bA_k) \cap \scrN_{\bot}^{0,2}(\bA_{k-1}^*,B_{k-1}^*) = \scrN(\bA_k) \cap \scrN_{\bot}(\bA_{k-1}^*,B_{k-1}^*)=\{0\},
\]
By the injectivity of the system $(\bA_k\oplus \bA_{k-1}^* \oplus \calS_k, B_{k-1}^*)$,
which implies that
\[
\scrN_{\bot}^{0,2}(\bA_{k-1}^*,B_{k-1}^*) \subseteq \scrR^{0,2}(\bA_k^*;B_k^*),
\]
and together with \propref{prop:correction_final_dual_new} we obtain an equality. Intersecting both sides with $W^{s,p}\Gamma(E_k)$, 
using the second clause of \propref{prop:correction_final_dual_new}, we obtain
\[
\scrR^{s,2}(\bA_k^*;B_k^*)=\scrN_{\bot}^{s,2}(\bA_{k-1}^*;B_{k-1}^*)
\]
for every $s\in\Nzero$.
Since this holds every $s\in\Nzero$,
\[
\scrR(\bA_k^*;B_k^*)=\scrN_{\bot}(\bA_{k-1}^*;B_{k-1}^*). 
\] 
The general $W^{s,p}$ version follows from the fact that the projections onto the various subspaces all belong to $\frakG^0$.
\end{proof}
%%%%%%%%%%%%%

%%%%%%%%%%%%%%%%%%%%%%%%%%%%%%%%%%%%%%%%%%%%%%%%
\section{Bianchi complexes}
\label{sec:bianchi}

%%%%%%%%%%%%%%%%%%%%%%%%%%%%%%%%%%%%%%%%%%%%%%%%%%%%
\subsection{The bundle of Bianchi covectors}

Let $(M,\g)$ be a $d$-dimensional Riemannian manifold with smooth boundary. We denote by 
\[
\Lkm{k}{m} = \Lambda^k T^*M \otimes \Lambda^m T^*M
\]
the vector bundle of \emph{$(k,m)$-covectors} (i.e., $k$-covectors taking values in the bundle of $m$-covectors), and by 
\[
\LkmAll = \bigoplus_{k,m} \Lkm{k}{m}
\]
the graded vector bundle of \emph{double-covectors}. The bundle $\LkmAll$ is a graded algebra, endowed with a  graded wedge-product, 
\[
\wedge : \Lkm{k}{m} \times \Lkm{\ell}{n} \to \Lkm{k+\ell}{m+n},
\]
and a graded involution,
\[
(\cdot)^T : \Lkm{k}{m} \to \Lkm{m}{k},
\]
obtained by switching the form and vector parts.
A $(k,k)$-covector $\psi$ satisfying $\psi^T = \psi$ is called \emph{symmetric}. The vector bundle $\LkmAll$ is equipped with a graded \emph{Hodge-dual} isomorphism,
\[
\starG : \Lkm{k}{m} \to \Lkm{d-k}{m},
\]
defined by its action on the form part. To every operation on the form part corresponds an operation on the vector part, via   involution; in this case,
\[
\starGV : \Lkm{k}{m} \to \Lkm{k}{d-m},
\]
is defined by $\starGV\psi = (\starG\psi^T)^T$. Additional graded bundle maps are the \emph{interior products}
\[
i_X  : \Lkm{k}{m} \to  \Lkm{k-1}{m}
\Textand
i^V_X : \Lkm{k}{m} \to  \Lkm{k}{m-1},
\]
where $X$ is a tangent vector, $i_X$ is defined as usual via its action on the form part, and $i^V_X\psi = (i_X \psi^T)^T$, and the \emph{metric trace},
\[
\trace_\g : \Lkm{k}{m} \to  \Lkm{k-1}{m-1}
\qquad\text{defined by}\qquad
\trace_\g\psi = \sum_{i=1}^d i_{E_i} i_{E_i}^V\psi,
\]
where $\{E_i\}_{i=1}^d$ is an orthonormal basis. 

The \emph{Bianchi sum} $\G:\Lkm{k}{m}\to \Lkm{k+1}{m-1}$ is a smooth bundle map given by \cite{Kul72,Gra70},
\[
\G = \sum_{i=1}^d \vartheta^i\wedge i_{E_i}^V,
\]
where $\{\vartheta^i\}_{i=1}^d$ is the basis of covectors dual to $\{E_i\}_{i=1}^d$.
For $\psi\in \Lkm{k}{m}$ and $\eta\in \LkmAll$, the Bianchi sum satisfies the product rule 
\[
\G(\psi\wedge\eta) = \G\psi\wedge\eta + (-1)^{k+m}\psi\wedge \G\eta.
\]
The operator $\GV:\Lkm{k}{m}\to \Lkm{k-1}{m+1}$ is the smooth bundle map $\GV\psi=(\G\psi^T)^T$. The operators $\G$ and $\GV$ are mutually dual with respect to the fiber metric,
\[
(\G\psi,\eta)_\g = (\psi,\GV\eta)_\g.
\] 
The following algebraic commutation and anti-commutation relations are readily verifiable from the definitions:
\[
\begin{aligned}
&\Brk{\G,\GV}|_{\Lkm{k}{m}} = (k-m)\id \\ 
&[\G,\g\wedge] = 0 
&\quad
&[\GV,\g\wedge] = 0 \\
%&\quad
&[\G,\trace_\g] = 0 
&\quad
&[\GV,\trace_\g] = 0 \\
&\{\G,i_X\} = i_X^V 
&\quad 
&\{\G,i_X^V\} = 0 \\
%&\quad
&\{\GV,i_X^V\} = i_X 
&\quad
&\{\GV,i_X\} = 0,
\end{aligned}
\]
where $[A,B] = AB - BA$ and $\{A,B\} = AB+BA$. The tensorial operators $\G$, $\GV$, $\g\wedge$ and $\trace_\g$ are related via the Hodge duals $\starG$ and $\starG^V$ \cite{KL21a}.
The following orthogonal decompositions are established in \cite{Cal61},
\[
\LkmAll = \ker\G\oplus\image\GV = \ker\GV\oplus\image\G,
\]
with $\ker\G=\{0\}$ when $\G$ is restricted to $\Lkm{k}{m}$ for $k<m$ and $\ker\GV=\{0\}$  when $\GV$ is restricted to $\Lkm{k}{m}$ for $k>m$. 
That is, $\G$ is injective and $\GV$ is surjective on $\Lkm{k}{m}$ for $k<m$ and $\GV$ is injective and $\G$ is surjective on $\Lkm{k}{m}$ for $k>m$.

%%%%%%%%%%%%%
\begin{definition}
\label{def:Bianchi_forms}
We define the vector bundles of \emph{Bianchi $(k,m)$-covectors}, 
\[
\Gkm{k}{m} = \Cases{
\Lkm{k}{m}\cap\ker\GV & k\le m \\
\Lkm{k}{m}\cap\ker\G & k\ge m ,
}
\]
along with the graded \emph{bundle of Bianchi coverctors}, 
\[
\GkmAll = \bigoplus_{k,m=0}^d \Gkm{k}{m}.
\]
\end{definition}
%%%%%%%%%%%%%

For $k=m$, the kernels of $\G$ and $\GV$ coincide, and consist of symmetric double-covectors \cite[Prop.~2.2]{Gra70}. In particular, $\Gkm{1}{1}$ coincides with the bundle of symmetric $(1,1)$-covectors and $\Gkm{2}{2}$ is the bundle of $(2,2)$-covectors satisfying the \emph{algebraic Bianchi identities} (also known as \emph{algebraic curvature tensors}). 

We denote by $\calPG:\Lkm{k}{m}\to\Gkm{k}{m}$ the orthogonal projection of a double-covector on $\Gkm{k}{m}$; it has an explicit representation which will not be needed. Since $\GV\psi = (\G\psi^T)^T$, it follows that $\calPG$ commutes with the involution, i.e., $(\calPG\psi)^T = \calPG\psi^T$.

Let $\xi\in\Lkm10$. 
The operators $\ixi$ and $\psi\mapsto \xi\wedge\psi$, which are dual with respect to the fiber metric $(\cdot,\cdot)_\g$, can be restricted to Bianchi forms. Since the first commutes with $\GV$ and the second commutes with $\G$,
\[
\begin{aligned}
& \ixi : \Gkm{k}{m} \to\Gkm{k-1}{m} & \qquad\qquad & k\le m \\
& \xi\wedge : \Gkm{k}{m} \to\Gkm{k+1}{m} & \qquad\qquad & k\ge m.
\end{aligned}
\]
The Bianchi symmetry is however not preserved for arbitrary $k,m$. 
We introduce the \emph{Bianchi wedge-product} and the corresponding \emph{Bianchi interior product}:
\[
\calPG (\xi\wedge):\Gkm{k}{m}\to \Gkm{k+1}{m}
\Textand
\calPG \ixi:\Gkm{k}{m}\to \Gkm{k-1}{m}.
\]

For values of $k,m$ for which a projection is needed, we obtain the following explicit formulas:

%%%%%%%%%%%%
\begin{proposition}
\label{prop:W_explicit}
Let $\psi\in\Gkm{k}{m}$. Then,
\[
\begin{aligned}
&\calPG(\xi\wedge\psi) =  \xi\wedge \psi - \frac{1}{\alpha(m,k)}\G(\xi_V\wedge\psi)
&\qquad\qquad & k<m  \\
&\calPG\ixi\psi = \ixi\psi - \frac{1}{\alpha(k,m)}\GV\ixiV\psi,
&\qquad\qquad & k>m ,
\end{aligned}
\]
where $\alpha(k,m) = k-m+1$.
\end{proposition}
%%%%%%%%%%%%%%%%%%%%%%%%%%%

%%%%%%%%%%%%%
\begin{proof}
We prove the second statement, as the first follows then from duality.
Let $\psi\in\Gkm{k}{m}$ for $k>m$, i.e., $\psi\in\ker\G$.
Consider the right-hand side. 
We start by verifying that it is in $\ker\G$, i.e., it is a Bianchi form,
\[
\begin{split}
\G\brk{\ixi\psi - \frac{1}{\alpha(k,m)}\GV\ixiV\psi} &=
\G\ixi\psi - \frac{1}{\alpha(k,m)}\brk{\GV\G + \alpha(k,m)\id}\ixiV\psi \\
&= \G\ixi\psi - \ixiV\psi \\
&= - \ixi\G\psi \\
&= 0,
\end{split} 
\]
where in the first equality we used the commutation relation of $\G$ and $\GV$, in the passage to the second line we used the fact that $\G$ commutes with $\ixiV$, and in the passage to the third line we used the commutation relation between $\G$ and $\ixi$.
Since $\image\GV$ is orthogonal to $\ker\G$, the application of $\calPG$ on right-hand side equals $\calPG\ixi\psi$, which completes the proof.
\end{proof}
%%%%%%%%%%%%%

%%%%%%%%%%%%%%%%%%%%%%%%%%%%%%%%%%%%%%%%%%%%%%%%%%%%
\subsection{First-order differential operators}

We denote by $\Wkm{k}{m} = \Gamma(\Lkm{k}{m})$ the space of \emph{$(k,m)$-forms}, 
endowed with the inner-product
\beq
\bra\psi,\eta\ket= \int_M (\psi,\eta)_\g \,d\Volume_\g.
\label{eq:L2_double_forms}
\eeq
All the bundle maps defined on $\Lkm{k}{m}$ extend into tensorial operations on $\Wkm{k}{m}$. We denote by $\Ckm{k}{m} = \Gamma(\Gkm{k}{m})$ the space of  \emph{Bianchi $(k,m)$-forms}, and by 
\[
\CkmAll = \bigoplus_{k,m} \Ckm{k}{m}
\]
the graded space of \emph{Bianchi forms}.

We denote by
\[
\dg : \Wkm{k}{m} \to \Wkm{k+1}{m}
\Textand
\dgV : \Wkm{k}{m} \to \Wkm{k}{m+1}
\]
the exterior covariant derivative (defined in the same way as for any bundle-valued form) and 
its vectorial counterpart, $\dgV\psi = (\dg\psi^T)^T$. We denote by 
\[
\deltag : \Wkm{k+1}{m} \to \Wkm{k}{m}
\Textand
\deltagV : \Wkm{k}{m+1} \to \Wkm{k}{m}
\]
the respective formal $L^2$-adjoint of $\dg$ and $\dgV$, where 
$\deltagV\psi = (\deltag\psi^T)^T$.

These first-order operators satisfy the following commutation and anti-commutation relations with the tensorial operators:
\[
\begin{aligned}
&\{\dg,\g\wedge\} = 0 
&\quad
&\{\dgV,\g\wedge\} = 0 
&\quad
&\{\deltag,\g\wedge\} = -\dgV 
&\quad
&\{\deltagV,\g\wedge\} = - \dg \\
&\{\dg,\trace_\g\} = -\deltagV
&\quad
&\{\dgV,\trace_\g\} = -\deltag 
&\quad
&\{\deltag,\trace_\g\} = 0
&\quad
&\{\deltagV,\trace_\g\} = 0 \\ 
&\{\dg,\G\} = 0
&\quad
&\{\dgV,\G\} = \dg 
&\quad
&\{\deltag,\G\} = \deltagV
&\quad
&\{\deltagV,\G\} = 0 \\
&\{\dgV,\GV\} = 0
&\quad
&\{\dg,\GV\} = \dgV 
&\quad
&\{\deltagV,\GV\} = \deltag
&\quad
&\{\deltag,\GV\} = 0 \\
&[\dg,\starG^V] = 0 
&\quad
&[\dgV,\starG] = 0.
\end{aligned}
\]

The operators $\dg$ and $\deltag$ can be restricted to Bianchi forms.  Due to the commutation relations $\{\G,\dg\}=0$ and $\{\GV,\deltag\}=0$, 
\[
\begin{gathered}
\dg:\Ckm{k}{m}\to\Ckm{k+1}{m} \qquad\text{for $k\ge m$} \\
\deltag:\Ckm{k}{m}\to\Ckm{k-1}{m} \qquad\text{for $k\le m$}.
\end{gathered}
\]
The Bianchi symmetry is however not preserved by $\dg$ and $\deltag$ for every $(k,m)$-form. This can be rectified by projecting their image onto the Bianchi bundle.

%%%%%%%%%%%%%
\begin{definition}
%\label{def:}
The \emph{Bianchi derivative}, $\dBianchi:\Ckm{k}{m}\to\Ckm{k+1}{m}$, the \emph{Bianchi coderivative}, $\delBianchi:\Ckm{k+1}{m}\to\Ckm{k}{m}$, and their vectorial counterparts, $\dBianchiV:\Ckm{k}{m}\to\Ckm{k}{m+1}$ and $\delBianchiV:\Ckm{k}{m+1}\to\Ckm{k}{m}$ are given by
\beq
\dBianchi\psi =  \calPG \dg\psi
\Textand
\delBianchi\psi =
\calPG \deltag\psi,
\label{eq:Bianchi_codifferential} 
\eeq
along with $\dBianchiV\psi =(\dBianchi\psi^T)^T$ and $\delBianchiV\psi = (\delBianchi\psi^T)^T$.
\end{definition}
%%%%%%%%%%%%%

The Bianchi derivative $\dBianchi$ and the Bianchi coderivative $\delBianchi$ (and likewise $\dBianchiV$ and $\delBianchiV$) are mutually adjoint with respect to the $L^2$-inner-product \eqref{eq:L2_double_forms}.

The following is proved in a similar way as \propref{prop:W_explicit}:

%%%%%%%%%%%%
\begin{proposition}
\label{prop:Bianchi_derivative_explicit}
For $\psi\in\Ckm{k}{m}$,
\[
\begin{aligned}
& \dBianchi\psi =  
\dg\psi - \frac{1}{\alpha(m,k)}\G \dgV\psi
&\qquad\qquad
& k<m 
\\
&\delBianchi\psi = 
\deltag\psi - \frac{1}{\alpha(k,m)}\GV\deltagV\psi 
&\qquad\qquad
& k>m. 
\end{aligned}
\]
\end{proposition}
%%%%%%%%%%%%%%%%%%%%%%%%%%%

The fact that $\dg\dg$ is a tensorial operator yields the following:

%%%%%%%%%%%%%%%%%%%%%%%%%%%
\begin{proposition}
\label{prop:DstarDstar}
The maps $\dBianchi\dBianchi:\Ckm{k}{m}\to\Ckm{k+2}{m}$ and  $\delBianchi\delBianchi:\Ckm{k+2}{m}\to\Ckm{k}{m}$ are tensorial for every $k,m$, except when $k=m-1$.  
\end{proposition}
%%%%%%%%%%%%%%%%%%%%%%%%%%%

%%%%%%%%%%%%%%%%%%%%%%%%%%%
\begin{proof}
When restricted to $\Ckm{k}{m}$ for $k\ge m$, $\dBianchi\dBianchi = \dg\dg$, hence is tensorial. When restricted to $\Ckm{k}{m}$ for $k< m-1$, $\dBianchi\dBianchi$ is dual to  $\deltag\deltag$, and since the latter is tensorial, so is the former.

\end{proof}
%%%%%%%%%%%%%%%%%%%%%%%%%%%%%%%%

Let $\jmath:\dM\to M$ denote as before the inclusion map of the boundary.
We
introduce mixed projections of tangential and normal boundary components,
\[
\begin{aligned}
&\PttD:\Wkm{k}{m}\to \plWkm{k}{m}
&\qquad 
&\PntD:\Wkm{k}{m}\to \plWkm{k-1}{m} \\ 
&\PtnD:\Wkm{k}{m}\to \plWkm{k}{m-1} 
&\qquad 
&\PnnD:\Wkm{k}{m}\to \plWkm{k-1}{m-1}.
\end{aligned}
\]
The first superscript in $\frakt\frakt,\frakt\frakn,\frakn\frakt,\frakn\frakn$ refers to the projection of the form part, whereas the second superscript refers to the projection of the vector part. 
Specifically,
\[
\begin{gathered}
\PttD\psi = \jmath^*\psi
\qquad
\PntD\psi = \jmath^*\idr \psi 
\qquad
\PtnD\psi = \jmath^*\idr^T \psi
\textand
\PnnD\psi = \jmath^*\idr^T \idr \psi, 
\end{gathered}
\]
where $\partial_r$ is the unit vector field normal to the level-sets of the distance from the boundary, which is defined in a collar neighborhood of $\dM$, and $\jmath^*$ pulls back to the boundary both the form and vector parts. For $\psi\in\Wkm{k}{m}$ and $\eta\in\Wkm{k+1}{m}$,
\beq
\bra \dg\psi,\eta\ket = \bra\psi,\deltag\eta\ket + \bra (\PttD\oplus\PtnD)\psi, (\PntD\oplus\PnnD)\eta\ket.
\label{eq:dg_by_parts}
\eeq
The definition of the Bianchi sum implies that the pullback $\jmath^*$ commutes with both $\G$ and $\GV$. Furthermore, $\idr$ anti-commutes with $\GV$ and  $\idr^V$ anti-commutes with $\G$. A direct calculation gives the following commutation and anti-commutation relations,
\[
\begin{aligned}
&[\PttD,\G] = 0 
&\quad
&\{\PtnD,\G\} = 0 
&\quad
&\{\PntD,\G\} = \PtnD 
&\quad
&[\PnnD,\G] = 0 \\
&[\PttD,\GV] = 0 
&\quad
&\{\PtnD,\GV\} = \PntD 
&\quad
&\{\PntD,\GV\} = 0 
&\quad
&[\PnnD,\GV] =0 .
\end{aligned}
\] 
As a result,
\[
\begin{aligned}
& \PttD : \Ckm{k}{m} \to \plCkm{k}{m} &\qquad & \text{for every $k,m$} \\
& \PnnD :  \Ckm{k}{m} \to  \plCkm{k-1}{m-1} &\qquad & \text{for every $k,m$} \\
& \PtnD : \Ckm{k}{m}\to \plCkm{k}{m-1} &\qquad & \text{for $k\ge m$} \\
& \PntD : \Ckm{k}{m}\to \plCkm{k-1}{m} &\qquad & \text{for $k\le m$}.
\end{aligned}
\] 
For $k<m$, $\PtnD:\Ckm{k}{m}\to \plWkm{k}{m-1}$ does not yield a Bianchi form, since $\idr^V$ does not commute with $\GV$. The same is true for $\PntD:\Ckm{k}{m}\to \plWkm{k-1}{m}$ when $k>m$. In the same spirit as in formula \eqref{eq:Bianchi_codifferential} for the Bianchi derivatives, we define:

%%%%%%%%%%%%%%%%%%%%%%%%%
\begin{definition}
The \emph{Bianchi boundary operators} 
\[
\begin{aligned}
& \PttG : \Ckm{k}{m} \to \plCkm{k}{m} 
&\qquad\qquad
& \PnnG :  \Ckm{k}{m} \to  \plCkm{k-1}{m-1} \\
& \PtnG : \Ckm{k}{m}\to \plCkm{k}{m-1} 
&\qquad\qquad
& \PntG : \Ckm{k}{m}\to \plCkm{k-1}{m}
\end{aligned}
\] 
are given by 
\[
\begin{gathered}
\PttG = \PttD
\qquad
\PntG = \calPG\PntD
\qquad
\PtnG = \calPG\PtnD
\textand
\PnnG = \PnnD,
\end{gathered}
\]
where $\calPG:\plLkm{k}{m}\to \plLkm{k}{m}$ denotes here the projection on Bianchi boundary forms. 
\end{definition}
%%%%%%%%%%%%%%%%%%%%%%%%%%%%%%%%

Similarly to \propref{prop:Bianchi_derivative_explicit}, we have:

%%%%%%%%%%%%%
\begin{proposition}
For $\psi\in\Ckm{k}{m}$,
\[
\begin{aligned}
& \PtnG\psi = 
\PtnD\psi - \frac{1}{\alpha(m,k)} \G\PntD \psi
&\qquad\qquad
& k< m
\\
& \PntG\psi 
= \PntD\psi - \frac{1}{\alpha(k,m)} \GV\PtnD \psi
&\qquad\qquad
& k> m.
\end{aligned}
\]
\end{proposition}
%%%%%%%%%%%%%

%%%%%%%%%%%%%%%%%%%%%%%%%%%%%
\begin{proposition}
For all $\eta\in W^{1,p}\CkmAll$ and $\sigma\in W^{1,q}\CkmAll$ (the precise class determined by the context), with $1/p+1/q=1$, 
\beq
\bra\dBianchi\eta,\sigma\ket = \bra\eta,\delBianchi\sigma\ket +
\bra B_\calG \eta,B_\calG^* \sigma\ket,
\label{eq:integration_by_parts_DV}
\eeq
where 
\[
B_\calG = \PttG\oplus\PtnG
\Textand
B_\calG^* = \PntG\oplus\PnnG.
\]
Moreover, $\dBianchi$ and $\delBianchi$ are both adapted differential operators. 
\end{proposition}
%%%%%%%%%%%%%%%%%%%%%%%

%%%%%%%%%%%%%%%%%%%%%%%
\begin{proof}
Eq.~\eqref{eq:integration_by_parts_DV} is an immediate consequence of \eqref{eq:dg_by_parts}, along with the properties of $\calPG$. Finally, $B_\calG$ and $B_\calG^*$ are both systems of trace operators associated with order 1; both are surjective, as direct sums of compositions of surjective operators.  
\end{proof}
%%%%%%%%%%%%%%%%%%%%%%

%%%%%%%%%%%%%%%%%%%%%%%%%%%%%%%%%%%%%
\subsection{Second-order differential operators}

In \cite{KL21a} we introduced the \emph{covariant curl-curl} operator, $H:\Wkm{k}{m}\to \Wkm{k+1}{m+1}$,
and its $L^2$-dual, $H^*: \Wkm{k+1}{m+1}\to\Wkm{k}{m}$,
\[
H=\tfrac12(\dg\dgV+\dgV\dg) 
\Textand 
H^*=\tfrac12(\deltag\deltagV+\deltagV\deltag).
\]
These second-order operators satisfy integration by part formulas involving both tensorial and first-order boundary operators. 
We also defined the first-order boundary operators, 
\[
\frakT : \Wkm{k}{m}\to\plWkm{k}{m} 
\Textand
\frakT^* : \Wkm{k}{m}\to\plWkm{k-1}{m-1},
\]
given by 
\[
\begin{aligned}
\frakT\psi &= \tfrac12 \brk{\PntD \dg\psi-\dg \PntD\psi}+\tfrac12 \brk{\PtnD\dgV\psi- \dgV \PtnD\psi} \\
\frakT^*\psi &=  -\tfrac12\brk{ \PtnD\deltag\psi+\deltag \PtnD\psi} - \tfrac12\brk{\PntD\deltagV\psi+\deltagV\PntD\psi},
\end{aligned}
\]
such that 
\[
\bra H\psi,\eta\ket=\bra \psi,H^*\eta\ket +\bra B_H\psi,B_H^*\eta\ket ,
\]
where 
\[
B_H : \Wkm{k}{m}\to (\plWkm{k}{m})^2 
\Textand
B_H^* : \Wkm{k}{m}\to(\plWkm{k-1}{m-1})^2
\]
are given by
\[
B_H=\PttD\oplus(-\frakT)
\Textand
B_H^*=\frakT^*\oplus\PnnD.
\] 

The operators $H$ and $H^*$ both commute with the Bianchi sums $\GV$, $\G$ \cite[Prop.~3.10]{KL21a}, which implies that for every $k,m$,
\[
H:\Ckm{k}{m}\rightarrow \Ckm{k+1}{m+1} 
\Textand
H^*:\Ckm{k+1}{m+1}\rightarrow \Ckm{k}{m}. 
\]
A similar calculation shows that the boundary operators also preserve the Bianchi structure:
\[
B_H : \Ckm{k}{m}\to (\plCkm{k}{m})^2 
\Textand
B_H^* : \Ckm{k}{m}\to(\plCkm{k-1}{m-1})^2.
\]

The fact that $B_H$ and $B_H^*$ are normal systems of trace operators associated with order $2$ is implied by the calculation in the proof of \cite[Lemma.~5.1]{KL21a}. Thus, 

%%%%%%%%%%%%%
\begin{proposition}
%\label{prop:}
The operators $H$ and $H^*$ are second-order adapted differential operators with respect to the boundary operators $B_H$ and $B_H^*$, which are systems of trace operators associated with order 2. 
\end{proposition}
%%%%%%%%%%%%%

%%%%%%%%%%%%%%%%%%%%%%%%%%%%%%%%%%%%%%%%%%%%%%%%%%%%
\subsection{Bianchi complexes}

Let $1\le m\le d$, and consider the following diagram, which we break into two lines:
\[
\begin{xy}
(-20,0)*+{0}="Em1";
(0,0)*+{\Ckm{0}{m}}="E0";
(30,0)*+{\Ckm{1}{m}}="E1";
(60,0)*+{\cdots}="E2";
(90,0)*+{\Ckm{m}{m}}="E3";
(-20,-20)*+{}="Gm1";
(0,-20)*+{\plCkm{0}{m}\oplus\plCkm{0}{m-1}}="G0";
(30,-20)*+{\plCkm{1}{m}\oplus\plCkm{1}{m-1}}="G1";
(60,-20)*+{\cdots}="G2";
(90,-20)*+{(\plCkm{m}{m})^2}="G3";
{\ar@{->}@/^{1pc}/^{\dBianchi}"E0";"E1"};
{\ar@{->}@/^{1pc}/^{\delBianchi}"E1";"E0"};
{\ar@{->}@/^{1pc}/^{\dBianchi}"E1";"E2"};
{\ar@{->}@/^{1pc}/^{\delBianchi}"E2";"E1"};
{\ar@{->}@/^{1pc}/^{\dBianchi}"E2";"E3"};
{\ar@{->}@/^{1pc}/^{\delBianchi}"E3";"E2"};
{\ar@{->}@/^{1pc}/^0"Em1";"E0"};
{\ar@{->}@/^{1pc}/^0"E0";"Em1"};
{\ar@{->}@/_{0pc}/^{B_\calG}"E0";"G0"};
{\ar@{->}@/_{0pc}/^{B_\calG}"E1";"G1"};
{\ar@{->}@/_{0pc}/^{B_\calG}"E2";"G2"};
{\ar@{->}@/^{1pc}/^{B_\calG^*}"E1";"G0"};
{\ar@{->}@/^{1pc}/^{B_\calG^*}"E2";"G1"};
%{\ar@{->}@/^{1pc}/^0"E0";"Gm1"};
{\ar@{->}@/^{1pc}/^{B_\calG^*}"E3";"G2"};
{\ar@{->}@/_{0pc}/^{B_H}"E3";"G3"};
\end{xy}
\]
\[
\begin{xy}
(-30,0)*+{\Ckm{m}{m}}="Em1";
(0,0)*+{\Ckm{m+1}{m+1}}="E0";
(30,0)*+{\cdots}="E1";
(60,0)*+{\Ckm{d}{m+1}}="E2";
(80,0)*+{0}="E3";
(-30,-20)*+{(\plCkm{m}{m})^2}="Gm1";
(0,-20)*+{\plCkm{m+1}{m+1}\oplus\plCkm{m+1}{m}}="G0";
(30,-20)*+{\cdots}="G1";
(60,-20)*+{\plCkm{d}{m+1}\oplus\plCkm{d}{m}}="G2";
(80,-20)*+{}="G3";
{\ar@{->}@/^{1pc}/^{\dBianchi}"E0";"E1"};
{\ar@{->}@/^{1pc}/^{\delBianchi}"E1";"E0"};
{\ar@{->}@/^{1pc}/^{\dBianchi}"E1";"E2"};
{\ar@{->}@/^{1pc}/^{\delBianchi}"E2";"E1"};
{\ar@{->}@/^{1pc}/^0"E2";"E3"};
{\ar@{->}@/^{1pc}/^0"E3";"E2"};
{\ar@{->}@/^{1pc}/^{H}"Em1";"E0"};
{\ar@{->}@/^{1pc}/^{H^*}"E0";"Em1"};
{\ar@{->}@/_{0pc}/^{B_H}"Em1";"Gm1"};
{\ar@{->}@/_{0pc}/^{B_\calG}"E0";"G0"};
{\ar@{->}@/_{0pc}/^{B_\calG}"E1";"G1"};
{\ar@{->}@/_{0pc}/^{B_\calG}"E2";"G2"};
{\ar@{->}@/^{1pc}/^{B_\calG^*}"E1";"G0"};
{\ar@{->}@/^{1pc}/^{B_\calG^*}"E2";"G1"};
{\ar@{->}@/^{1pc}/^{B_H^*}"E0";"Gm1"};
\end{xy}
\]

%%%%%%%%%%%%%
\begin{theorem}
\label{thm:SVcomplex}
This sequence forms an elliptic pre-complex; we call the induced complex a \emph{Bianchi complex}.
\end{theorem}
%%%%%%%%%%%%%

\thmref{thm:SVcomplex} is proved in the next section. In the remaining part of this section we examine its implications.

\thmref{thm:corrected_complex} yields the existence of a unique chain of operators 
\[
\begin{aligned}
&\DBianchi: \Ckm{k}{m} \to \Ckm{k+1}{m}  
&\qquad\qquad
&k=0,\dots,m-1 \\
&\bHg: \Ckm{m}{m} \to \Ckm{m+1}{m+1} \\
&\DBianchi: \Ckm{k}{m+1} \to \Ckm{k+1}{m+1}  
&\qquad\qquad
&k=m+1,\dots,d-1,
\end{aligned}
\]
satisfying,
\[
\begin{aligned}
&\Ckm{k}{m}, k<m &:
&\quad 
&\DBianchi \DBianchi = 0
&\Textand&
&\DBianchi = \dBianchi \quad \text{on $\scrN(\DelBianchi,B_\calG^*)$} \\
&\Ckm{m}{m} &:
&\quad 
&\bHg \DBianchi = 0
&\Textand&
&\bHg = H \quad \text{on $\scrN(\DelBianchi,B_\calG^*)$} \\
&\Ckm{m+1}{m+1} &:
&\quad 
&\DBianchi\bHg  = 0
&\Textand&
&\DBianchi = \dBianchi \quad \text{on $\scrN(\bHg^*,B_H^*)$} \\
&\Ckm{k}{m+1}, k>m &:
&\quad 
&\DBianchi \DBianchi = 0
&\Textand&
&\DBianchi = \dBianchi \quad \text{on $\scrN(\DelBianchi,B_\calG^*)$},
\end{aligned}
\]
where $\DelBianchi = (\DBianchi)^*$. 

The Bianchi complex induces the following (smooth versions of) Hodge-like decompositions:
\[
\begin{aligned}
&\Ckm{k}{m} = 
\lefteqn{\overbrace{\phantom{\scrR(\DBianchi)\oplus \ker(\DBianchi\oplus\DelBianchi\oplus B_\calG^*)}}^{\scrN(\DBianchi)}} 
\scrR(\DBianchi) \oplus \underbrace{\ker(\DBianchi\oplus\DelBianchi\oplus B_\calG^*)\oplus \scrR(\DelBianchi;B_\calG^*)}_{\scrN(\DelBianchi,B_\calG^*)} 
&\qquad\qquad
&k<m \\
&\Ckm{m}{m} = 
\lefteqn{\overbrace{\phantom{\scrR(\DBianchi)\oplus \ker(\DelBianchi\oplus\bHg\oplus B_\calG^*)}}^{\scrN(\bHg)}} 
\scrR(\DBianchi) \oplus \underbrace{\ker(\DelBianchi\oplus\bHg\oplus B_\calG^*)\oplus \scrR(\bHg^*;B_H^*)}_{\scrN(\DelBianchi,B_\calG^*)}   \\
&\Ckm{m+1}{m+1} = 
\lefteqn{\overbrace{\phantom{\scrR(\bHg)\oplus \ker(\bHg^*\oplus\DBianchi\oplus B_H^*)}}^{\scrN(\DBianchi)}} 
\scrR(\bHg) \oplus \underbrace{\ker(\bHg^*\oplus\DBianchi\oplus B_H^*)\oplus \scrR(\DelBianchi;B_\calG^*)}_{\scrN(\bHg^*,B_H^*)}  \\
&\Ckm{k}{m+1} = 
\lefteqn{\overbrace{\phantom{\scrR(\DBianchi)\oplus \ker(\DBianchi\oplus\DelBianchi\oplus B_\calG^*)}}^{\scrN(\DBianchi)}} 
\scrR(\DBianchi) \oplus \underbrace{\ker(\DBianchi\oplus\DelBianchi\oplus B_\calG^*)\oplus \scrR(\DelBianchi;B_\calG^*)}_{\scrN(\DelBianchi,B_\calG^*)} 
&\qquad\qquad
&k>m.
\end{aligned}
\]
where the middle term of every direct sum is finite-dimensional, and consists of smooth sections, even in Sobolev regularity. 
Implications of some of these decompositions are detailed in the introduction.

Finally, we note that the theory holds verbatim if $H$ is altered by any tensorial operator preserving the Bianchi symmetry. This is important, since if we replace $H$ by $H+D$, then by the uniqueness of the correction, $\bHg = H+D$ in constant sectional curvature.  

%%%%%%%%%%%%%%%%%%%%%%%%%%%%%%%%%%%%%%%%%%%%%%%%%%%%
\subsection{Proof of \thmref{thm:SVcomplex}}

To prove \thmref{thm:SVcomplex}, we need to prove the following two propositions:

%%%%%%%%%%%%%
\begin{proposition}\ 
%\label{prop:}
\begin{enumerate}[itemsep=0pt,label=(\alph*)]
\item For $k\le m-2$, $\dBianchi\dBianchi :\Ckm{k}{m}\to \Ckm{k+2}{m}$ is of order 0. 
\item $H\dBianchi :\Ckm{m-1}{m}\to \Ckm{m+1}{m+1}$ is of order 1. 
\item $\dBianchi H :\Ckm{m}{m}\to \Ckm{m+2}{m+1}$ is of order 1. 
\item For $k\ge m+1$, $\dBianchi\dBianchi :\Ckm{k}{m+1}\to \Ckm{k+2}{m+1}$ is of order 0. 
\end{enumerate}
\end{proposition}
%%%%%%%%%%%%%

%%%%%%%%%%%%%
\begin{proof}
Items (a) and (d) were proved  in \propref{prop:DstarDstar}. For Item (b) take for example, for $\psi\in \Ckm{m-1}{m}$,
\[
\begin{split}
H\dBianchi\psi &= \tfrac12(\dg \dgV + \dgV\dg) \brk{\dg - \tfrac12 \G\dgV} \psi \\
&= \tfrac12(\dg \dgV + \dgV\dg) \dg  \psi - \tfrac14 \G(\dg \dgV + \dgV\dg) \dgV \psi \\ 
\end{split}
\]
where in the passage to the second line we used the commutation of $\G$ with $H$. Since $\dg\dg$ and $\dgV\dgV$ are tensorial, and so is the commutator of $\dg$ and $\dgV$, it follows that $H\dBianchi$ is a first-order operator.
Item (c) follows similarly.
\end{proof}
%%%%%%%%%%%%%

%%%%%%%%%%%%%
\begin{proposition}\ 
%\label{prop:}
\begin{enumerate}[itemsep=0pt,label=(\alph*)]
\item $(\delBianchi\oplus\dBianchi,B_\calG^*)$ with domain $\Ckm{k}{m}$, $k<m$,  is OD elliptic.
\item $(\delBianchi\oplus H,B_\calG^*)$ with domain $\Ckm{m}{m}$ is OD elliptic.
\item $(H^*\oplus\dBianchi,B_H^*)$ with domain $\Ckm{m+1}{m+1}$ is OD elliptic.
\item $(\delBianchi\oplus\dBianchi,B_\calG^*)$ with domain $\Ckm{k}{m}$, $k>m$, is OD elliptic.
\end{enumerate}
\end{proposition}
%%%%%%%%%%%%%

%%%%%%%%%%%%%
\begin{proof}
Noting the following Hodge-dualities (up to multiplicative constants),
\[
\begin{aligned}
& \starG\starGV(H^*\oplus\dBianchi,B_H^*)\starG\starGV = (H\oplus\delBianchi,B_H) \\
& \starG\starGV(\delBianchi\oplus\dBianchi,B_\calG^*)\starG\starGV = (\dBianchi\oplus\delBianchi,B_\calG),
\end{aligned}
\]
we may equivalently prove that 
\begin{enumerate}[itemsep=0pt,label=(\alph*)]
\item $(\delBianchi\oplus\dBianchi,B_\calG^*)$ and $(\delBianchi\oplus\dBianchi,B_\calG)$ with domain $\Ckm{k}{m}$, $k<m$,  are OD elliptic.
\item $(\delBianchi\oplus H,B_\calG^*)$ and $(\delBianchi\oplus H,B_H)$ with domain $\Ckm{m}{m}$ are OD elliptic.
\end{enumerate}
Statement (a) is proved in \propref{prop:ellipticity_bianchi} below.
As for Statement (b), we note that $\delBianchi = \deltag$ and $B_\calG^* = \PnnD\oplus\PntD$, when restricted to $\Ckm{m}{m}$, and that $\Ckm{m}{m}$ is a subspace of the symmetric $(m,m)$-forms, hence (b) can be replaced by showing that  
\[
\text{
$(\deltag\oplus H,\PnnD\oplus\PntD)$ and $(\deltag\oplus H,B_H)$ restricted to symmetric $\Wkm{m}{m}$ forms are OD elliptic.
}
\]
This is proved in \propref{prop:ellipticity_H}.
\end{proof}
%%%%%%%%%%%%%

%%%%%%%%%%%%%%%%%%%%%%%%%%%%%%%%%%
\begin{proposition}
\label{prop:ellipticity_bianchi}
The systems 
\[
(\delBianchi\oplus\dBianchi,B_\calG^*)
\Textand
(\delBianchi\oplus\dBianchi, B_\calG),
\] 
restricted to $\Ckm{k}{m}$, $k< m$, are OD elliptic.
\end{proposition}
%%%%%%%%%%%%%%%%%%%%%%%%%%%%%%%%%%

%%%%%%%%%%%%%
\begin{proof}
Since all the operators, including the boundary operators are differential operators,  we need to verify that the criteria in \thmref{thm:adatped_OD_ellipticity_estiamte} are satisfied: we need to show that for every $x\in M$ and $\xi\in T_x^*M\setminus\{0\}$, 
\[
\sigma_{\delBianchi}(x,\xi) \oplus \sigma_{\dBianchi}(x,\xi) : \Gkm{k}{m} \to \Gkm{k-1}{m} \oplus  \Gkm{k+1}{m}
\]
is injective, and so are the maps
\[
\begin{aligned}
\Xi^1_{x,\xi'} &= \sigma_{\PnnG}(x,\xi'+ \imath\,\partial_s dr)\oplus \sigma_{\PntG}(x,\xi'+ \imath\,\partial_s dr)|_{s=0} \\
\Xi^2_{x,\xi'} &= \sigma_{\PttG}(x,\xi'+ \imath\,\partial_s dr)\oplus \sigma_{\PtnG}(x,\xi'+ \imath\,\partial_s dr)|_{s=0}
\end{aligned}
\]
for $x\in\dM$ and $\xi'\in T_x^*\dM\setminus\{0\}$, when restricted to the space $\bbM_{x,\xi'}^+$ of decaying solutions to the ordinary differential system
\beq
\brk{\sigma_{\delBianchi}(x,\xi' + \imath\,\partial_s dr) \oplus \sigma_{\dBianchi}(x\xi' ,+ \imath\,\partial_s dr)}\psi(s) = 0.
\label{eq:ODE_ellipticity_bianchi}
\eeq

Let $x\in M$ and let $\xi\in T^*_xM$; we denote $\xi_V = \xi^T = \GV\xi$; without loss of generality, we may assume that $|\xi|_\g=1$.
For $\psi\in\Gkm{k}{m}|_x$, $k<m$,
\[
\begin{aligned}
-\imath \sigma_{\dBianchi}(x,\xi) \psi &= \xi\wedge\psi - \frac{1}{\alpha(m,k)} \G(\xi_V\wedge \psi)
\\
\imath \sigma_{\delBianchi}(x,\xi) \psi &= \ixi\psi.
\end{aligned}
\]
Suppose that $(\sigma_{\delBianchi}(x,\xi) \oplus \sigma_{\dBianchi}(x,\xi)) \psi=0$, then
\[
\ixi\brk{\xi\wedge\psi - \frac{1}{\alpha(m,k)} \G(\xi_V\wedge \psi)} = 0.
\]
Using the fact that $\sigma_{\delBianchi}(x,\xi) \psi=0$, $\ixi (\xi\wedge) + \xi\wedge\ixi = \id$ and the anti-commutation relation between $\ixi$ and $\G$, we obtain
\[
\psi + \frac{1}{\alpha(m,k)} \ixiV(\xi_V\wedge \psi) = 0.
\]
Taking an inner-product with $\psi$,
\[
|\psi|_\g^2 + \frac{1}{\alpha(m,k)} |\xi_V\wedge\psi|_\g^2 = 0, 
\]
which implies that $\psi=0$, i.e., $\sigma_{\delBianchi}(x,\xi) \oplus \sigma_{\dBianchi}(x,\xi)$ is injective.

We proceed to establish the injectivity of the boundary symbols.
Let $x\in\dM$ and $\xi \in T_x^*\dM$. 
The ordinary differential equation \eqref{eq:ODE_ellipticity_bianchi} takes the form of a system,
\begin{gather}
\xi\wedge \psi - \frac{1}{\alpha(m,k)} \G(\xi_V\wedge \psi) + \imath \brk{dr\wedge \dot{\psi} - \frac{1}{\alpha(m,k)} \G(dr_V\wedge \dot{\psi})} = 0
\label{eq:ODE1} \\ 
\ixi\psi + \imath i_{\dr} \dot{\psi} =0.
\label{eq:ODE2} 
\end{gather}
To solve it, we decompose $\psi$ orthogonally (as an element in $\Lkm{k}{m}|_x$)
\[
\begin{split}
\psi &= \brk{\psi_{00} + \xi\wedge \psi_{01} + dr\wedge \psi_{02} + \xi\wedge dr \wedge \psi_{03}} \\
&+ \xi_V\wedge \brk{\psi_{10} + \xi\wedge \psi_{11} + dr\wedge \psi_{12} + \xi\wedge dr \wedge \psi_{13}} \\
&+ dr^T\wedge \brk{\psi_{20} + \xi\wedge \psi_{21} + dr\wedge \psi_{22} + \xi\wedge dr \wedge \psi_{23}} \\
&+ \xi_V\wedge dr^T\wedge \brk{\psi_{30} + \xi\wedge \psi_{31} + dr\wedge \psi_{32} + \xi\wedge dr \wedge \psi_{33}},
\end{split}
\]
where 
\[
\ixi\psi_{ij}=0
\qquad  
\ixiV\psi_{ij}=0
\qquad  
i_{\dr}\psi_{ij}=0
\Textand
i_{\dr}^V\psi_{ij}=0
\]
for every $i,j=0,1,2,3$.
Substituting into \eqref{eq:ODE1} and \eqref{eq:ODE2}, equating like terms we obtain the following equations:
\beq
\psi_{00}  =  0
\label{eq:00}
\eeq
\beq
\Cases{
\psi_{01} + \imath \dot\psi_{02} = 0 \\
\psi_{02} - \imath \dot\psi_{01} = 0 
}
\qquad\qquad
\Cases{
\psi_{10} + \tfrac{1}{\alpha(m,k+1)}(-\G\psi_{01} + \psi_{10}) = 0 \\
\psi_{20} + \tfrac{\imath}{\alpha(m,k+1)}(-\G\dot\psi_{01} + \dot\psi_{10}) = 0
}
\label{eq:01}
\eeq
\beq
\Cases{
\psi_{11} + \imath \dot\psi_{12}  = 0 \\
\psi_{21} + \imath \dot\psi_{22}  = 0
}
\qquad\qquad
\Cases{
\psi_{12} - \imath\dot\psi_{11} + \frac{1}{\alpha(m,k+1)}(\psi_{12} - \psi_{21})  = 0\\
\psi_{22} - \imath\dot\psi_{21} + \frac{\imath}{\alpha(m,k+1)}(\dot\psi_{12} - \dot\psi_{21}) = 0
}
\label{eq:11}
\eeq
\beq
\Cases{
\psi_{31} + \imath \dot\psi_{32}  = 0 \\
\psi_{32} - \imath\dot\psi_{31} = 0
}
\label{eq:31}
\eeq
and
\beq
\psi_{30} + \tfrac{1}{\alpha(m,k+1)} (\G\psi_{21} + \psi_{30}) - \tfrac{\imath}{\alpha(m,k+1)} \G\dot\psi_{11} = 0,
\label{eq:30}
\eeq
and
\[
\psi_{i3}=0
\qquad\qquad
\text{for $i=0,1,2,3$}. 
\]

The initial conditions, can be orthogonally decomposed similarly (noting that terms including $dr$ and $dr^T$ are annihilated by the pullback to boundary).

Let $\psi\in \bbM_{x,\xi'}^+$. The condition that $\Xi^1_{x,\xi'}\psi=0$ results in
\[
\psi_{02}(0) = 0
\qquad
\psi_{12}(0) = 0
\qquad
\psi_{22}(0) = 0
\Textand
\psi_{32}(0) = 0.
\]
Substituting $\psi_{01}(0)=0$ into \eqref{eq:01}, along with the condition the solution is decaying at infinity, yields
\[
\psi_{01}(s) = 0
\qquad
\psi_{02}(s) = 0
\qquad
\psi_{10}(s) = 0
\Textand
\psi_{20}(s) = 0.
\]
Substituting $\psi_{12}(0) = 0$ and 
$\psi_{22}(0) = 0$ into \eqref{eq:11} yields
\[
\psi_{11}(s) = 0
\qquad
\psi_{12}(s) = 0
\qquad
\psi_{21}(s) = 0
\Textand
\psi_{22}(s) = 0.
\]
Substituting $\psi_{32}(0) = 0$ into \eqref{eq:31} yields
\[
\psi_{31}(s) = 0
\Textand
\psi_{32}(s) = 0.
\]
Finally, \eqref{eq:30} yield that
\[
\psi_{30}(s) = 0, 
\]
thus $\psi(s) = 0$, proving the injectivity of $\Xi^1_{x,\xi'}$ when restricted to $\bbM^+_{x,\xi'}$.

The condition that $\Xi^2_{x,\xi'}\psi=0$ results in
\[
\psi_{01}(0) = 0
\qquad
\psi_{11}(0) = 0 
\qquad
\psi_{21}(0) + \tfrac{1}{\alpha(m,k+1)}(\psi_{21}(0)  - \psi_{12}(0)) = 0 
\textand
\psi_{31}(0) = 0.
\]
Substituting $\psi_{01}(0) = 0$ into \eqref{eq:01} yields,
\[
\psi_{01}(s) = 0
\qquad
\psi_{02}(s) = 0
\qquad
\psi_{10}(s) = 0
\Textand
\psi_{20}(s) = 0.
\]
Substituting $\psi_{11}(0) = 0$ and 
$\psi_{21}(0) + \frac{1}{\alpha(m,k+1)}(\psi_{21}(0)  - \psi_{12}(0)) = 0$ into \eqref{eq:11}
yields
\[
\psi_{11}(s) = 0
\qquad
\psi_{12}(s) = 0
\qquad
\psi_{21}(s) = 0
\Textand
\psi_{22}(s) = 0.
\]
Substituting $\psi_{31}(0) = 0$ into \eqref{eq:31} yields
\[
\psi_{31}(s) = 0
\Textand
\psi_{32}(s) = 0.
\]
Finally, \eqref{eq:30} yield that
\[
\psi_{30}(s) = 0, 
\]
thus $\psi(s) = 0$, proving the injectivity of $\Xi^2_{x,\xi'}$ when restricted to $\bbM^+_{x,\xi'}$.
\end{proof}
%%%%%%%%%%%

%%%%%%%%%%%%%%%%%%%%%%%%%%%%%%%%%%
\begin{proposition}
\label{prop:ellipticity_H}
The systems 
\[
(\deltag\oplus H,\PnnD\oplus\PntD)
\Textand
(\deltag\oplus H, \PttD\oplus\frakT),
\] 
restricted to the symmetric elements of $\Wkm{m}{m}$ are OD elliptic.
\end{proposition}
%%%%%%%%%%%%%%%%%%%%%%%%%%%%%%%%%%

%%%%%%%%%%%%%
\begin{proof}
Since all the operators, including the boundary operators are differential operators,  we need to verify that the criteria in \thmref{thm:adatped_OD_ellipticity_estiamte} are satisfied:  we need to show that for every $x\in M$ and $\xi\in T_x^*M\setminus\{0\}$, 
\[
\sigma_{\deltag}(x,\xi) \oplus \sigma_{H}(x,\xi) : \Gkm{m}{m} \to \Gkm{m-1}{m} \oplus  \Gkm{m+1}{m+1}
\]
is injective, and so are the maps
\[
\begin{aligned}
\Xi^1_{x,\xi'} &= \sigma_{\PnnD}(x,\xi'+ \imath\,\partial_s dr)\oplus \sigma_{\PntD}(x,\xi'+ \imath\,\partial_s dr)|_{s=0} \\
\Xi^2_{x,\xi'} &= \sigma_{\PttD}(x,\xi'+ \imath\,\partial_s dr)\oplus \sigma_{\frakT}(x,\xi'+ \imath\,\partial_s dr)|_{s=0}
\end{aligned}
\]
for $x\in\dM$ and $\xi'\in T_x^*\dM\setminus\{0\}$, when restricted to the space $\bbM_{x,\xi'}^+$ of decaying solutions to the ordinary differential system
\beq
\brk{\sigma_{\deltag}(x,\xi'+ \imath\,\partial_s dr) \oplus \sigma_{H}(x,\xi'+ \imath\,\partial_s dr)}\psi(s) = 0.
\label{eq:ODE_H}
\eeq

Let $x\in M$ and $\xi\in T_x^*M\setminus\{0\}$; once again, we assume  without loss of generality that $|\xi|_\g^2=1$.  
The symbols of $H$ and $\deltag$ are
\[
\sigma_{\deltag}(x,\xi)\psi = \ixi \psi
\Textand
\sigma_H(x,\xi)\psi = \xi\wedge\xi_V\wedge \psi .
\]
Suppose that $\psi\in\Gkm{m}{m}|_x$ is symmetric and satisfies
\[
\brk{\sigma_{\deltag}(x,\xi)\oplus\sigma_H(x,\xi)}\psi = 0. 
\]
Then,
\[
\ixiV \ixi(\xi\wedge\xi_V\wedge \psi) = \psi = 0,
\]
proving the injectivity of  $\sigma_{\deltag}(x,\xi)\oplus\sigma_H(x,\xi)$  when restricted to symmetric $(m,m)$-covectors. 

We proceed with the boundary symbols. Let $x\in\dM$ and $\xi'\in T_x^*\dM\setminus\{0\}$. The ordinary differential system \eqref{eq:ODE_H} takes the explicit form
\beq
\begin{aligned}
& \xi\wedge\xi_V\wedge \psi + \imath dr\wedge\xi_V\wedge \dot{\psi} + \imath \xi\wedge dr^T\wedge \dot{\psi} - dr\wedge dr^T \wedge \ddot{\psi} = 0 \\
& \ixi \psi  + \imath i_{\dr} \dot{\psi} = 0.
\end{aligned}
\label{eq:eqsH}
\eeq
As in the proof of \propref{prop:ellipticity_bianchi}, we decompose $\psi$ orthogonally. Substituting into \eqref{eq:ODE_H} and equaling like terms, we obtain
\beq
\psi_{00} = 0,
\label{EQ:00}
\eeq
\beq
\Cases{
\psi_{02} - \imath \dot\psi_{01}  = 0 \\
\psi_{01} + \imath \dot\psi_{02}  = 0
}
\label{EQ:01}
\eeq
and
\beq
\Cases{\psi_{11} + \imath \dot\psi_{12} = 0\\
\psi_{12} + \imath \dot\psi_{22} = 0  \\
\psi_{22} - 2\imath\dot\psi_{12}   - \ddot\psi_{11} = 0
}
\label{EQ:11}
\eeq
and
\[
\psi_{i3}=0
\Textand
\psi_{3i}=0
\qquad
\text{for $i=0,1,2,3$}.
\]

Let $\psi\in \bbM_{x,\xi'}^+$ satisfy  $\Xi^1_{x,\xi'}\psi=0$. This results in 
\[
\psi_{02}(0)  = 0 
\qquad
\psi_{12}(0)  = 0 
\Textand
\psi_{22}(0)  = 0.
\]
Substituting $\psi_{02}(0)  = 0$  into \eqref{EQ:01} yields
\[
\psi_{01}(s)  = 0 
\Textand
\psi_{02}(s)  = 0.
\]
Substituting $\psi_{12}(0)  = 0$ and $\psi_{22}(0)  = 0$ into \eqref{EQ:11} yields
\[
\psi_{11}(s)  = 0 
\qquad
\psi_{12}(s)  = 0 
\Textand
\psi_{22}(s)  = 0.
\]
This proves the injectivity of $\Xi^1_{x,\xi'}\psi=0$ when restricted to $\bbM_{x,\xi'}^+$.

Let $\psi\in \bbM_{x,\xi'}^+$ satisfy  $\Xi^2_{x,\xi'}\psi=0$. 
The condition $\PttD\psi(0)=0$ yields
\[
\psi_{01}(0) = 0
\Textand
\psi_{11}(0) = 0.
\]
Substituting the first into \eqref{EQ:01} yields.
\[
\psi_{01}(s)  = 0 
\Textand
\psi_{02}(s)  = 0.
\]
We then note that,
\[
-\imath \sigma_{\frakT}(x,\xi + \imath\,\partial_s dr) \psi|_{s=0} = \imath \PttD\dot{\psi}(0) -  \xi\wedge \PntD\psi(0)  -  \xi_V\wedge \PtnD\psi(0) = 0,
\]
yielding the additional initial condition
\[
\imath\dot\psi_{11}(0) - 2\psi_{12}(0) = 0,
\]
which substituted together with $\psi_{11}(0)=0$ into \eqref{EQ:11}  yields
\[
\psi_{11}(s)  = 0 
\qquad
\psi_{12}(s)  = 0 
\Textand
\psi_{22}(s)  = 0.
\]
This proves the injectivity of $\Xi^2_{x,\xi'}\psi=0$ when restricted to $\bbM_{x,\xi'}^+$.
\end{proof}
%%%%%%%%%%%%%

\bibliographystyle{amsalpha}
%\bibliography{MyBibs}

\providecommand{\bysame}{\leavevmode\hbox to3em{\hrulefill}\thinspace}
\providecommand{\MR}{\relax\ifhmode\unskip\space\fi MR }
\providecommand{\MRhref}[2]{%
  \href{http://www.ams.org/mathscinet-getitem?mr=#1}{#2}
}
\providecommand{\href}[2]{#2}

\end{document}